\documentclass[a4paper,11pt]{amsart}

\usepackage[english]{babel}
\usepackage{amsmath,amssymb,amsfonts,amsthm,enumerate}
\usepackage{url}
\usepackage{graphicx,epstopdf,color}
\usepackage[colorlinks=true, allcolors=blue]{hyperref}
\usepackage{amsrefs}
\usepackage{mathtools}

\numberwithin{equation}{section}
\allowdisplaybreaks

\newtheorem{theorem}{Theorem}[section]
\newtheorem{lemma}[theorem]{Lemma}
\newtheorem{definition}[theorem]{Definition}
\newtheorem{remark}[theorem]{Remark}
\newtheorem{proposition}[theorem]{Proposition}
\newtheorem{corollary}[theorem]{Corollary}

\newcommand{\R}{\mathbb{R}}

\newcommand{\N}{\mathbb{N}}

\renewcommand{\epsilon}{\varepsilon}
\newcommand{\eps}{\varepsilon}
\newcommand{\F}{\mathcal{F}}
\newcommand{\E}{\mathcal{E}}
\newcommand{\G}{\mathcal{G}}
\newcommand{\U}{\mathcal{U}}

\renewcommand{\leq}{\leqslant}
\renewcommand{\le}{\leqslant}
\renewcommand{\geq}{\geqslant}
\renewcommand{\ge}{\geqslant}
\newcommand{\sgn}{\operatorname{sgn}}
\newcommand{\per}{\operatorname{Per}}

\allowdisplaybreaks

\title[Asymptotic expansion of a nonlocal phase transition energy]
{Asymptotic expansion 
of a \\ nonlocal phase transition energy}

\author[S.~Dipierro]{\href{https://research-repository.uwa.edu.au/en/persons/serena-dipierro}{Serena Dipierro} }

\author[S.~Patrizi]{\href{http://stepatrizi.altervista.org/}{Stefania Patrizi} }

\author[E.~Valdinoci]{\href{https://research-repository.uwa.edu.au/en/persons/enrico-valdinoci}{Enrico Valdinoci} }

\author[M. Vaughan]{\href{https://maryvaughan.github.io/}{Mary Vaughan}}

\address[S. Dipierro, E.~Valdinoci, M.~Vaughan]{The University of Western Australia, 
Department of Mathematics and Statistics,
35 Stirling HWY,
Crawley WA 6009, Australia}
\email{serena.dipierro@uwa.edu.au, 
enrico.valdinoci@uwa.edu.au, 
mary.vaughan@uwa.edu.au}

\address[S.~Patrizi]{The University of Texas at Austin,
Department of Mathematics, 2515 Speedway, Austin, TX 78751, USA}
\email{spatrizi@math.utexas.edu}



\begin{document}
%%%%%%%%%%%%%%%%%%%%%%%%%%

\begin{abstract}
We study the asymptotic behavior of the fractional Allen--Cahn energy functional in bounded domains with prescribed Dirichlet boundary conditions.
 
When the fractional power~$s \in \left(0,\frac12\right)$, we establish the first-order asymptotic development up to the boundary in the sense of $\Gamma$-convergence.
In particular, we prove that the first-order term is the nonlocal minimal surface functional.  
Also, we show that, in general, the second-order term
is not properly defined and intermediate orders may have to be taken into account. 

For~$s \in \left[\frac12,1\right)$, we focus on the one-dimensional case and we prove that the first-order term is the classical perimeter functional plus a penalization on the boundary.
Towards this end, we establish
existence of minimizers to a corresponding fractional energy in a half-line,
which provides itself a new feature with respect to the existing literature.
 \end{abstract}

\maketitle

%\tableofcontents

%%%%%%%%%%%%%%%%%%%%%%%%%%
\section{Introduction}
%%%%%%%%%%%%%%%%%%%%%%%%%%

\subsection{Higher-order expansions and boundary effects}

In this paper, we initiate the study of the asymptotic behavior as~$\eps \searrow 0$ of the nonlocal phase coexistence model
\begin{equation}\label{FUN} 
\mathcal{E}_\eps(u) := a_\eps \iint_{(\R^n \times \R^n) \setminus (\Omega^c \times \Omega^c)} \frac{|u(x) - u(y)|^2}{|x-y|^{n+2s}} \, dx\,dy
	+ b_\eps \int_{\Omega} W(u(x)) \, dx,
\end{equation}
where~$n\geq 1$, $\Omega \subset \R^n$
is a bounded domain,
$W$ is a double-well potential (see~\eqref{eq:W}), $\eps>0$ is a small phase parameter, and~$a_\eps$, $b_\eps>0$ take into account the scaling properties of the energy functional in terms of the fractional parameter~$s \in (0,1)$ (see~\eqref{eq:initial-scaling}).  
In particular, the effect of the potential energy, which favors the pure phases of the system, is balanced by the kinetic energy that takes into account the long-range particle interactions and is ferromagnetic.

The energy~\eqref{FUN} was considered in~\cite{SV-gamma, SV-dens}. In the former, the authors prove that, under appropriate scaling, the energy $\Gamma$-converges to the classical perimeter functional when~$s \in \left[\frac12,1\right)$
and to the nonlocal perimeter functional when~$s \in \left(0,\frac12\right)$. 

In this work, we initiate the analysis of the higher-order asymptotic behavior as~$\eps \searrow 0$.
This is important in the study of slow motion of interfaces in the corresponding evolutionary problem and showcases improved convergence rates for the minimum values. 
Different than the interior result in~\cite{SV-gamma}, we also take into account the nonlocal Dirichlet boundary conditions by restricting the class of minimizers to those satisfying~$u = g$ a.e.~on~$\Omega^c$ for a given function~$g$ (see~\eqref{eq:FeEe}). In this sense, the novelties provided by our work cover two
distinct but intertwined aspects: on the one hand, we analyze the effect of the
boundary data on the limit energy functional; on the other hand, we provide
a higher-order expansion of the limit functional with respect to the perturbation
parameter~$\eps$.
\medskip

Let us now precisely describe our setting mathematically and state our main results. 

We let~$\Omega \subset \R^n$ be a bounded domain of class~$C^2$. 
Roughly speaking, this domain is the ``container''
in which the phase transition takes place. Notice in~\eqref{FUN} that the set 
\[
Q_\Omega:= (\R^n \times \R^n) \setminus (\Omega^c \times \Omega^c)
	= (\Omega\times\Omega) \cup (\Omega\times\Omega^c)\cup (\Omega^c\times\Omega)
\]
collects all the couples~$(x,y)\in \R^n\times\R^n$ for which at least one entry lies in~$\Omega$: that is, $\mathcal{E}_\eps$ takes into account all the possible long-range interactions which involve particles inside the container. 
As is customary,
$\Omega^c$ denotes the complement of~$\Omega$ in~$\R^n$.

To write the total energy of the system
in a convenient notation, given a measurable
function~$u:\R^n\to\R$ and a measurable
set~$A\subset\R^{2n}$,
we write
$$ u(A) := \iint_{A} \frac{|u(x)-u(y)|^2}{|x-y|^{n+2s}}\, dy \,dx.$$
In particular, if~$A=A_1\times A_2$ with~$A_1$ and~$A_2$
measurable subsets of~$\R^n$, we write
$$ u(A_1,A_2) := u(A_1\times A_2)=
\int_{A_1}\int_{ A_2} \frac{|u(x)-u(y)|^2}{|x-y|^{n+2s}}\,dy \,dx.$$
%%% We denote by~$H^s(A)$ the set of functions~$u$ such that~$[u]_{H^s(A)}=u(Q_A) <+\infty$. 

We assume that~$W\in C^2(\R)$ is a double-well potential satisfying
\begin{equation}\label{eq:W}
\begin{cases}
W(-1)=W(1)=0, \\
W'(-1)=W'(1)=0, \\
W(r)>0 \quad \hbox{ for any } r\in(-1,1),\\
 W''(-1),~W''(1)>0.
\end{cases}
\end{equation}

Following~\cite{SV-gamma, SV-dens}, the scaling in the  fractional Allen--Cahn energy~\eqref{FUN} is given by
\begin{equation}\label{eq:initial-scaling}
 a_\eps:=
\begin{cases}
\eps & {\mbox{ if }}s\in\left(0,\frac12\right),\\
\displaystyle\frac{\eps}{ |\ln\eps|}& {\mbox{ if }}s=\frac12,\\
\eps^{2s} & {\mbox{ if }}s\in\left(\frac12,1\right)
\end{cases}
\qquad \hbox{and} \qquad
 b_\eps:=\begin{cases}
\eps^{1-2s} & {\mbox{ if }}s\in\left(0,\frac12\right),\\
\displaystyle\frac{1}{ |\ln\eps|}& {\mbox{ if }}s=\frac12 ,\\
1 & {\mbox{ if }}s\in\left(\frac12,1\right).
\end{cases} 
\end{equation}
As~$s\nearrow1$, the functional in~\eqref{FUN} recovers the classical
Allen--Cahn phase transition energy.

Let~$g:\R^n \to [-1,1]$ be a given Lipschitz continuous function. 
We define
\begin{equation}\label{eq:me-defn}
m_\eps := \min\left\{ \E_\eps(u)~\hbox{s.t.}\,\,~\substack{\displaystyle u:\R^n \to \R~\hbox{is measurable},\\
 \displaystyle u= g~\hbox{a.e.~in}~\Omega^c~\hbox{and}~
 u-g\in H^s(\R^n)} \right\}.
\end{equation}

\medskip

We establish a first-order asymptotic expansion of~$m_\eps$ for any fixed $s \in (0,1)$. 
%For~$s \in \left(0,\frac{1}{2}\right)$, we establish the full asymptotic expansion of~$m_\eps$,
%namely, we provide an expansion of the nonlocal phase coexistence energy
%with respect to any power of the perturbation parameter~$\eps$,
%in any spatial dimension.

For $s \in (0,\frac12)$, we prove a first-order expansion in any spatial dimension. 
Morally speaking, we show that~$m_\eps/\eps$ approximates the fractional
perimeter of a nonlocal minimal surface in~$\Omega$ with prescribed Dirichlet boundary data as $\eps \searrow 0$.
See the seminal paper~\cite{CRS} for more on nonlocal minimal surfaces. 
The precise result goes as follows.

\begin{theorem}\label{THM:1}
Let~$n \geq 1$, $s \in \left(0, \frac12\right)$ and~$\Omega \subset \R^n$ be a bounded domain of class~$C^2$. 
It holds that
\begin{equation}\label{explicitiden54896}
m_\varepsilon = \varepsilon m_1 + o(\eps)
\end{equation}
where
\[
m_1 := \inf_{\substack{u~\text{s.t.} \\ u |_{\Omega} = \chi_E - \chi_{E^c}}} \left[ u(\Omega,\Omega)
	+2\displaystyle\iint_{\Omega\times\Omega^c}\frac{|u(x)-g(y)|^2}{|x-y|^{n+2s}}\,dx\,dy\right].
\]
Here, the infimum is taken over measurable sets~$E \subset \R^n$. 
\end{theorem}

Here above and in what follows, $\chi_A$ denotes the characteristic function associated to a measurable set~$A \subset \R^n$. 

Notably, for $s \in (0,\frac12)$ and in dimension $n=1$, we can find a specific scenario in which $m_\eps$ does \emph{not} have a meaningful second-order expansion and, in fact, intermediate fractional powers of~$\eps$ could be needed for a meaningful expansion (see Theorem~\ref{lem:counterexample} and Corollary~\ref{9823iek.09k.n} below). 

%We observe that the explicit identity in~\eqref{explicitiden54896} and the fact
%that~$m_1$ is independent of~$\eps$
%highlight the interesting feature that the expansion of the energy functional
%when~$s\in\left(0,\frac12\right)$ contains only the first order term in~$\eps$
%and all the others are equal to~$0$ (a more detailed analysis of the energy expansion
%will be given in the forthcoming Section~\ref{sec:subasymptdevep0659}).
\medskip

The case~$s \in \left[\frac12,1\right)$ is fundamentally different as the long-range interactions are lost in the limit as~$\eps \searrow 0$. 
In this paper, we provide the first-order asymptotic expansion of~$m_\eps$ in dimension~$n=1$.  
Since the proof is already quite involved, we leave the general case~$n \ge2$ and the higher-order asymptotic expansion of~$m_\eps$ as future work. 

We use the standard notation~$P(E,\Omega)$ to denote the (classical) perimeter of a set~$E$ in an open set~$\Omega \subset \R$, see~\cite{Giusti}. 

For a function~$u:\R\to \R$ such that~$u \big|_{\Omega} = \chi_E - \chi_{E^c}$ for some set~$E \subset \R$ with~$P(E,\Omega) <+ \infty$,
the trace of~$u$ along~$\partial\Omega$, given by 
\begin{equation}\label{eq:trace}
Tu(x):= \lim_{\Omega\ni y \to x} u(x) \quad {\mbox{ for all }} x \in \partial  \Omega,
\end{equation}
is well-defined.
As an abuse of notation in this setting, 
we often write~$u(x)$ in place of~$Tu(x)$, even
when~$x\in\partial\Omega$.

When~$s\in\left[\frac12,1\right)$, we also take an extra assumption
to control the behavior of the functions considered near the boundary.
This is a technical assumption that we aim at relaxing in a forthcoming work,
but, roughly speaking, it is used in our setting to prevent additional oscillations
of minimizing heteroclinics.

This technical hypothesis goes as follows.
Suppose that~$\Omega = (\bar{x}_1, \bar{x}_2) \subset \R$. 
Let~$\kappa \in\left[0, \frac{|\Omega|}2\right)$. 
Consider the class~$Y_\kappa$ of measurable functions~$u:\R\to \R$ such that, for~$i=1,2$,
\begin{equation}\label{eq:Ykappa}
\begin{array}{rl}
\hbox{if}~g(\bar{x}_i)>0, & \hbox{then}~u(x) \geq g(\bar{x}_i)~\hbox{for all}~x\in\Omega \cap B_{\kappa}(\bar{x}_i);\\
\hbox{if}~g(\bar{x}_i)<0,&\hbox{then}~u(x) \leq g(\bar{x}_i)~\hbox{for all}~x\in\Omega \cap B_{\kappa}(\bar{x}_i);\\
\hbox{if}~g(\bar{x}_i)=0,&\hbox{then either}~u(x) \geq g(\bar{x}_i)~\hbox{or}~u(x) \leq g(\bar{x}_i)~\hbox{for all}~x\in \Omega \cap B_{\kappa}(\bar{x}_i).
%
%\hbox{either}~u(x) \geq g(\bar{x}_i)\geq 0 & \hbox{for all}~x\in\Omega \cap B_{\kappa}(\bar{x}_i)\\
%\hbox{or}~u(x) \leq g(\bar{x}_i)\leq 0 &\hbox{for all}~x\in\Omega \cap B_{\kappa}(\bar{x}_i).
\end{array}
\end{equation}
We point out that
when~$\kappa = 0$, the set~$Y_0$ consists of all the
measurable functions~$u:\R\to \R$.
  
Consider
\begin{equation}\label{eq:me-defn-kappa}
m_\eps^\kappa := \min\Big\{ \E_\eps(u)~\hbox{s.t.}\;\;
u \in Y_\kappa,\;
u= g~\hbox{a.e.~in}~\Omega^c,\; u-g\in H^s(\R)  \Big\}.
\end{equation}
In this setting, the first-order of the asymptotic expansion
of the nonlocal phase coexistence functional goes as follows:

\begin{theorem}\label{THM:1b}
Let~$s \in \left[ \frac12,1\right)$,
$\Omega \subset \R$ be a bounded interval, 
and~$\kappa \in\left(0, \frac{|\Omega|}2\right)$.
Suppose that~$|g|<1$ on~$\partial \Omega$. 

It holds that
\[
{m_\varepsilon^\kappa} = \eps {m_1^\kappa}+ o(\eps)
\]
where
\[
{m_1^\kappa} := \inf_{\substack{{u \in Y_\kappa}~\text{s.t.} \\ u |_{\Omega} = \chi_E - \chi_{E^c}}} \left[ 
c_\star \per\,(E,\Omega)+\displaystyle\int_{\partial\Omega}
\Psi(u(x),g(x))\,d{\mathcal{H}}^{0}(x)\right]
\]
Here, the infimum is taken over measurable sets~$E \subset \R$ such that~$\per\,(E,\Omega) < +\infty$ and
the constant~$c_\star>0$ depends only on~$s$ and~$W$.

Also, $\Psi:\{\pm1\} \times (-1,1) \to (0,+\infty)$
is a suitable function, whose explicit definition will be given in~\eqref{eq:Psi}
relying on the energy of a suitable minimal layer solution.
\end{theorem}

Roughly speaking, for any~$\gamma \in (-1,1)$ and~$s \in \left(\frac12,1\right)$, the penalization functional~$\Psi$
mentioned in Theorem~\ref{THM:1b} is given by
\begin{equation}\label{eq:Psi-intro}
\Psi(-1, \gamma) = \min_{w \in X_\gamma} \left[w(Q_{\R^-})
 + \int_{\R^-} W(w(x)) \, dx\right],
\end{equation}
where~$X_\gamma$ is the class of functions~$w \in H^s(\R)$ such that~$-1 \leq w \leq \gamma$, $w(-\infty) = -1$, and~$w = \gamma$ in~$\R^+$, 
and similarly for~$\Psi(+1,\gamma)$.
That is, $\Psi(u(x),g(x))$ for~$x \in \partial \Omega$ is the minimum energy connecting the interior value~$u(x) \in \{\pm1\}$ (in the trace sense) to the boundary value~$\gamma = g(x) \in (-1,1)$. 

Thus, in order to construct the penalization function~$\Psi$,
we need to construct the heteroclinic function connecting~$\pm1$ to~$\gamma\in(-1,1)$.
This construction is also somewhat delicate and
does not follow trivially from the
constructions of the heteroclinic function connecting~$-1$ to~$+1$
put forth in~\cite{CABSI, PAL}. 
Indeed, in our setting, cutting methods and monotonicity considerations
are not available. Moreover, since~$\gamma$ is not an equilibrium
for the potential~$W$, a bespoke analysis is needed.

The case~$s=\frac12$ is even more delicate, due to the presence
of infinite energy contributions. To overcome this difficulty, the minimal layer solution
when~$s=\frac12$ will be obtained by a limit procedure from the cases~$s\in\left(\frac12,1\right)$.

We stress that the construction of this minimal heteroclinic connection is an important
ingredient for the recovery sequence in the $\Gamma$-convergence expansion
that we propose and introduces a novelty with respect to the available methods in the
nonlocal setting. Indeed, while the interior estimates of the $\Gamma$-limits rely
on the existing literature (see~\cite{SV-gamma}), the control of the boundary terms requires
an ad-hoc analysis and an appropriate interpolation that involves this new heteroclinic
connection.
\medskip

We now describe in detail the mathematical setting related to this layer solution.
Towards this end, for any~$A\subseteq\R$, we define
$$
\G_s(u,A) := u(Q_{A}) +   \int_{A} W(u(x)) \, dx$$
and
\begin{equation*}
\G_s(u) := 
\begin{cases}
 \G_s(u,\R^-) & \hbox{if}~s \in \left(\frac{1}{2},1\right), \\
\displaystyle \liminf_{R \to +\infty} 
 	\left(\frac{ \G_s(u,B_R^-)}{\ln R}\right) & \hbox{if}~s = \frac{1}{2},
 \end{cases}
\end{equation*}
where~$B_R^-:=(-R,0)$.

Given~$\gamma \in (-1,1)$, we define~$X_\gamma$ \label{inventatiunatlabel}
to be the closure, with respect to the $H^s$-seminorm, of the set
\begin{equation*}
\left\{ u \in C^{\infty}(\R)~\hbox{s.t.} \quad \substack{\displaystyle u \leq \gamma, \\
\displaystyle u(x) = \gamma~\hbox{for any}~x>0,~\hbox{and} \\
\displaystyle \hbox{there exists~$x_o <0$ such that}~u(x) = -1~\hbox{for all}~x < x_o}\right\}.
\end{equation*} 
Loosely speaking, $X_\gamma$ is the family of functions in~$H^s(\R)$ that connect~$-1$ at~$-\infty$ to~$\gamma$ in~$\R^+$.

With this notation, we can establish
the existence of minimizers to the energy~$\G_s$
in the class~$X_\gamma$:

\begin{theorem}\label{thm:1Dmin}
Given~$s \in \left[\frac12,1\right)$ and~$\gamma \in (-1,1)$, there exists a unique global minimizer~$w_0 \in X_\gamma$ of the energy~$\mathcal{G}_s$.

Moreover, $w_0$ is strictly increasing in~$\R^-$, 
$w_0(x) \in (-1,\gamma)$ for all~$x \in \R^-$, 
$w_0 \in C^s(\R) \cap C^\alpha_{\text{loc}}(\R^-)$ for all~$\alpha \in (0,1)$, and~$w_0$ solves
\begin{equation}\label{eq:PDE-half}
\begin{cases}
(-\Delta)^s w_0 + W'(w_0) = 0 & \hbox{in}~\R^-, \\
w_0 = \gamma & \hbox{in}~\R^+ ,\\
\displaystyle \lim_{x \to -\infty} w_0(x)= -1. &
\end{cases}
\end{equation}

Furthermore, there exist constants~$C$, $R \geq 1$ such that 
\begin{equation}\label{eq:watinfinity}
0 \leq w_0(x) +1 \leq \frac{C}{|x|^{2s}} \qquad \hbox{for all}~x < -R
\end{equation}
and, for all~$\alpha \in(0, 1)$,
\begin{equation}\label{eq:w-holder}
[w_0]_{C^\alpha((-\infty,-R))} \leq \frac{C}{R^\alpha}. 
\end{equation}

If~$s \in \left(\frac12,1\right)$, then we additionally have~$w_0 \in C_{\text{loc}}^{2s}(\R^-)$ and there exist~$C$, $R\geq 1$ such that
\begin{equation}\label{eq:watinfinity-deriv}
|w_0'(x)| \leq \frac{C}{|x|} \;\hbox{ for all }x < -R \qquad \hbox{and} \qquad
[w_0]_{C^{2s-1}((-\infty,-R))} \leq \frac{C}{R^{2s}}. 
\end{equation}
%If $\gamma = -1$, then the minimizer is $w_0 \equiv -1$. \red{If $\gamma = +1$, then \dots} 
\end{theorem}

A similar result holds when~$-1$ is replaced by~$1$.

We faced several difficulties in proving the existence result in Theorem~\ref{thm:1Dmin}.
First, if~$\gamma>0$, then minimizers may instead prefer to connect~$1$ to~$\gamma$ in the sense that~$w_0(-\infty) = 1$. 
For instance, due to the boundary condition~$w_0 \equiv \gamma$ in~$\R^+$, classical sliding methods as used in~\cite{PAL} are ineffectual in proving both monotonicity and uniqueness.
To solve the issue,
we impose the additional constraint~$ w \leq \gamma$ in the class~$X_\gamma$. 
Then, since the energy on the right-hand side of~\eqref{eq:Psi-intro} is finite for~$s \in \left(\frac12,1\right)$, existence is proved using compactness. 
However, for~$s = \frac12$, the energy is infinite and must be defined through an appropriate rescaling. 
In the latter case, we give a novel approach to proving existence by sending~$s \searrow \frac{1}{2}$. 
For this, we first prove uniform asymptotics in~$s$ near~$\frac12$ for the heteroclinic orbits connecting~$-1$ at~$-\infty$ to~$+1$ at~$-\infty$. 
See Sections~\ref{sec:heteroclinic} and~\ref{sec:connections} for full details. 
\medskip

Theorems~\ref{THM:1} and~\ref{THM:1b} are a consequence
of a general approach for asymptotic developments
in the classical setting of~$\Gamma$-convergence which we now describe. 

%%%%%%%%%%%%%%%%%%%%%%%%%%
\subsection{Asymptotic development}\label{sec:subasymptdevep0659}
%%%%%%%%%%%%%%%%%%%%%%%%%%

The notion of asymptotic development in the sense of $\Gamma$-convergence was introduced by Anzellotti and Baldo in~\cite{Baldo1}, see also~\cite{Baldo2}.
We will present the definitions and details surrounding asymptotic development in our setting. 
For reference, let us first recall the definition of $\Gamma$-convergence, see~\cite[Definition~1.1]{Baldo1}. 

\begin{definition}
For a topological space~$X$, let~$\F_\eps:X \to \R \cup \{\pm\infty\}$ be a family of functions with parameter~$\eps>0$. 
We say that~$\F:X \to \R \cup \{\pm\infty\}$ is the $\Gamma$-limit of~$\F_\eps$ if the following hold.

Let~$\eps_j \searrow 0$ as~$j \to  +\infty$. 
\begin{enumerate}
\item For any sequence~$u_j \to \bar{u}$ in~$X$ as~$j \to +\infty$, it holds that
\[
\liminf_{j \to +\infty} \F_{\eps_j}(u_j) \geq \F(\bar{u}). 
\]
\item Given~$\bar{u} \in X$,  there exists a sequence~$u_j \to \bar{u}$ in~$X$ as~$j \to +\infty$ such that
\[
\limsup_{j \to +\infty} \F_{\eps_j}(u_j) \leq \F(\bar{u}). 
\]
\end{enumerate}
In this case, we write~$\F= \Gamma-\lim_{\eps \searrow 0} \F_\eps$ in~$X$. 
\end{definition}

Consider the space~$X$ of all the measurable
functions~$u:\R^n\to\R$ such that the restriction of~$u$ to~$\Omega$
belongs to~$L^1(\Omega)$.
We endow~$X$, as well as its subspaces,
with the metric of~$L^1(\Omega)$:
\begin{equation}\label{CON:DE}
\begin{split}
&{\mbox{a sequence of functions~$u_j\in X$
converges to~$\bar u\in X$}}\\ &{\mbox{if~$\| u_j-\bar u\|_{L^1(\Omega)}\to0$
as~$j\to+\infty$.}}\end{split}\end{equation}
Notice that the values outside~$\Omega$
are neglected in this procedure.

We comprise the Dirichlet datum
inside the functional by defining
\begin{equation}\label{Xgdefinizione} X_g := \big\{
{\mbox{$u\in X$
s.t.~$u=g$ a.e. in~$\Omega^c$ and~$u-g\in H^s(\R^n)$}} \big\}\end{equation}
and 
\begin{equation}\label{eq:FeEe}
 \F_\eps(u):=
	\begin{cases} \E_\eps (u) & {\mbox{ if }} u\in X_g,\\
	+\infty & {\mbox{ if }} u\in X\setminus X_g.
	\end{cases}
\end{equation}
Notice that~$X_g$ is the
functional space used in the 
construction of~$m_\eps$ in~\eqref{eq:me-defn}, so that
$$ m_\eps = \min_{X_g} \F_\eps.$$
We also set~$\F^{(0)}_\eps := \F_\eps$ and~${\mathcal{U}}_{-1}:=X$.

Following~\cite{Baldo1, Baldo2}
(see in particular~\cite[pages 109--110]{Baldo1}), the
\emph{asymptotic development of order~$k \in \N$}, written as
$$ \F_\eps =_{\Gamma} \F^{(0)} +\eps \F^{(1)} +\dots+\eps^k \F^{(k)}+o(\eps^{k}),$$
holds \emph{in the sense of $\Gamma$-convergence} if
\begin{enumerate}
\item $\F^{(0)} =
\Gamma-\displaystyle\lim_{\eps\searrow0} \F_\eps$ in~${\mathcal{U}}_{-1}$;
\item for any~$j\in \{0,\dots, k-1\}$, we have that
$\F^{(j+1)} =\Gamma-\displaystyle\lim_{\eps\searrow0} \F_\eps^{(j+1)}$ 
in~${\mathcal{U}}_j$, where
\begin{eqnarray*}
&& {\mathcal{U}}_j := \{ u\in 
 {\mathcal{U}}_{j-1} {\mbox{ s.t. }} \F^{(j)}(u)=m_j\}\\
{\mbox{with }}&& m_j :=\inf_{{\mathcal{U}}_{j-1}}\F^{(j)}\\
{\mbox{and }}&& \F^{(j+1)}_\eps:=\frac{\F^{(j)}_\eps-m_j}{\eps}
.\end{eqnarray*}
\end{enumerate}
One can show (see~\cite[pages~106 and 110]{Baldo1}) that 
\[
\{\text{limits of minimizers of}~\mathcal{E}_\eps\} \subset \U_k \subset \dots \subset \U_0 \subset \U_{-1}
\]
and  
\begin{equation}\label{eq:me-expansion}
m_\eps = m_0 + \eps m_1 + \dots + \eps^k m_k+o(\eps^{k}).
\end{equation}
Therefore, the asymptotic development provides a refined selection criteria for minimizers of~$\mathcal{E}_\eps$. 

In this setting, we prove the following asymptotic behavior of~$\mathcal{E}_\eps$ in~\eqref{FUN} in the sense of~$\Gamma$-convergence. 

\begin{theorem}\label{THM:2}
Let~$n \geq 1$, $s \in \left(0,\frac12\right)$ and~$\Omega \subset \R^n$ be a bounded domain of class~$C^2$. 
For all~$k \geq 2$, it holds in the sense of $\Gamma$-convergence that
\[
\F_\varepsilon 
	=_{\Gamma} \varepsilon \F^{(1)} + o(\eps)
\]
where
\begin{equation*}
 \F^{(1)}(u) 
 	= \begin{cases}
u(\Omega,\Omega) +2\displaystyle\iint_{\Omega\times\Omega^c}
\frac{|u(x)-g(y)|^2}{|x-y|^{n+2s}}\,dx \, dy & 
\begin{matrix}
{\mbox{ if~$X\ni u=\chi_E-\chi_{E^c}$
a.e. in~$\Omega$,}} \\ {\mbox{ for some~$E \subset \R^n$,}}\end{matrix} \\
+\infty & {\mbox{ otherwise}}.
\end{cases}
%\label{eq:F1-case1}
\end{equation*} 
%and
%\begin{equation*}
%\F^{(2)}(u)
%	= \begin{cases}
%	0 & \hbox{if}~u~\hbox{is a minimizer of}~\F^{(1)},\\
%	+\infty & \hbox{otherwise}.
%	\end{cases}
%	%\label{eq:F2-case1}
%\end{equation*}
%
%More precisely,
%\begin{equation}\label{fu43ty43gfekwgfwetyi34576889}
%\F_\varepsilon 
%	=_{\Gamma} \varepsilon \F^{(1)} + \frac{\varepsilon^2}{1-\varepsilon}\F^{(2)}.
%\end{equation}
\end{theorem}

For~$s \in\left[\frac12,1\right)$, we recall~\eqref{eq:Ykappa} and define
%\begin{equation}\label{Xgdefinizione-kappa} 
%X_g^\kappa := \big\{
%{\mbox{$u\in X \cap Y_\kappa$
%s.t.~$u=g$ a.e.~in~$\Omega^c$ and~$u-g\in H^s(\R^n)$}} \big\}\end{equation}
%and 
\begin{equation}\label{eq:FeEe-kappa}
 \F_\eps(u):=
	\begin{cases} \E_\eps (u) & {\mbox{ if }} u\in X_g \cap Y_\kappa,\\
	+\infty & {\mbox{ if }} u\in X\setminus (X_g \cap Y_\kappa).
	\end{cases}
\end{equation}
Here, $X_g \cap Y_\kappa$ is used in the 
construction of~$m_\eps^\kappa$ in~\eqref{eq:me-defn-kappa}. In this setting, we prove the following first-order asymptotic of~$\F_\eps$. 

\begin{theorem}\label{THM:2b} 
Let~$s \in \left[ \frac12,1\right)$,~$\Omega \subset \R$ be a bounded interval,
and~$\kappa \in\left(0 ,\frac{|\Omega|}2\right)$. Suppose that~$|g|<1$ on~$\partial \Omega$.

It holds in the sense of $\Gamma$-convergence that
\[
\F_\varepsilon 
	=_{\Gamma} \F^{(0)} + \varepsilon \F^{(1)} + o(\eps)
\]
where
\begin{equation*}
{\mathcal{F}}^{(0)}(u):= \chi_{\left(\frac{1}{2},1\right)}(s) \int_{\Omega} W(u(x)) dx \qquad \hbox{for all}~u \in X
%\label{eq:F0-case2}
\end{equation*} and
\begin{equation*}
{\mathcal{F}}^{(1)}(u):= 
\begin{cases}
c_\star \per\,(E,\Omega)+\displaystyle\int_{\partial\Omega}
\Psi(u(x),g(x))\,d{\mathcal{H}}^{0}(x)
&\begin{matrix}
{\mbox{ if~${X \cap Y_\kappa}\ni u=\chi_E-\chi_{E^c}$
a.e. in~$\Omega$,}} \\ {\mbox{ for some~$E \subset \R$~\hbox{s.t.}~$\per\,(E,\Omega) < +\infty$,}}\end{matrix} \\
+\infty & {\mbox{ otherwise.}}
\end{cases}
%\label{eq:F1-case2}
\end{equation*}
\end{theorem} 

\begin{remark}\label{rem:zero}
{\rm{The limit $ \mathcal{F}^{(0)} = \Gamma - \lim_{\eps \searrow 0} \mathcal{F}_\eps$ in~$X$ in Theorem~\ref{THM:2b} holds more generally
for~$\Omega\subset\R^n$ in all dimensions~$n \geq 1$, for all~$s \in(0,1)$, for~$\kappa =0$, and allowing~$|g| = 1$ on~$\partial \Omega$. See Lemma~\ref{L:0} for the precise statement. }}
\end{remark}

\begin{remark}\label{rem:kappa-zero}
{\rm{We will see in the proof that the $\limsup$-inequality in the first-order $\Gamma$-convergence in Theorem~\ref{THM:2b} holds for~$\kappa = 0$, see Proposition~\ref{prop:limsup}.}}
\end{remark}

In this framework and recalling~\eqref{eq:me-expansion}, Theorems~\ref{THM:1} and~\ref{THM:1b}
are a consequence of Theorems~\ref{THM:2} and~\ref{THM:2b}, respectively.

For~$s \in \left(0,\frac12\right)$, notice in Theorems~\ref{THM:1} and~\ref{THM:2} that the nonlocal energies~$\mathcal{F}_\eps$ with exterior boundary conditions~$u \equiv g$ in~$\Omega^c$ give rise to a nonlocal energy with exterior boundary conditions. 
In contrast, for~$s \in \left[\frac12,1\right)$, we see in Theorems~\ref{THM:1b} and~\ref{THM:2b} that the limiting energy localizes both in the interior and at the boundary, in the sense that the penalization energy only sees~$g$ on~$\partial \Omega$.

For $s \in (0,\frac12)$ and in a specific one-dimensional setting, we prove by direct calculation that $\mathcal{F}_{\eps}$ does \emph{not} have a meaningful asymptotic development of order $2$ or in fact any non-integer order $\mu+1>2 -2s$. 
%Loosely speaking, if $\bar{u}$ gives rise to a $2s$-fractional minimal surface, one can find smooth transitions $u_j$ such that ${\mathcal{F}_{\eps_j}^{(2)}(u_j) \simeq - \eps_j^{1-2s}}$. 
We present the precise result here. An analogous statement is expected to hold in a more general setting and is left for future work.
\begin{theorem}\label{lem:counterexample}
Let~$n=1$, $s\in\left(0,\frac12\right)$, $\Omega:=(-1,1)$, and~$g:=\bar u:=
\chi_{(0,+\infty)}-\chi_{(-\infty,0)}$.
Then, there exists a sequence~$v_\eps\in X$ such that~$v_\eps \to \bar{u}$ in $X$ as~$\eps\searrow0$, and
\[
\mathcal{F}_{\eps}^{(1)}(v_\eps) - m_1 = -2s\left(\frac{1-2s}{\omega}\right)^{\frac{1-2s}{2s}}\varsigma^{\frac1{2s}}\eps^{1-2s}.
\]
In particular, 
\[
\lim_{\eps \searrow 0} \mathcal{F}_{\eps}^{(2)}(v_\eps) = -\infty
\]
and, for all~$\mu>1-2s$,
\[
\lim_{\eps \searrow 0} \frac{\mathcal{F}_{\eps}^{(1)}(v_\eps) - m_1 }{\eps^\mu} = -\infty.
\]
\end{theorem}

\begin{corollary}\label{9823iek.09k.n} In the setting of Theorem~\ref{lem:counterexample}, we have that
$$ {\mathcal{F}}^{(2)}(u)=\begin{dcases}+\infty &{\mbox{if }}u\in X\setminus{\mathcal{U}}_1,\\
-\infty &{\mbox{if }}u\in {\mathcal{U}}_1.
\end{dcases}$$
\end{corollary}

Since the inaugural works~\cite{Baldo1,Baldo2} there have been several papers devoted to asymptotic development in the local setting. See for instance~\cite{Braides} for additional details.
See also~\cite{MR759767, MR2971613, MR3385247}, where a second-order $\Gamma$-convergence expansion is produced for a classical, one-dimensional, Modica-Mortola energy functional,
and~\cite{DalMaso, Leoni}, for the higher-dimensional case.

To the best of our knowledge, we are the first to consider asymptotic development in the nonlocal setting. Our main obstacle to overcome is understanding the penalization function~$\Psi$ and its role at the boundary for~$s \in \left[\frac12,1\right)$. 

%%%%%%%%%%%%%%%%%%%%%%%%%%
\subsection{Organization of the paper}
%%%%%%%%%%%%%%%%%%%%%%%%%%

The rest of the paper is organized as follows. 
First, in Section~\ref{sec:zero}, we prove $\Gamma$-convergence to $\F^{(0)}$ as described in Remark~\ref{rem:zero}.
Then, we set notation for first-order $\Gamma$-convergence in Section~\ref{sec:setup}. 
The proof of Theorem~\ref{THM:2} and a discussion surrounding Theorem \ref{lem:counterexample} for~$s \in \left(0,\frac12\right)$ is in Section~\ref{sec:complete-small-s}. After that, we will assume for the remainder of the paper that~$s \in \left[\frac12,1\right)$ and set notation in Section~\ref{sec:notation}. 
Background and preliminaries on heteroclinic connections are provided in Section~\ref{sec:heteroclinic}. 
Section~\ref{sec:connections} contains the proof of Theorem~\ref{thm:1Dmin} and the construction of the penalization function~$\Psi$. 
To prove  Theorem~\ref{THM:2b} for~$s \in \left[\frac12,1\right)$, we establish the $\liminf$-inequality in Section~\ref{sec:liminf} (see Proposition~\ref{prop:liminf}) and 
the $\limsup$-inequality in Section~\ref{sec:limsup} (see Proposition~\ref{prop:limsup}). 
Lastly, we collect some auxiliary energy estimates in Appendix~\ref{sec:appendix}. 

%%%%%%%%%%%%%%%%%%%%%%%%%%
\section{Computation of $\F^{(0)}$}\label{sec:zero}
%%%%%%%%%%%%%%%%%%%%%%%%%%

In this section, we will establish the zero-th order term in Theorems~\ref{THM:2} and~\ref{THM:2b} in a general setting. 
Let~$X_g$ be as in~\eqref{Xgdefinizione} and~$\F_\eps$ as in~\eqref{eq:FeEe} for all~$s \in (0,1)$ (i.e.~$\kappa =0$ for~$s \in \left[\frac12,1\right)$).

\begin{lemma}\label{L:0}
Let~$n\geq 1$, $s \in (0,1)$ and~$\Omega \subset \R^n$
be a bounded domain of class~$C^2$. 
It holds that $\F^{(0)} = \displaystyle \Gamma- \lim_{\varepsilon \searrow 0} \F_{\eps}$
where
$$\F^{(0)} (u) = \chi_{(1/2,\,1)}(s)\,\int_\Omega W(u(x))\,dx 
\qquad \hbox{for all}~u \in X.$$
\end{lemma}

\begin{proof} Take a sequence~$\eps_j\searrow0$.

We first show that if~$u_j$ is a sequence in~$X$
with~$u_j\to\bar u$ as~$j\to+\infty$, then
\begin{equation}\label{L:0:1}
\liminf_{j\to+\infty} \F_{\eps_j}(u_j)\ge
\chi_{(1/2,\,1)}(s)\,\int_\Omega W(\bar u(x))\,dx.
\end{equation}
We recall that the notion of convergence in~$X$
is the one in~\eqref{CON:DE}.

Notice that if
$$ \liminf_{j\to+\infty} \F_{\eps_j}(u_j)=+\infty,$$
the claim in~\eqref{L:0:1} is obvious. Hence, we suppose that
$$ \liminf_{j\to+\infty} \F_{\eps_j}(u_j)=L<+\infty .$$
In this case, we take a subsequence~$u_{j_k}$ realizing
the above limit. 

Moreover, we 
recall~\eqref{FUN} and~\eqref{eq:FeEe} and we
observe that
$$ \F_{\eps_{j_k}}(u_{j_k})\ge \E_{\eps_{j_k}}(u_{j_k})\ge
b_{\eps_{j_k}}\,\int_\Omega W(u_{j_k}(x))\,dx.$$
We also point out that, as~$\eps\searrow0$,
we have that~$b_\eps\to0$ 
when~$s\in \left(0,\frac12\right]$
and~$b_\eps\to1$
when~$s\in \left(\frac12,1\right)$, that is
\begin{equation}\label{li:b}
\lim_{\eps\searrow0} b_\eps =\chi_{(1/2,\,1)}(s).\end{equation}
Furthermore, by~\eqref{CON:DE}, we take a further subsequence,
still denoted by~$u_{j_k}$,
that converges to~$\bar u$ a.e.~in~$\Omega$. Hence, by Fatou's Lemma,
\begin{equation*}L=
\lim_{k\to+\infty} \F_{\eps_{j_k}}(u_{j_k})\ge 
\lim_{k\to+\infty} b_{\eps_{j_k}}\,\int_\Omega W(u_{j_k}(x))\,dx\ge
\chi_{(1/2,\,1)}(s)\,\int_\Omega W(\bar u(x))\,dx,
\end{equation*}
which proves~\eqref{L:0:1}. 

Now, to complete the proof of
Lemma~\ref{L:0},
we show that for every~$\bar u\in X$
there exists a sequence~$u_j\in X$ which converges to~$\bar u$
as~$j\to+\infty$ and such that
\begin{equation}\label{L:0:2}
\limsup_{j\to+\infty} \F_{\eps_j}(u_j)\le
\chi_{(1/2,\,1)}(s)\,\int_\Omega W(\bar u(x))\,dx.
\end{equation}
To construct such a recovery sequence,
we perform several surgeries, such as
truncations, mollifications and cutoffs
(roughly speaking one smooths a bit the
function~$\bar u$ and glues it to~$g$ smoothly near
the boundary). The arguments
are, in a sense, of elementary nature, but they require some
delicate quantifications.

For this, we fix~$\ell\in\N$. For any~$M\in\N$, we define
$$ \bar u_M :=\max\big\{ -M,\,\min\{\bar u,\,M\}\big\}.$$
Notice that~$\bar u_M\to \bar u$ a.e.~in~$\Omega$ as~$M\to+\infty$
and~$|\bar u_M|\le |\bar u|\in L^1(\Omega)$. Therefore,
by the Dominated Convergence Theorem, we have that
$$ \lim_{M\to+\infty} \| \bar u-\bar u_M\|_{L^1(\Omega)}=0.$$
In particular, we can find~$M_\ell\in\N$ such that
\begin{equation}\label{L1ets:01}
\| \bar u-\bar u_{M_\ell}\|_{L^1(\Omega)}\le \frac{1}{\ell}.
\end{equation}

Now, for any~$\delta>0$, we define~$\Omega_\delta$ to be
the set of all the points of~$\Omega$ which are at distance larger than~$\delta$
from~$\partial\Omega$. We set~$\bar u_{\ell,\delta}:=
\bar u_{M_\ell}\chi_{\Omega_{4\delta}}$. 

Notice that
\begin{equation}\label{J8:65:00A1}
|\bar u_{\ell,\delta}|\le
|\bar u_{M_\ell}|\le M_\ell,
\end{equation}
and, again by the Dominated Convergence Theorem, we have that
$$ \lim_{\delta\searrow0} \| \bar u_{M_\ell} - \bar u_{\ell,\delta}
\|_{L^1(\Omega)} =
\lim_{\delta\searrow0} \| \bar u_{M_\ell} - 
\bar u_{M_\ell}\chi_{\Omega_{4\delta}}
\|_{L^1(\Omega)}=0.$$
Therefore, we can find~$\delta_\ell>0$ sufficiently small such that
\begin{equation}\label{L1ets:02}
\| \bar u_{M_\ell} - \bar u_{\ell,\delta_\ell}
\|_{L^1(\Omega)}\le \frac{1}{\ell}
\end{equation}
and also that
\begin{equation}\label{MIsO}
\big|\Omega\setminus\Omega_{\delta_\ell}\big|\le \frac{1}{\ell\,
\left(M_\ell+\sup_\Omega|g| \right)}.
\end{equation}

Now we perform a mollification argument. Let~$\phi\in C^\infty_0(B_1,[0,1])$.
For any~$\eta>0$, we define
$$\phi_\eta(x):=\frac{1}{\eta^n}
\phi\left(\frac{x}{\eta}\right)\qquad{\mbox{and}}\qquad \tilde u_{\ell,\eta}:=
\bar u_{\ell,\delta_\ell} * \phi_\eta.$$
By construction,~$\tilde u_{\ell,\eta}\in C^\infty(\Omega)$
and
\begin{equation}\label{SUPL:0}
\lim_{\eta\searrow0} \| \bar u_{\ell,\delta_\ell}-\tilde u_{\ell,\eta}
\|_{L^1(\Omega)} =0.
\end{equation}
In addition, by~\eqref{J8:65:00A1},
\begin{equation}\label{J8:65:00A2}
\|\tilde u_{\ell,\eta}\|_{L^\infty(\R^n)}\le
\|\bar u_{\ell,{\delta_\ell}}\|_{L^\infty(\R^n)}\le M_\ell.
\end{equation}

We also claim that, if~$\eta\in \left(0,{\delta_\ell}\right)$,
\begin{equation}\label{SUPL}
{\mbox{$\tilde u_{\ell,\eta}$ is supported in~$
\Omega_{2\delta_\ell}$.}}
\end{equation}
To check this, suppose that~$\tilde u_{\ell,\eta}(x_o)\ne0$, that is
$$ 0\ne \int_{\R^n} \bar u_{\ell,\delta_\ell}(x_o-y)\,\phi\left(
\frac{y}{\eta}\right)\,dy=
\int_{B_\eta} \bar u_{\ell,\delta_\ell}(x_o-y)\,\phi\left(
\frac{y}{\eta}\right)\,dy.$$
This implies that there exists~$y_o\in B_\eta$ such that
$$ 0\ne \bar u_{\ell,\delta_\ell}(x_o-y_o)
=\bar u_{M_\ell}(x_o-y_o)\,\chi_{\Omega_{4\delta_\ell}}(x_o-y_o).$$
In particular, we have that~$x_o-y_o\in \Omega_{4\delta_\ell}$.
Accordingly, for any~$p\in\partial\Omega$, we have that~$|x_o-y_o-p|>
4\delta_\ell$. So, we find that
$$ |x_o-p|\ge |x_o-y_o-p|-|y_o|>4\delta_\ell-\eta>3\delta_\ell,$$
which says that~$x_o\in\Omega_{3\delta_\ell}$, thus establishing~\eqref{SUPL}.

By~\eqref{SUPL:0} and~\eqref{SUPL}, we can find~$\eta_\ell$ sufficiently
small, such that~$
\tilde u_{\ell,\eta_\ell}\in C^\infty_0(\Omega_{\delta_\ell})$ and
\begin{equation}\label{L1ets:03}
\| \bar u_{\ell,\delta_\ell}-\tilde u_{\ell,\eta_\ell}
\|_{L^1(\Omega)} \le\frac1\ell.
\end{equation}

Now we perform a cutoff argument. Namely, we take~$\tau_\ell\in C^\infty_0
(\Omega_{\delta_\ell/2},\,[0,1])$, with~$\tau_\ell=1$ in~$\Omega_{\delta_\ell}$
and we set
$$ u^*_\ell := \tau_\ell \,\tilde u_{\ell,\eta_\ell}
+(1-\tau_\ell)g.$$
By construction, 
\begin{equation}\label{LAK:824p}
{\mbox{$u^*_\ell$ is a Lipschitz function,}} 
\end{equation}
which coincides
with~$g$ outside~$\Omega$, and such that
$$ u^*_\ell-g = \tau_\ell \,(\tilde u_{\ell,\eta_\ell}-g)\in C^{0,1}_0(\Omega)
\subset H^s(\R^n).$$
In particular,
\begin{equation}\label{XGu}
u^*_\ell\in X_g.\end{equation}
Furthermore, by~\eqref{J8:65:00A2}, we have that, in~$\Omega$,
$$ |u^*_\ell| \le \tau_\ell \,|\tilde u_{\ell,\eta_\ell}|
+(1-\tau_\ell)|g|\le M_\ell+\sup_\Omega|g| .$$
Therefore, by~\eqref{MIsO},
\begin{align*} &\| \tilde u_{\ell,\eta_\ell} -u^*_\ell\|_{L^1(\Omega)}=
\| (1-\tau_\ell)(g-\tilde u_{\ell,\eta_\ell}) \|_{L^1(\Omega)} 
\le \left(M_\ell+\sup_\Omega|g| \right)
\,\| 1-\tau_\ell \|_{L^1(\Omega)}\\ 
&\qquad\qquad \qquad \le
\left(M_\ell+\sup_\Omega|g| \right)\,\big| 
\Omega\setminus\Omega_{\delta_\ell}\big|\le \frac{1}{\ell}.\end{align*}
Combining this with~\eqref{L1ets:01}, \eqref{L1ets:02}
and~\eqref{L1ets:03}, we find that
\begin{equation}\label{fianLOop}
\begin{split}
 &\| \bar u-u^*_\ell\|_{L^1(\Omega)}\\ \le\;&
\| \bar u-\bar u_{M_\ell}\|_{L^1(\Omega)}+
\| \bar u_{M_\ell} - \bar u_{\ell,\delta_\ell}\|_{L^1(\Omega)}+
\| \bar u_{\ell,\delta_\ell}-\tilde u_{\ell,\eta_\ell}
\|_{L^1(\Omega)}+
\| \tilde u_{\ell,\eta_\ell} -u^*_\ell\|_{L^1(\Omega)}
\\ 
\le\;&\frac{4}{\ell}.
\end{split}
\end{equation}

Now we remark that, for any~$s\in(0,1)$,
\begin{equation}\label{a-li0}
\lim_{ \eps\searrow0 } a_\eps=0.\end{equation}
On the other hand, for any fixed~$\ell$,
we have that~$u^*_\ell(Q_\Omega)<+\infty$, thanks to~\eqref{LAK:824p}.
{F}rom these observations, we conclude that there exists~$j_\ell\in\N$
such that, for any~$j\ge j_\ell$,
\begin{equation}\label{J:01001}
u^*_\ell(Q_\Omega) \le \frac{1}{\sqrt{a_{\eps_j}}}.
\end{equation}

Now we define
$$ \psi(\ell):=\ell + \max_{i\in\{1,\dots,\ell\}} j_i.$$
We observe that~$\psi(\ell)\ge j_\ell$ and that~$\psi$
is strictly increasing. So, by a linear interpolation,
we can extend~$\psi$ to a piecewise linear function on~$[1,+\infty)$,
which is strictly increasing
and therefore invertible on its image,
with~$\psi^{-1}$ strictly increasing.

As a consequence, the inequality in~\eqref{J:01001}
holds true for every~$j\ge \psi(\ell)$,
that is, equivalently, for every~$\ell \le \psi^{-1}(j)$.
Hence, in particular, if we set
$$ u_j:=u^*_{\psi^{-1}(j)},$$
we have that
\begin{equation}\label{HAJ451274:09A}
u_j(Q_\Omega) \le \frac{1}{\sqrt{a_{\eps_j}}}.
\end{equation}

Now we observe that
\begin{equation}\label{LI-1M}
\lim_{j\to+\infty} \psi^{-1}(j)=+\infty.
\end{equation}
Indeed, suppose not. Then, since~$\psi^{-1}$ is increasing,
the limit above exists and
$$ \R \ni \lambda=: \lim_{j\to+\infty} \psi^{-1}(j) = \sup_{[1,+\infty)}\psi^{-1}.$$
In particular, there exists~$j_o\in\N$ such that for any~$j\ge j_o$
we have that~$\lambda-\psi^{-1}(j)\le 1$. As a consequence, for any~$i\in\N$,
$$ \psi^{-1}(j_o+i)-\psi^{-1}(j_o)\le \lambda-\psi^{-1}(j_o)\le 1$$
and so
$$ j_o+i =\psi(\psi^{-1}(j_o+i))\le \psi( \psi^{-1}(j_o)+1).$$
Sending~$i\to+\infty$, we obtain a contradiction and so~\eqref{LI-1M}
is proved.

Now, from~\eqref{fianLOop} (used here with~$\ell:=\psi^{-1}(j)$)
and~\eqref{LI-1M}, we infer that
$$ \lim_{j\to+\infty}
\| \bar u- u_j\|_{L^1(\Omega)} \le \lim_{j\to+\infty} \frac{4}{\psi^{-1}(j)}=0.$$
So, up to a subsequence, we have that
\begin{equation}\label{J8LakaP}
{\mbox{$u_j \to \bar u$
a.e. in~$\Omega$, as~$j\to+\infty$.}}\end{equation}
Furthermore,
using~\eqref{XGu} and~\eqref{HAJ451274:09A}, we find that
\begin{eqnarray*}
\F_{\eps_j}(u_j) &=& \E_{\eps_j}(u_j) \\
&=& a_{\eps_j}u_j(Q_\Omega)+b_{\eps_j}\,\int_\Omega W(u_j(x))\,dx
\\ &\le& \sqrt{ a_{\eps_j} } +b_{\eps_j}\,\int_\Omega W(u_j(x))\,dx.
\end{eqnarray*}
Hence, by~\eqref{li:b} and~\eqref{a-li0},
\begin{equation*}
\limsup_{j\to+\infty} \F_{\eps_j}(u_j)\le 
\chi_{(1/2,\,1)}(s)\,\limsup_{j\to+\infty} 
\int_\Omega W(u_j(x))\,dx.\end{equation*}
Therefore, recalling~\eqref{J8LakaP} and the Dominated Convergence
Theorem,
\begin{equation*}
\limsup_{j\to+\infty} \F_{\eps_j}(u_j)\le
\chi_{(1/2,\,1)}(s)\,
\int_\Omega W(\bar u(x))\,dx.\end{equation*}
This proves~\eqref{L:0:2}.

By combining~\eqref{L:0:1} and~\eqref{L:0:2}, we obtain that
$$ \Gamma-\displaystyle\lim_{\eps\searrow0} \F_\eps(u) =
\chi_{(1/2,\,1)}(s)\,\int_\Omega W(u(x))\,dx$$
in~$X$, which
completes the proof of Lemma~\ref{L:0}.
\end{proof}

\begin{remark}\label{rem:zero}
{\rm{The proof of the zero-th order convergence in Theorem \ref{THM:2b} for~$\Omega \subset \R$ and~$ \kappa \in\left(0,\frac{ |\Omega|}2\right)$ follows exactly as the proof of Lemma~\ref{L:0}. }}
\end{remark}

%%%%%%%%%%%%%%%%%%%%%%%%%%
\section{Set up for first-order expansion}\label{sec:setup}
%%%%%%%%%%%%%%%%%%%%%%%%%

To determine the first-order term $\F^{(1)}$ in Theorems~\ref{THM:2} and~\ref{THM:2b}, we first describe the functionals $\mathcal{F}_\eps^{(1)}$. 
By Lemma~\ref{L:0}, we clearly have that~$m_0 = 0$ if~$s \in \left(0,\frac{1}{2}\right]$. For $s \in \left(\frac{1}{2},1\right)$, 
$$ m_0 =\inf_{u\in X}\F^{(0)}(u) = \inf_{u\in X} \int_\Omega W(u(x))\,dx =0$$
since we can always choose~$|u|$ constantly equal to~$1$.
Therefore, the set of minimizers is 
\begin{align*}
{\mathcal{U}}_0 
&=\{ u\in X{\mbox{ s.t. }} \F^{(0)}(u)=0\}\\
&=
\begin{cases}
X & \hbox{if}~s \in\left(0,\frac{1}{2}\right],\\
\{ u\in X{\mbox{ s.t. }} u=\chi_E-\chi_{E^c} {\mbox{ a.e. in }}\Omega,
{\mbox{ for some }}E\subset\R^n \} & \hbox{if}~s \in \left(\frac{1}{2},1\right)
\end{cases}
\end{align*}
and the functionals~$\mathcal{F}_\eps^{(1)}$ are
\begin{equation}\label{defesplt806574058670}\begin{split}
&\F^{(1)}_\eps(u) = \frac{\F^{(0)}_\eps(u)-0}{\eps} 
= 
\begin{cases} \displaystyle\eps^{-1} \E_\eps (u) & {\mbox{ if }} u\in X_g,\\
+\infty & {\mbox{ if }} u\in X\setminus X_g,
\end{cases}\\
&\qquad\qquad=
\begin{cases}
\displaystyle \tilde{a}_\eps \,u(Q_\Omega)+\tilde{b}_\eps\int_\Omega W(u(x))\,dx
& {\mbox{ if }} u\in X_g,\\
+\infty & {\mbox{ if }} u\in X\setminus X_g,
\end{cases}\end{split}
\end{equation}
where
\begin{equation}\label{eq:ab-tilde}
 \tilde{a}_\eps:=
\begin{cases}
1& {\mbox{ if }}s \in \left(0,\frac12\right),\\
|\ln\eps|^{-1}& {\mbox{ if }}s=\frac12,\\
\eps^{2s-1} & {\mbox{ if }}s\in\left(\frac12,1\right)
\end{cases}
\qquad \hbox{and} \qquad
\tilde{ b}_\eps:=\begin{cases}
\eps^{-2s} & {\mbox{ if }}s \in \left(0,\frac12\right),\\
(\eps|\ln\eps|)^{-1}& {\mbox{ if }}s=\frac12, \\
\eps^{-1}& {\mbox{ if }}s\in\left(\frac12,1\right).
\end{cases}
\end{equation}
When~$s \in \left[\frac12,1\right)$, we replace~$X_g$ by~$X_g \cap Y_\kappa$ in~\eqref{defesplt806574058670}.

%%%%%%%%%%%%%%%%%%%%%%%%%%
\section{Asymptotic development for $s \in \left(0,\frac{1}{2}\right)$
and proof of Theorem~\ref{THM:2}}\label{sec:complete-small-s}
%%%%%%%%%%%%%%%%%%%%%%%%%%

This section is devoted to the proofs of Theorem~\ref{THM:2} and Theorem \ref{lem:counterexample}. We assume throughout that~$s \in \left(0,\frac12\right)$ and $n \geq 1$ (unless otherwise stated) are fixed. 

%%%%%%%%%%%%%%%%%%%%%%%%%%
\subsection{Computation of $\F^{(1)}$}
%%%%%%%%%%%%%%%%%%%%%%%%%%

Here, we establish the first-order term~$\F^{(1)}$ in  Theorem~\ref{THM:2}. 

\begin{lemma}\label{L:1-small s}
 Let~$s\in\left(0,\frac12\right)$. It holds that $\displaystyle \F^{(1)} = \Gamma- \lim_{\varepsilon \searrow 0} \F^{(1)}_{\eps}$
 where $\mathcal{F}^{(1)}$ is given by
 \[
 \F^{(1)}(u) 
 	= \begin{cases}
u(\Omega,\Omega) +2\displaystyle\iint_{\Omega\times\Omega^c}
\frac{|u(x)-g(y)|^2}{|x-y|^{n+2s}}\,dx \, dy & 
\begin{matrix}
{\mbox{ if $X\ni u=\chi_E-\chi_{E^c}$
a.e. in $\Omega$,}} \\ {\mbox{ for some~$E \subset \R^n$,}}\end{matrix} \\
+\infty & {\mbox{ otherwise}}
\end{cases}
\]
\end{lemma}

\begin{proof} 
We take a sequence~$\eps_j\searrow0$.

First, we show that if~$u_j$ is a sequence in~$X$
with~$u_j\to\bar u$ in~$X$ as~$j\to+\infty$, then
\begin{equation}\label{L:0:1:ORD1}
\liminf_{j\to+\infty} \F_{\eps_j}^{(1)}(u_j)\ge \F^{(1)}(\bar u).
\end{equation}
To this aim, we may assume that
\begin{equation}\label{LA:98iA:0}
u_j\in X_g,\end{equation}
otherwise~$\F_{\eps_j}^{(1)}(u_j)=+\infty$ (recall~\eqref{defesplt806574058670}) and we are done.

In addition, we may suppose that
\begin{equation}\label{LA:98iA}
{\mbox{$|\bar u|=1$ a.e. in $\Omega$,}}\end{equation}
since, if not,
$$ \displaystyle\int_\Omega W(\bar u(x))\,dx >0,$$
and so, by Fatou's Lemma,
$$ \liminf_{j\to+\infty} \F_{\eps_j}^{(1)}(u_j)\ge
\liminf_{j\to+\infty}
 \frac{1}{\eps_j^{2s}}\,\displaystyle\int_\Omega W(u_j(x))\,dx=
+\infty.$$
Accordingly, by~\eqref{LA:98iA}
we know that~$\bar u \big|_{\Omega}=\chi_E-\chi_{E^c}$ for some~$E\subset\R^n$.
By~\eqref{LA:98iA:0}
and Fatou's Lemma,
\begin{align*}
&\liminf_{j\to+\infty} \F_{\eps_j}^{(1)}(u_j)
\geq \liminf_{j\to+\infty} u_j(Q_\Omega) 
=\liminf_{j\to+\infty} \left[ u_j(\Omega,\Omega)+2
\iint_{\Omega\times\Omega^c}
\frac{|u_j(x)-g(y)|^2}{|x-y|^{n+2s}}\,dx \, dy \right] \\
&\qquad\qquad\ge
\bar u(\Omega,\Omega)+
2\iint_{\Omega\times\Omega^c}
\frac{|\bar u(x)-g(y)|^2}{|x-y|^{n+2s}}\,dx \, dy=\F^{(1)}(\bar u).
\end{align*}
This proves~\eqref{L:0:1:ORD1}.

Now we show that for every~$\bar u\in X$
there exists a sequence~$u_j\in X$ which converges to~$\bar u$
as~$j\to+\infty$ and such that
\begin{equation}\label{L:0:2:ORD1}
\limsup_{j\to+\infty} \F_{\eps_j}^{(1)}(u_j)\le
\F^{(1)}(\bar u).
\end{equation}
For this, we may suppose that~$\bar u=\chi_E-\chi_{E^c}$ for some~$E\subset\R^n$.
Otherwise~$\F^{(1)}(\bar u)=+\infty$, and we are done.
In particular, we have that
$$ \int_\Omega W(\bar u(x))\,dx=0.$$
So, we define
$$ u_j (x):=\begin{cases}
\bar u (x)& {\mbox{ if }}x\in\Omega,\\
g(x) & {\mbox{ if }}x\in\Omega^c.
\end{cases} $$
Then, $u_j\in X_g$ and we have that
\begin{align*}
& \limsup_{j\to+\infty} \F_{\eps_j}^{(1)}(u_j)
=
\limsup_{j\to+\infty} \left[ u_j(Q_\Omega)+
 \frac{1}{\eps_j^{2s}}\,\displaystyle\int_\Omega W(u_j(x))\,dx\right]
\\
&\qquad\qquad=
\limsup_{j\to+\infty}\left[
\bar u(\Omega,\Omega)+
2\iint_{\Omega\times\Omega^c}
\frac{|\bar u(x)-g(y)|^2}{|x-y|^{n+2s}}\,dx \, dy
+ 0\right]
=\F^{(1)}(\bar u).
\end{align*}
This proves~\eqref{L:0:2:ORD1}.

The desired result then follows from~\eqref{L:0:1:ORD1}
and~\eqref{L:0:2:ORD1}.
\end{proof}

%%%%%%%%%%%%%%%%%%%%%%%%%%
\subsection{Proof of Theorem~\ref{THM:2}}
%%%%%%%%%%%%%%%%%%%%%%%%%%

\begin{proof}[Proof of Theorem~\ref{THM:2}]
The result follows from Lemmata~\ref{L:0} and~\ref{L:1-small s}.
\end{proof}

%%%%%%%%%%%%%%%%%%%%%%%%%%
\subsection{Computation for~$\F^{(2)}$}
%%%%%%%%%%%%%%%%%%%%%%%%%%

The rest of this section is devoted to the second-order asymptotic development when $s \in (0,\frac12)$. 
In light of Lemma~\ref{L:1-small s}, we have that
\[
m_1 
	=\inf_{u \in \U_0}\F^{(1)}(u) 
	=\inf_{\substack{u\in X~\text{s.t.} \\  u |_{\Omega} = \chi_E - \chi_{E^c}}} \left[ u(\Omega,\Omega) 
	+2\displaystyle\iint_{\Omega\times\Omega^c}\frac{|u(x)-g(y)|^2}{|x-y|^{n+2s}}\,dx\,dy\right].
\]
The set of minimizers is denoted by
\begin{align*}
\U_1
	&= \{ u\in \U_0~\hbox{s.t.}~\F^{(1)}(u)=m_1\}\\
	&= \left\{ u \in X~\hbox{s.t.}~u |_{\Omega} = \chi_E - \chi_{E^c}~\hbox{and}~u(\Omega,\Omega) +2\displaystyle\iint_{\Omega\times\Omega^c}
\frac{|u(x)-g(y)|^2}{|x-y|^{n+2s}}\,dx\,dy = m_1\right\}
\end{align*}
and the functionals~$\F_\eps^{(2)}$ are given by
\begin{align*}
\F_\varepsilon^{(2)}(u)
	&= \frac{\F_\varepsilon^{(1)}(u) - m_1}{\varepsilon}\\
	&=\begin{cases}
	 \displaystyle\frac{1}{\varepsilon} \big(
	u(\Omega, \Omega) + 2 u(\Omega, \Omega^c)
	-m_1\big) + \frac{1}{\varepsilon^{1+2s}} \int_{\Omega}W(u(x)) \, dx & \hbox{if}~u \in X_g,\\
	+\infty & \hbox{if}~u \in X \setminus X_g.
	 \end{cases}
\end{align*}

Now we notice that $\mathcal{F}_\eps^{(2)}$ $\Gamma$-converges to $+\infty$ in $X \setminus \U_1$. 

\begin{lemma}\label{1832yrgfibvCC.02weijf}
Assume $s \in (0,\frac12)$, and
fix $\bar{u} \in X \setminus \mathcal{U}_1$. 
Let $u_j$ be such that $u_j \to \bar{u}$ in $X$ and 
$\eps_j \searrow 0$ as $j \to +\infty$. 

Then
\[
\liminf_{j \to +\infty} \mathcal{F}_{\eps_j}^{(2)}(u_j) = + \infty. 
\]
\end{lemma}
\begin{proof}
As in~\eqref{LA:98iA:0}, we may assume that~$u_j \in X_g$. 
Suppose first that $\bar{u}\big|_\Omega= \chi_E - \chi_{E^c}$ for some measurable $E \subset \R^n$. By Fatou's Lemma and the definition of $m_1$, 
\begin{align*}
&\liminf_{j \to +\infty} \bigg( u_j(\Omega, \Omega) 
		+ 2 \iint_{\Omega \times \Omega^c} \frac{|u_j(x) - g(y)|^2}{|x-y|^{n+2s}} \,dx \,dy-m_1\bigg)\\
%= \liminf_{j \to +\infty}\left( u_j(\Omega, \Omega) 
%		+ 2 \iint_{\Omega \times \Omega^c} \frac{|u_j(x) - g(y)|^2}{|x-y|^{n+2s}} \,dx \,dy
%		-m_1\right)\\
	&\qquad\qquad\geq \bar{u}(\Omega, \Omega) + 2 \iint_{\Omega \times \Omega^c} \frac{|\bar{u}(x) - g(y)|^2}{|x-y|^{n+2s}} \,dx \,dy
		- m_1\\
	&\qquad\qquad= \F^{(1)}(\bar{u}) - m_1>0.
\end{align*}
%since $\bar{u} \in X \setminus \U_1$ is not a minimizer of $\F^{(1)}$. 
Consequently,
\[
\liminf_{j \to +\infty} \F_{\eps_j}^{(2)}(u_j) 
	\geq  \liminf_{j \to +\infty}  \frac{1}{\varepsilon_j}\bigg( u_j(\Omega, \Omega) 
		+ 2 \iint_{\Omega \times \Omega^c} \frac{|u_j(x) - g(y)|^2}{|x-y|^{n+2s}} \,dx \,dy- m_1\bigg)= +\infty.  
\]

Now, suppose that $|\{x \in \Omega~\hbox{s.t.}~\bar{u}(x) \ne \pm1\}|>0$.  In this case,
\[
\int_{\Omega} W(\bar{u}(x)) \, dx > 0.
\]
By Fatou's Lemma,
\[
\liminf_{j \to \infty} \left(\frac{1}{\eps_j^{2s}} \int_{\Omega} W(\bar{u_j}(x)) \, dx - m_1\right) = +\infty,
\]
and consequently,
\[
\liminf_{j \to \infty} \F_{\eps_j}^{(2)}(u_j)
	\geq \liminf_{j \to \infty}  \frac{1}{\eps_j} \left(\frac{1}{\eps_j^{2s}} \int_{\Omega} W(u_j(x)) \, dx - m_1\right) = + \infty. 
\]
This completes the proof. 
\end{proof}

We now consider $\Gamma$-convergence of $\mathcal{F}_\eps^{(2)}$ in $\U_1$ in dimension $n=1$. In particular, we establish Theorem \ref{lem:counterexample}. The proof is broken up into several small observations. 

%\begin{lemma}\label{lem:counterexample}
%Let $n=1$, $s \in (0,\frac12)$, $\Omega := (-1,1)$, and $g:= \chi_{(0,+\infty)} - \chi_{(-\infty,0)}$. 
%
%There exists $\bar{u} \in \U_1$ and a sequence $u_j \in X_g$ satisfying $u_j \to \bar{u}$ and $\eps_j \searrow 0$ as $j \to +\infty$ and 
%\[
%\liminf_{j \to +\infty} \mathcal{F}_{\eps_j}^{(2)}(u_j) = -\infty. 
%\]
%\end{lemma}

\begin{lemma}\label{PRIM}
Let~$n=1$ and~$s\in\left(0,\frac12\right)$.
Let~$\bar u:=\chi_{(0,+\infty)}-\chi_{(-\infty,0)}$ and~$\delta\in(0,1]$.
Let also~$u_\delta$ be the minimizer of~$u(Q_{(-1,1)})$ among measurable functions~$u:\R\to\R$ with~$u=\bar u$ in~$\R\setminus(-\delta,\delta)$.

Then,
$$ u_\delta(x)=u_1\left(\frac{x}\delta \right)$$
and
$$ u_\delta(Q_{(-1,1)})=\bar u(Q_{(-1,1)})-\varsigma {\delta^{1-2s}},$$
where
$$ \varsigma:=\bar u(Q_{(-1,1)})-u_1(Q_{(-1,1)}).$$
\end{lemma}

\begin{proof} Given~$v:\R\to\R$ with~$v=\bar u$ in~$\R\setminus(-\delta,\delta)$, we define~$w(x):=v(\delta x)$ and
$$\phi(x,y):=\frac{|v(x)-v(y)|^2-|\bar u(x)-\bar u(y)|^2}{|x-y|^{1+2s}}.$$

Since~$\phi$ vanishes identically in~$(\R\setminus(-\delta,\delta))^2\supseteq(\R\setminus(-1,1))^2$ and~$\bar{u}(x) = \bar{u}(\delta x)$ for all~$x \in \R$, we have that 
\begin{eqnarray*}
&&v(Q_{(-1,1)})-\bar u(Q_{(-1,1)})=
\iint_{Q_{(-1,1)}}\phi(x,y)\,dx\,dy=
\iint_{\R\times\R}\phi(x,y)\,dx\,dy\\&&\quad
={\delta^2}\iint_{\R\times\R}\phi(\delta x,\delta y)\,dx\,dy={\delta^2}\iint_{Q_{(-1,1)}}\phi(\delta x,\delta y)\,dx\,dy\\&&\quad={\delta^2}\iint_{Q_{(-1,1)}}
\frac{|w(x)-w(y)|^2-|\bar u(x)-\bar u(y)|^2}{|\delta x-\delta y|^{1+2s}}\,dx\,dy\\&&\quad={\delta^{1-2s}}
\big(w(Q_{(-1,1)})-\bar u(Q_{(-1,1)})\big).
\end{eqnarray*}
{F}rom this, we obtain the desired result (we stress that~$u_1$ means~$u_\delta$ with~$\delta:=1$).
\end{proof}

\begin{lemma}\label{PRIM2}
In the notation of Lemma~\ref{PRIM}, we have that~$\varsigma>0$.
\end{lemma}

\begin{proof} Due to the minimality of~$u_1$, we know that~$\varsigma\ge0$. Also, a minimizer~$u_1$ satisfies the Euler-Lagrange equation~$(-\Delta)^s u_1=0$~in $\Omega =(-1,1)$ and therefore it is necessarily continuous at the origin, ruling out the possibility for~$\bar u$ to attain the minimum.
\end{proof}

\begin{lemma}\label{PRIM3}
In the notation of Lemma~\ref{PRIM}, we have that~$\bar u$
is a minimizer of~$u(Q_{(-1,1)})$ among measurable functions~$u:\R\to\{-1,1\}$ with~$u=1$ in~$[1,+\infty)$
and~$u=-1$ in~$(-\infty,-1]$.
\end{lemma}

\begin{proof} We know that a minimizer can be taken to be monotone (see~\cite[Theorem~2.11]{Alberti}) and odd symmetric
(see~\cite[Lemma~A.1]{MR3596708}). Since the minimizer is constrained to take values in~$\{-1,1\}$, the desired result follows.
\end{proof}

\begin{corollary}\label{PRIM4}
Let~$n=1$, $s\in\left(0,\frac12\right)$, $\Omega:=(-1,1)$, $g:=
\chi_{(0,+\infty)}-\chi_{(-\infty,0)}$, and~$\eps$, $\delta\in(0,1]$.

Let~$u_\delta$ be as in Lemma~\ref{PRIM}.

Then,\begin{equation}\label{PRIMq}
u_\delta(Q_{{(-1,1)}}) - m_1+\frac1{\eps^{2s}}\int_{-1}^1 W(u_\delta(x))\,dx=-\varsigma {\delta^{1-2s}}+\frac{\omega\delta}{\eps^{2s}},\end{equation}
where~$\varsigma>0$ is as in Lemmata~\ref{PRIM} and~\ref{PRIM2}, and
$$ \omega:=\int_{-1}^1 W(u_1(x))\,dx>0.$$
\end{corollary}

\begin{proof} It follows from Lemma~\ref{PRIM3} that~$m_1=\bar u(Q_{(-1,1)})=g(Q_{(-1,1)})$
and consequently, by Lemma~\ref{PRIM},
\begin{eqnarray*}
&&u_\delta(Q_{(-1,1)}) - m_1+\frac1{\eps^{2s}}\int_{-1}^1W(u_\delta(x))\,dx=
-\varsigma {\delta^{1-2s}}+\frac1{\eps^{2s}}\int_{-\delta}^\delta W(u_\delta(x))\,dx\\&&\qquad=-\varsigma {\delta^{1-2s}}+\frac1{\eps^{2s}}\int_{-\delta}^\delta W\left(u_1\left(\frac{x}\delta \right)\right)\,dx=-
\varsigma {\delta^{1-2s}}+\frac\delta{\eps^{2s} }\int_{-1}^1 W(u_1(x))\,dx.
\end{eqnarray*}
Since~$u_1(0)=0$ and~$u_1$ is continuous at the origin (see the proof of Lemma~\ref{PRIM2}), the desired result follows.
\end{proof}

\begin{corollary}\label{sojchn.wosdl} In the notation of Corollary~\ref{PRIM4}, the minimum for~$\delta\in(0,1]$ of the quantity in~\eqref{PRIMq} is attained when~$\delta=\left(\frac{(1-2s)\varsigma}{\omega}\right)^{\frac1{2s}}\eps$ and equals~$-2s\left(\frac{1-2s}{\omega}\right)^{\frac{1-2s}{2s}}\varsigma^{\frac1{2s}}\eps^{1-2s}$. 
\end{corollary}

\begin{proof} We consider the function
$$[0,1)\ni\delta\mapsto f(\delta):=-\varsigma {\delta^{1-2s}}+\frac{\omega\delta}{\eps^{2s}}$$
and we observe that~$f'(\delta)=-(1-2s)\varsigma {\delta^{-2s}}+\frac{\omega}{\eps^{2s}}$
which is positive if and only if~$\delta>\left(\frac{(1-2s)\varsigma}{\omega}\right)^{\frac1{2s}}\eps$
and the desired result follows.
\end{proof}

\begin{proof}[Proof of Theorem \ref{lem:counterexample}]
Use Corollary~\ref{sojchn.wosdl}
with~$v_\eps := u_\delta$ for $\delta: = \left(\frac{(1-2s)\varsigma}{\omega}\right)^{\frac1{2s}}\eps$.
\end{proof}

\begin{proof}[Proof of Corollary~\ref{9823iek.09k.n}]
This is a consequence of
Theorem \ref{lem:counterexample}
and Lemma~\ref{1832yrgfibvCC.02weijf}.
\end{proof}

%%%%%%%%%%%%
\section{Notation for $s \in \left[\frac12,1\right)$}\label{sec:notation}
%%%%%%%%%%%%

Throughout the remainder of the paper, we assume that~$s \in \left[\frac12,1\right)$. 

We use the following notation for measurable sets~$A$, $\Omega \subset \R$:
\begin{equation}\label{eq:G-general-domain}
\begin{aligned}
	\F_{\eps}^{(1)}(u,A) &:=  \tilde{a}_\eps u(Q_A) +  \tilde{b}_\eps \int_A W(u(x)) \, dx,\\
	I_\eps(u,A,\Omega) &:= \tilde{a}_\eps \Big[u(A \cap \Omega,A \cap \Omega) +2u(A \cap \Omega,A \setminus (A \cap \Omega))\Big] 
	+ \tilde{b}_\eps \int_{A \cap \Omega} W(u(x)) \, dx,\\
\G_s(u,A) &:= u(Q_{A}) +   \int_{A} W(u(x)) \, dx.
\end{aligned}
\end{equation}
We recall that~$\tilde{a}_\eps$ and~$\tilde{b}_\eps$ are defined
in~\eqref{eq:ab-tilde}.
We point out that in contrast to~$\F_{\eps}^{(1)}(u,A)$, the energy~$I_\eps(u,A,\Omega)$ only depends on the values of~$u$ in~$A$. 

Also, let~$A$, $B \subset \R$ and assume that~$v$, $w:\R \to \R$ are such
that~$v = w$ in~$(A \cap B)^c$. 
It is a straightforward computation to show that
\begin{equation}\label{eq:energy-difference}
\F_\eps^{(1)}(w, A)  - \F_\eps^{(1)}(v, A) = \F_\eps^{(1)}(w,B)  - \F_\eps^{(1)}(v,B). 
\end{equation}

Lastly, we note how the energies in~\eqref{eq:G-general-domain} scale. First, if~$u_\rho(x) := u(\rho x)$ for~$\rho>0$, then 
it is easy to check that
\begin{equation}\label{eq:rho-scaling}
\F_\eps^{(1)}(u_\rho, A) =
\begin{cases} \displaystyle
\F_{\rho \eps}^{(1)}(u, \rho A) & \hbox{if}~s \in \left(\frac{1}{2},1\right),\\
\\
\displaystyle\frac{|\ln (\rho \eps)|}{|\ln \eps|} \F_{\rho \eps}^{(1)}(u, \rho A) & \hbox{if}~s = \frac{1}{2} \quad \hbox{as long as}~\rho \not= \eps^{-1},
\end{cases}
\end{equation}
and similarly for $I_\eps(u_\rho,A,\Omega)$. 
In the special case~$u_\eps(x) := u(x/\eps)$ (i.e.~$\rho = \eps^{-1}$), we have 
\begin{equation}\label{eq:g-rescale}
\F_\eps^{(1)}(u_\eps, A) =
\begin{cases} \G_s(u, A/\eps) & \hbox{if}~s \in \left(\frac{1}{2},1\right),\\
\displaystyle |\ln \eps|^{-1} \G_{\frac12}(u, A/\eps)& \hbox{if}~s = \frac{1}{2}.
\end{cases}
\end{equation}

%%%%%%%%%%%%%%%%
\section{Heteroclinic connections}\label{sec:heteroclinic}
%%%%%%%%%%%%%%%%%

In this section, we give background and preliminaries on heteroclinic connections that connect~$-1$ at~$-\infty$ to~$+1$ at~$+\infty$. 
As proved in~\cite{SV-gamma}, these are used to construct the recovery sequence in the interior (i.e.~in compact subsets~$ \Omega' \Subset \Omega$) for $\Gamma$-convergence. 
We will also use the heteroclinic connections as barrier when studying connections at the boundary. 

Recalling~\eqref{eq:G-general-domain}, 
%we emphasize dependence on $s$  by setting
%\begin{equation}\label{eq:G-s-domain}
%\tilde{\G}_s(u,A)
%	= u(Q_A)
%	%\iint_{Q_A} \frac{|u(x) - u(y)|^{2}}{|x-y|^{1+2s}} \, dy \, dx
%	+ \int_{A}W(u(x)) \, dx
%\end{equation}
%and 
we set
\begin{equation}\label{eq:G-s-limit}
\tilde{\G}_s(u)
 :=\begin{cases}
	{\G}_s(u, \R) &\hbox{if}~s \in \left(\frac12,1\right),\\
	\displaystyle\lim_{R \to +\infty} \frac{{\G}_s(u, B_R)}{\ln R}& \hbox{if}~s =\frac12.
\end{cases}
\end{equation}
Existence and uniqueness (up to translations) of minimizers of~$\tilde{\G}_s$ over the class of~$H^s$-functions that connect~$-1$ at~$-\infty$ and~$+1$ at~$+\infty$ was established in~\cite{PAL, CABSI}. 
After fixing the value of the minimizer at the origin, the authors also show that
the unique minimizer~$u_0:\R \to (-1,1)$ is in the class~$C^2(\R)$ and satisfies
\begin{equation}\label{eq:1DinR}
\begin{cases}
 (-\Delta)^s u_0(x)+W'(u_0(x))=0 & \;\hbox{for all}~x \in \R,\\
 \displaystyle\lim_{x \to\pm\infty} u_0(x) =\pm1,&\\
 u_0(0)=0 ,&\\
 u'>0. 
\end{cases}
\end{equation}
Moreover, by~\cite[Theorem~2]{PAL}, for any~$ s\in \left[\frac12,1\right)$, there exist constants~$C$, $R \geq 1$ such that 
\begin{equation}\label{eq:PAL-asymp}
|u_0(x) - \operatorname{sgn}(x)| \leq \frac{C}{|x|^{2s}} \qquad \hbox{and} \qquad 0 < u_0'(x) \leq \frac{C}{|x|^{1+2s}}  \qquad \hbox{for all}~|x| \geq R. 
\end{equation}

%%%%%%%%%%%%%%%%%%%%
\subsection{Limiting behavior near $s = \frac12$} 
%%%%%%%%%%%%%%%%%%%%

We will need the following results on the limiting behavior of the minimizers~$u_0$ as~$s \searrow \frac12$. 
%This will be needed in the next subsection. 

For clarity, let~$u_0^s$ denote the minimizer of~$\tilde{\G}_s$ given in~\eqref{eq:1DinR}. 
Since~$|u_0^{s}| \leq 1$, as a consequence of~\cite[Theorem~26]{CS}, there exist~$\alpha \in (0,1)$ and~$C>0$ such that, for all~$s\in\left[\frac12,1\right)$, 
\begin{equation}\label{eq:u0-uniform-holder}
\|u_0^{s}\|_{C^\alpha(\R)} \leq C.
\end{equation}
Fix now~$s_0$ such that
\begin{equation}\label{eq:s0}
\frac12 < s_0 < \frac{\alpha+1}{2}.
\end{equation}
We see that the asymptotic behavior of~$u_0^s$ in~\eqref{eq:PAL-asymp} is uniform in~$s \in \left[\frac12,s_0\right]$. 

\begin{lemma}\label{lem:u0-uniform-asymptotics}
There exist constants~$C$, $R \geq 1$ such that, for all~$s \in \left[\frac12,s_0\right]$,
\[
|u_0^s(x) - \operatorname{sgn}(x)| \leq \frac{C}{|x|^{2s}} \qquad \hbox{for all}~|x|\geq R. 
\]
\end{lemma}

The proof relies on the barrier constructed in~\cite{SV-dens}. 
More precisely, one can check that the following version, in which all constants are uniformly bounded in~$s$ in compact subsets of~$(0,1)$, holds.

\begin{lemma}[Lemma~3.1 in~\cite{SV-dens}]\label{lem:barrier}
Let~$n \geq 1$. Fix~$\delta>0$ and~$0 < s_1 < s_2 < 1$.
 
There exists~$C = C(\delta, n, s_1,s_2) \geq 1$ such that, for any~$R \geq C$ and all~$s \in [s_1,s_2]$, there exists a rotationally symmetric function
\[
\omega_s \in C(\R;[-1+CR^{-2s},1])
\]
with 
\[
\omega_s=1 \quad \hbox{in}~(B_R)^c
\]
such that
\[
-(-\Delta)^s \omega_s(x) = \int_{\R^n} \frac{\omega_s(y) - \omega_s(x)}{|x-y|^{1+2s}} \, dy\le \delta
\big(1+\omega_s(x)\big)
\]
and
\begin{equation}\label{eq:barrier-decay}
\frac{1}{C}(R+1-|x|)^{-2s} \leq 1+ \omega_s(x) \leq C(R+1-|x|)^{-2s}
\end{equation}
for all~$x \in B_R$. 
\end{lemma}

Several times in the paper, we will use for~$b>a>0$ that, for~$s \in \left(\frac12,1\right)$,
\begin{equation}\label{eq:ln-limit}
\frac{ a^{1-2s}-b^{1-2s}}{2s-1} \leq \lim_{s \searrow \frac12} \frac{a^{1-2s} - b^{1-2s}}{2s-1}
	%= \lim_{s \searrow \frac12} \frac{-2b^{1-2s} \ln b+2a^{1-2s}\ln a}{2} 
	= \ln b - \ln a = \ln \frac{b}{a}. 
\end{equation}

\begin{proof}[Proof of Lemma~\ref{lem:u0-uniform-asymptotics}]
We will prove that there exist~$C$, $R \geq 1$ such that, for all~$s \in \left[\frac12,s_0\right)$, 
\begin{equation}\label{eq:u0-decay}
0 \leq u_0^s(x)+1 \leq \frac{C}{|x|^{2s}} \qquad \hbox{for}~x <-R. 
\end{equation}
The uniform decay at~$+\infty$ is proved similarly.

Fix~$\delta >0$ small and~$s \in \left[\frac12,s_0\right]$. 
Let~$r_{s,\delta}>0$ be such that~$u_0^s(-r_{s,\delta}) = -1+\delta$.

We claim that
\begin{equation}\label{nbvcxjyhtgrfew464}
{\mbox{$r_{s,\delta}$ can be bounded from above by some~$r_\delta$, independent of~$s \in \left[\frac12,s_0\right]$.}}\end{equation}
In the following, $C$ denotes a positive constant, independent of~$s \in \left[\frac12,s_0\right]$. 

We may assume that~$r_{s,\delta} \geq 4$, otherwise we are done. 
For ease in notation, we write~$r = r_{s,\delta}$
and consider the sets~$A := [-r,0]$ and~$A' := [-r+1, -1]$. 
Let~$\zeta \in C(\R;[0,1])$ be such that
\[
\zeta = 1~\hbox{in}~A^c, \quad \zeta = 0~\hbox{in}~A', \quad \hbox{and}\quad \zeta \in C^{0,1}(A). 
\]
Note that the Lipschitz constant of~$\zeta$ in~$A$ can be taken independently of~$s$. 
Define the function
\[
u_*(x) := \zeta(x) u_0^s(x) - (1-\zeta(x)).
\]
Since~$u_* = u_0^s$ in~$A^c$ and~$u_0^s$ is a minimizer of~${\G}_s$, it holds that
\begin{equation}\label{eq:local-min-u0u*}
0 \geq {\G}_s(u_0^s,A) - {\G}_s(u_*,A).
\end{equation}

Now, we notice that~$-1+\delta \leq u_0^s \leq 0$ in $A$, and therefore
there exists~$C_1= C_1(\delta,W)>0$ such that
\[
\int_{A} W(u_0^s(x)) \, dx \geq \inf_{t \in [-1+\delta,0]} W(t) |A| = C_1r_{s,\delta}. 
\]
On the other hand, since~$W(u_*(x)) = W(-1) = 0$ in~$A'$, there exists~$C_2= C_2(\delta,W)>0$ such that
\[
\int_A W(u_*(x)) \, dx
 = \int_{A \setminus A'} W(u_*(x)) \, dx
 \leq 2 \sup_{t \in [-1,0]} W(t) = C_2. 
\]
Therefore, from~\eqref{eq:local-min-u0u*}, we have that
\begin{equation}\label{eq:QA-r}
0 \geq u_0^s(Q_A) - u_*(Q_A) + C_1r_{s,\delta} - C_2. 
\end{equation}

We next show that there exists~$C_3 = C_3(\delta)>0$ such that
\begin{equation}\label{eq:QA-final}
u_*(Q_A)-u_0^s(Q_A) \leq C_3(1 + \ln r_{s,\delta}). 
\end{equation}
For this, for any~$x$, $y \in \R$, we use the fact that
\begin{align*}
|u_*(x) - u_*(y)|
	&= |\zeta(x) u_0^s(x)  - \zeta(y)u^s(y)+\zeta(x) - \zeta(y)|\\
&=  |\zeta(x) u_0^s(x) - \zeta(x) u_0^s(y) + \zeta(x) u_0^s(y) - \zeta(y)u^s(y)+\zeta(x) - \zeta(y)|\\
&\leq |\zeta(x)|\,|u_0^s(x) - u_0^s(y)| + |u_0^s(y)+1| \,|\zeta(x) - \zeta(y)|\\
	&\leq  |u_0^s(x) - u_0^s(y)| + 2|\zeta(x) - \zeta(y)|
\end{align*}
to note that
\begin{equation*}
|u_*(x) - u_*(y)|^2 - |u_0^s(x) - u_0^s(y)|^2
	\leq 4|u_0^s(x) - u_0^s(y)|\,
	|\zeta(x) - \zeta(y)| + 4|\zeta(x) - \zeta(y)|^2.
\end{equation*}
{F}rom this and~\eqref{eq:u0-uniform-holder}, and since~$\zeta$ is Lipschitz in~$A$, for any~$A''\subseteq A$,
\begin{align*}
\int_{A \setminus A'} &\int_{A''\cap\{|x-y|<1\}} \frac{|u_*(x) - u_*(y)|^2}{|x-y|^{1+2s}} \, dy \, dx
	- \int_{A \setminus A'} \int_{A''\cap\{|x-y|<1\}} \frac{|u_0^s(x) - u_0^s(y)|^2}{|x-y|^{1+2s}} \, dy \, dx\\
&\leq 4\int_{A \setminus A'} \int_{A''\cap\{|x-y|<1\}} \frac{|u_0^s(x) - u_0^s(y)|\,
|\zeta(x) - \zeta(y)| + |\zeta(x) - \zeta(y)|^2}{|x-y|^{1+2s}} \, dy \, dx\\
&\leq C\int_{A \setminus A'} \int_{\{|x-y|<1\}} |x-y|^{\alpha-2s}\, dy \, dx.
\end{align*}
Recalling~\eqref{eq:s0}, it follows for all~$s \in \left[\frac12,s_0\right]$ that
\begin{equation}\label{9486tfudksegfeuwiEESEG}\begin{split}
\int_{A \setminus A'} &\int_{A''\cap\{|x-y|<1\}} \frac{|u_*(x) - u_*(y)|^2}{|x-y|^{1+2s}} \, dy \, dx
	- \int_{A \setminus A'} \int_{A''\cap\{|x-y|<1\}} \frac{|u_0^s(x) - u_0^s(y)|^2}{|x-y|^{1+2s}} \, dy \, dx\\
	&\leq \frac{C}{\alpha-2s_0+1} |A \setminus A'| \leq C. 
\end{split}\end{equation}

Also,
\begin{equation}\label{eq:AA'-far}\begin{split}&
\int_{A \setminus A'} \int_{\{|x-y|>1\}} \frac{|u_*(x) - u_*(y)|^2}{|x-y|^{1+2s}} \, dy \, dx
	\leq C \int_{A \setminus A'} \int_{\{|x-y|>1\}} \frac{dx\,dy}{|x-y|^{1+2s}}
	\\&\qquad\qquad
	\leq \frac{C}{2s} |A \setminus A'|  \leq C. \end{split}
\end{equation}
As a consequence of~\eqref{9486tfudksegfeuwiEESEG} and~\eqref{eq:AA'-far}, for any~$A''\subseteq A$,
\begin{equation}\label{43t43y34647hersdfbgh}
u_*(A\setminus A', A'')  - u_0^s(A\setminus A', A'')
\le C.\end{equation}

Moreover, since~$u^* \equiv -1$ in~$A'$, we have that
\[
u_*(A', A') -u_0^s(A',A') = 0 - u_0^s(A',A')  \leq 0,
\]
and thus
\begin{eqnarray*}
&&u_*(A,A)-u_0^s(A,A)\\
	&&\qquad\leq 2\big[u_*(A\setminus A', A') -u_0^s(A\setminus A',A')\big]
		+\big[u_*(A\setminus A', A\setminus A') -u_0^s(A\setminus A',A\setminus A')\big]. 
\end{eqnarray*}
Combining this with~\eqref{43t43y34647hersdfbgh}
(used with~$A'':=A'$ and~$A'':=A\setminus A'$),
\begin{equation}\label{eq:AA-final}
u_*(A,A)-u_0^s(A,A) \leq C. 
\end{equation}

Next, we check the interactions in~$A \times A^c$ by first writing
\begin{align*}
u_*(A,A^c)&-u_0^s(A,A^c)\\
	&\leq \big[u_*(A', A^c) -u_0^s(A', A^c)\big]
		+ \big[u_*(A \setminus A', A^c) - u_0^s(A \setminus A', A^c)\big].
\end{align*}
Since~$|u_*| \leq 1$ and with~\eqref{eq:ln-limit}, we estimate, for~$s \in \left(\frac12,s_0\right)$,
\begin{eqnarray*}&&
u_*(A', A^c) -u_0^s(A', A^c)
\leq \int_{A'} \int_{A^c} \frac{|u^*(x) - u^*(y)|^2}{|x-y|^{1+2s}} \, dy \, dx \\
&&\qquad\qquad\leq 4 \int_{-r_{s,\delta}+1}^{-1} \left[ \int_{-\infty}^{-r_{s,\delta}} (x-y)^{-1-2s} \, dy 
		+ \int_0^{+\infty} (y-x)^{-1-2s} \, dy \right] \, dx\\
	%&= \frac{4}{2s} \int_{-r+1}^{-1} \left[ (x+r)^{-2s}  
	%	+ (-x)^{-2s}  \right] \, dx\\
	%&= \frac{2}{s(2s-1)} \left[ -(x+r)^{1-2s} + (-x)^{-2s} \right]
	&&\qquad\qquad= \frac{4}{s(2s-1)} \big[ 1-(r_{s,\delta}-1)^{1-2s} \big] \leq C \ln r_{s,\delta}. 
\end{eqnarray*}
Taking~$s = \frac12$ gives the same estimate.

On the other hand, we use~\eqref{eq:AA'-far} and estimate as in~\eqref{9486tfudksegfeuwiEESEG} for all~$s \in \left[\frac12,s_0\right]$ to find 
\begin{eqnarray*}&&
u_*(A \setminus A', A^c) - u_0^s(A \setminus A', A^c) 
	\\&&\qquad\leq \int_{A\setminus A'} \int_{\{|x-y|<1\}} \frac{|u_*(x) - u_*(y)|^2-|u_0^s(x) - u_0^s(y)|^2}{|x-y|^{1+2s}} \, dy \, dx
 + C 
\leq C. 
\end{eqnarray*}
Combining the previous two displays gives
\begin{equation}\label{eq:AAc-final}
u_*(A,A^c)-u_0^s(A,A^c) \leq C(1+ \ln r_{s,\delta}). 
\end{equation}

Thus, the estimate in~\eqref{eq:QA-final} follows from~\eqref{eq:AA-final} and~\eqref{eq:AAc-final}. 

By~\eqref{eq:QA-r} and~\eqref{eq:QA-final}, there exists~$C>0$ such that,
for all~$s \in \left[ \frac12,s_0\right]$,
\[
r_{s,\delta} \leq C(1 + \ln r_{s,\delta}).
\]
Consequently, $r_{s,\delta}$ is bounded uniformly in~$s \in \left[ \frac12,s_0\right]$, thus giving the claim in~\eqref{nbvcxjyhtgrfew464}. 

We are now ready to show~\eqref{eq:u0-decay}. 
The proof follows the lines of~\cite[Proposition~3]{PAL} except that we have to carefully track the dependence on~$s \in \left[\frac12,s_0\right]$. 
For the sake of the reader, we sketch the idea. 

First, note that there exists some~$\delta>0$ such that
\[
W'(t_2) \geq W'(t_1) + \delta(t_2-t_1) \qquad \hbox{for }
-1 \leq t_1 \leq t_2 \leq -1+\delta. 
\]
For this particular~$\delta>0$ and for~$s \in \left[\frac12,s_0\right]$, let~$\omega_s$ be the barrier in Lemma~\ref{lem:barrier} and take~$R \geq C$. 
As in the proof of~\cite[Proposition~3]{PAL} (note that all the
constants in~\cite[Corollary~4 and Lemma~9]{PAL}
can be made uniform in~$s$ in compact subsets~$[s_1,s_2] \subset (0,1)$), there exist~$\bar{k}$, $\bar{x} \in \R$ (possibly depending on~$s$) such that
\[
\omega_s(\bar{x} - \bar{k}) = u_0^s(\bar{x}) > -1+\delta \qquad \hbox{and} \qquad
\omega_s(x - \bar{k}) \geq u_0^s(x) \;\hbox{ for all }x \in \R.
\]
Moreover, one can show that there exists~$C'>0$ such that, for all~$s \in \left[\frac12,s_0\right]$,
\[
\bar{x} - \bar{k} \in [R-C',R].
\]
Recall from above that~$r_{s,\delta}$ is such that~$u_0^s(-r_{s,\delta}) = -1+\delta$ and that there exists some~$r_\delta$ such that~$r_\delta \geq r_{s,\delta}$ for all~$s \in \left[\frac12,s_0\right]$. Since~$u_0^s$ is strictly increasing, it must be that~$-r_\delta \leq \bar{x}$ and 
\[
u_0^s(x - r_\delta) \leq u_0^s(x+\bar{x}) \quad \hbox{for all}~x \in \R. 
\]

Now take~$y \in \left[\frac{R}{2},R\right]$. As in~\cite{PAL}, one can check that
\[
\bar{x} - y - \bar{k} \in \left[-\frac{R}{2}, \frac{R}{2}\right] \quad \hbox{for all}~s \in \left[\frac12,s_0\right],
\]
so, with~\eqref{eq:barrier-decay}, we have that
\[
1 + \omega_s(\bar{x} - y - \bar{k}) \leq C (R/2)^{-2s} \leq 4Cy^{-2s}.
\]
Hence, it holds that
\[
u_0^s(-r_\delta -y) \leq u_0^s(\bar{x} - y) \leq \omega_s(\bar{x}-y-\bar{k}) \leq -1 + 4Cy^{-2s}.
\]
Since~$r_\delta$ is a constant and~$R$ can be made arbitrarily large, \eqref{eq:u0-decay} holds. 
\end{proof}

\begin{corollary}\label{lem:u0-limit}
There exists a subsequence~$u_0^{s_k}$ such that
\[
\lim_{s_k \searrow \frac12} u_0^{s_k}  = u_0^{\frac12} \quad \hbox{locally uniformly in}~\R. 
\]
\end{corollary}

\begin{proof}
By~\eqref{eq:u0-uniform-holder}, there exist
a subsequence and a measurable function~$v$ such that
\[
\lim_{s \searrow \frac12} u_0^{s}(x) = v(x)\quad \quad \hbox{locally uniformly in}~\R.  
\]
Note that~$v:\R \to [-1,1]$ is non-decreasing, $v(0) = 0$, and~$v$ solves
\[
 (-\Delta)^{\frac12} v(x)+W'(v(x))=0 \quad \hbox{for all}~x \in \R. 
\]

We claim that~$v = u_0^{\frac12}$. 
Indeed, by Lemma~\ref{lem:u0-uniform-asymptotics}, there exist~$C$, $R\geq 1$ such that, for all~$|x| \geq R$,
\[
|v(x) - \sgn(x)| 
	= \lim_{s \searrow \frac12} |u_0^s(x) -\sgn(x)| 
	\leq \lim_{s \searrow \frac12} \frac{C}{|x|^{2s}} 
	= \frac{C}{|x|}.
\] 
In particular, 
\[
\lim_{x \to \pm \infty} v(x) = \pm 1.
\]
We now have that~$v$ satisfies~\eqref{eq:1DinR} for~$s = \frac12$. By uniqueness (see~\cite{PAL}), we then have that~$v = u_0^{\frac12}$, as desired.
\end{proof}

%%%%%%%%%%%%%%%
\section{Connections at the boundary and proof of Theorem~\ref{thm:1Dmin}}%{Construction of the penalization function $\Psi$ at the boundary}
\label{sec:connections}
%%%%%%%%%%%%%%%

In this section, we construct the penalization function~$\Psi$ in~\eqref{eq:Psi-intro} and prove Theorem~\ref{thm:1Dmin}. 

We recall the setting of~$X_\gamma$ stated on page~\pageref{inventatiunatlabel}
and define the functional
\begin{equation}\label{eq:G}
\G_s(u) := 
\begin{cases}
 \G_s(u,\R^-) & \hbox{if}~s \in \left(\frac{1}{2},1\right), \\
\displaystyle \liminf_{R \to +\infty} 
 	\left(\frac{ \G_s(u,B_R^-)}{\ln R}\right) & \hbox{if}~s = \frac{1}{2},
 \end{cases}
\end{equation}
where~$\G_s(u, \cdot)$ is given in~\eqref{eq:G-general-domain}. 
Note the contrast with~\eqref{eq:G-s-limit}.

%\begin{theorem}\label{thm:1Dmin}
%Given \red{$\gamma \in (-1,1)$}, there exists a unique \emph{global minimizer} $w_0(x) \in X_\gamma$ of $\mathcal{G}$. 
%Moreover,  $w_0$ is strictly increasing in $\R^-$, 
%$w_0(x) \in (-1,\gamma)$ for $x \in \R^-$, 
%$w_0 \in C^s(\R) \cap C^\alpha_{\text{loc}}(\R^-)$ for all $\alpha \in (0,1)$, and
%$w_0$ solves
%\begin{equation}\label{eq:PDE-half}
%\begin{cases}
%(-\Delta)^s w_0 + W'(w_0) = 0 & \hbox{in}~(-\infty,0) \\
%w_0 = \gamma & \hbox{in}~(0,\infty) \\
%\displaystyle \lim_{x \to -\infty} w_0(x)= -1. &
%\end{cases}
%\end{equation}
%Furthermore, there are constants $C,R \geq 1$ such that 
%\begin{equation}\label{eq:watinfinity}
%0 \leq w_0(x) +1 \leq \frac{C}{|x|^{2s}} \quad \hbox{for all}~x < -R
%\end{equation}
%and for all $0  < \alpha < 1$, there is a $C\geq 1$ such that
%\begin{equation}\label{eq:w-holder}
%[w_0]_{C^\alpha((-\infty,-R))} \leq \frac{C}{R^\alpha}. 
%\end{equation}
%If $s \in (\frac12,1)$, then we additionally have that $w_0 \in C_{\text{loc}}^{2s}(\R^-)$ and there exists $C,R\geq 1$ such that
%\begin{equation}\label{eq:watinfinity-deriv}
%|w_0'(x)| \leq \frac{C}{|x|}~\hbox{for all}~x < -R \quad \hbox{and} \quad
%[w_0]_{C^{2s-1}((-\infty,-R))} \leq \frac{C}{R^{2s}}. 
%\end{equation}
%%If $\gamma = -1$, then the minimizer is $w_0 \equiv -1$. \red{If $\gamma = +1$, then \dots} 
%\end{theorem}

Let~$w_0(x;+1,\gamma)$ denote the minimizer of~$\G_s$ in the class~$X^\gamma$
(whose definition is obtained replacing~$-1$ with~$1$ in the definition of~$X_\gamma$) and
\begin{equation}\label{wie7658ktoupkjhgf}
{\mbox{$w_0(x;-1,\gamma)$ denote the minimizer in the class~$X_\gamma$ described in Theorem~\ref{thm:1Dmin}.}}\end{equation}
The penalization function~$\Psi: \{\pm1\} \times (-1,1) \to (0,+\infty)$ is then defined as 
\begin{equation}\label{eq:Psi}
\Psi(\pm1,\gamma) := 
	\G_s(w_0(\cdot; \pm 1, \gamma)).
\end{equation}

\medskip

We are left to prove Theorem~\ref{thm:1Dmin} for~$w_0(x) := w_0(x;-1,\gamma)$ for a fixed~$\gamma \in (-1,1)$. 
The proof is split into several lemmata. 
We first establish properties of non-decreasing solutions to the PDE in~\eqref{eq:PDE-half}. 
Then, we prove existence of minimizers for~$s \in \left(\frac12,1\right)$ and show that they satisfy~\eqref{eq:PDE-half}. 
By sending~$s \searrow \frac12$, we prove existence of minimizers for~$s = \frac12$. 
Next, we collectively show asymptotic behavior at~$-\infty$ for all~$s \in \left[\frac12,1\right)$ and finally prove uniqueness. 

For clarity, let us state the definition of local and global minimizers in our setting.

\begin{definition}
We say that a function~$w$ is a \emph{local minimizer} of~$\G_s$ in~$B_R^-$ if~$w=\gamma$ in~$\R^+$ and
\[
\G_s(w,B_R^-) \leq \G_s(w + \phi, B_R^-)\qquad \substack{\hbox{for any measurable}~\displaystyle \phi~\hbox{such that} \\
\displaystyle \operatorname{supp} \phi \subset B_R^-~\hbox{and}~w + \phi \leq \gamma.}
\]
We say that~$w$ is a \emph{global minimizer} of~$\G_s$ if~$w$ is a local minimizer of~$\G_s$ in~$B_R^-$ for all~$R>0$ and if~$\G_s(w) < +\infty$.
\end{definition}

%%%%%%%%%%%%%%%%%%
\subsection{Properties of solutions}
%%%%%%%%%%%%%%%%%%

For later reference, we start with the following properties of solutions to the PDE in~\eqref{eq:PDE-half}. 

\begin{lemma}\label{lem:w-PDE-properties}
Let~$s \in \left[\frac12,1\right)$ and~$\gamma \in (-1,1)$. 
Let~$w_0$ be a non-decreasing solution to~\eqref{eq:PDE-half} such that~$-1 \leq w_0(x) < \gamma$ for all~$x \in \R^-$. 

It holds that~$w_0 \in C^s(\R) \cap C^{\alpha}_{\text{loc}}(\R^-)$ for any~$\alpha \in (0,1)$, $w_0$ is strictly increasing in~$\R^-$, and~$w_0 > -1$.

If~$s \in \left(\frac12,1\right)$, then additionally~$w_0 \in C^{2s}_{\text{loc}}(\R^-)$. 
\end{lemma}

\begin{proof}
The interior regularity~$w_0 \in C^{\alpha}_{\text{loc}}(\R^-)$ (and~$C^{2s}_{\text{loc}}(\R^-)$ for~$s \not=\frac12$) follows from~\cite[Theorem~1.1]{ROSerra} and the global regularity~$w_0 \in C^s(\R)$ follows from~\cite[Proposition~2.6.15]{RosOtonBook}. 

Now, we show that~$w_0 > -1$. For this, suppose by contradiction
that there exist a point~$x_0$ such that~$w_0(x_0) = -1$. Since~$\gamma>-1$, it must be that~$x_0<0$, so we have
\[
0 = W'(w_0(x_0)) = -(-\Delta)^s w_0(x_0) = \int_{\R} \frac{w_0(y) +1}{|x_0-y|^{1+2s}} \, dy.
\]
Since~$w_0+1 \geq 0$, it follows that~$w_0 \equiv -1$ in~$\R$, a contradiction.

Lastly, we show that~$w_0$ is strictly increasing in~$\R^-$. Suppose, by way of contradiction, that there exist~$a < b < 0$ such that~$w_0(a) = w_0(b)$. Observe that
\[
(-\Delta)^sw_0(a) - (-\Delta)^sw_0(b) = W'(w_0(b)) - W'(w_0(a)) = 0. 
\]
Therefore, 
\[
0 = \int_{\R} \frac{w_0(a) - w_0(a+y)}{|y|^{1+2s}} \, dy -
 \int_{\R} \frac{w_0(b) - w_0(b+y)}{|y|^{1+2s}} \, dy 
 =\int_{\R} \frac{w_0(b+y)- w_0(a+y)}{|y|^{1+2s}} \, dy.
 \]
 Since~$w_0$ is non-decreasing and~$b > a$, the integrand is nonnegative. Consequently, $w_0(b+y)=w_0(a+y)$ for all~$y \in \R$, which is false (for example take~$y = -b$). Therefore, $w_0$ is strictly increasing in~$\R^-$. 
\end{proof}

%%%%%%%%%%%%%%%%%%
\subsection{Existence of solutions for~$s \in \left(\frac12,1\right)$}
%%%%%%%%%%%%%%%%%%

Here, we prove existence of global minimizers~$w_0$ in Theorem~\ref{thm:1Dmin} for~$s \in \left(\frac12,1\right)$ and then show that, after translating, the heteroclinic function~$u_0$ in Section~\ref{sec:heteroclinic}
is above~$w_0$ and touches~$w_0$ at the origin (and therefore
it can be used as a barrier). 

\begin{proposition}\label{lem:existence-s}
Let~$s \in \left(\frac12,1\right)$ and~$\gamma \in (-1,1)$.

There exists a global minimizer~$w_0 \in X_\gamma$ of~$\mathcal{G}_s$ such that~$w_0 \in C^s(\R) \cap C^{2s}_{\text{loc}}(\R^-)$,
$w_0$ is strictly increasing in $\R^-$, $-1 < w_0(x) < \gamma$ for all~$x \in \R^-$,
and~$w_0$ solves~\eqref{eq:PDE-half}.
\end{proposition}

\begin{proof}
By Lemma~\ref{lem:h}, there exists a function~$h \in X_\gamma$ such that~$\G_s(h) < C\left(1+\frac{1}{2s-1}\right)$ for some~$C >0$. 
Consequently,
\begin{equation}\label{eq:G-finite}
0 \leq \inf_{w\in X_\gamma} \G_s(w) < +\infty.
\end{equation}

Consider the set of functions
\[
X_\gamma^\star :=  \left\{u \in X_\gamma ~\hbox{s.t.}~\substack{\displaystyle -1 \leq w \leq \gamma,~\G_s(w)<+\infty,~\hbox{and}\\
\displaystyle w~\hbox{is non-decreasing}}\right\}.
\]
We claim that
\begin{equation}\label{eq:claim-G-monotone}
\inf_{w\in X_\gamma} \G_s(w) = \inf_{w\in X_\gamma^\star} \G_s(w).
\end{equation}
Since~$X_\gamma^\star \subset X_\gamma$, we have that~$\inf_{w\in X_\gamma} \G_s(w) \leq \inf_{w\in X_\gamma^\star} \G_s(w)$.
For the reverse inequality, 
first note that~$\G_s$ is decreasing under truncations at~$-1$, so it is not restrictive to minimize among all functions~$w \in X_\gamma$ satisfying~$-1 \leq w \leq \gamma$ such that~$\G_s(w)<+\infty$. 

It remains to check that
\begin{equation}\label{gitrgy4rtfgajhrt3}
{\mbox{$\G_s$ is decreasing under monotone rearrangements.}}\end{equation} 
This is a consequence of~\cite[Theorem~2.11]{Alberti} (see also~\cite[Theorem~9.2]{AlmgrenLiem} for symmetric decreasing rearrangements). 
Indeed, consider the superlevel sets of~$w$
\[
A_\ell(w) := \{x \in \R \;{\mbox{ s.t. }}\; w(x)\geq \ell\}, \quad {\mbox{ with }}\ell \in (-1,\gamma].
\]
We claim that
\begin{equation}\label{mktnjibhef325893634hbdf}\begin{split}&
{\mbox{for~$w \in X_\gamma$ satisfying~$\G_s(w)< +\infty$}}\\
&{\mbox{the symmetric difference~$A_\ell (w)\triangle \R^+$ has finite measure for all~$\ell \in (-1,\gamma]$.}}\end{split}\end{equation}
For this, note that 
\[
A_\ell(w) \triangle \R^+ = (\R^+ \setminus A_\ell(w)) \cup (A_\ell(w) \setminus \R^+) = A_\ell(w) \cap \R^-. 
\]
If~$|A_\ell(w) \cap \R^-| = +\infty$, then
\[
\G_s(w) \geq \int_{A_\ell(w) \cap \R^-} W(w(x)) \, dx
	\geq \inf_{r \in [\ell,\gamma]} W(r) \, |A_\ell(w) \cap \R^-| = + \infty
\] 
which gives a contradiction. This establishes~\eqref{mktnjibhef325893634hbdf}.

We can now apply~\cite[Theorem~2.11]{Alberti} to find that the increasing rearrangement~$w^*$ of~$w$ satisfies~$\G_s(w^*) \leq \G_s(w)$,
and therefore~\eqref{gitrgy4rtfgajhrt3} holds true.

Consequently, $ \inf_{w\in X_\gamma^\star} \G_s(w) \leq \inf_{w\in X_\gamma} \G_s(w)$, and so the proof of~\eqref{eq:claim-G-monotone} is complete. 

In light of~\eqref{eq:claim-G-monotone}, we now
let~$\{w_k\} \subset X_\gamma^\star$ be a minimizing sequence
for~$\G_s$, namely
\[
\lim_{k \to +\infty} \G_s(w_k) = \inf_{w\in X_\gamma} \G_s(w).
\]
As a consequence of~\cite[Theorem~7.1]{Hitchhikers}, 
up to a subsequence, there exists a measurable function~$w_0$ such that~$w_k \to w_0$ almost everywhere as~$k \to +\infty$. 
By construction, $w_0 = \gamma$ in~$\R^+$, $-1 \leq w_0\leq \gamma$, and~$w_0$ is non-decreasing. 
Also, by Fatou's Lemma, 
\begin{equation}\label{eq:w0-Fatou}
\G_s(w_0)
	\leq \liminf_{k \to +\infty}  \G_s(w_k) =  \inf_{w\in X_\gamma} \G_s(w) < +\infty,
\end{equation}
which shows that~$w_0$ is a (global) minimizer of~$\G_s$.
 
Furthermore, since~$w_0$ is non-decreasing and bounded, the limit
\[
\lim_{x \to -\infty} w_0(x) =: a \in [-1,\gamma]
\]
exists. 
If~$a \not=-1$, then 
\[
\int_{\R^-} W(w_0) \, dx = +\infty
\]
which contradicts~\eqref{eq:w0-Fatou}. Therefore, $a=-1$ and we have that~$w_0 \in X_\gamma$. 

Moreover, by Lemma~\ref{lem:gamma-plateau}, $w_0 < \gamma$ in~$\R^-$. With this and the minimizing property, $w_0$ satisfies~\eqref{eq:PDE-half}. 
The remaining properties follow from Lemma~\ref{lem:w-PDE-properties}. 
\end{proof}

We will now use the solution~$u_0$ to~\eqref{eq:1DinR} as a barrier for~$w_0$. 
Towards this end, let~$x_\gamma \in \R$ be the unique point at which~$u_0(x_\gamma) = \gamma$ and set  
\[
u_\gamma(x) := u_0(x + x_\gamma).
\]
Via the sliding method (see~\cite{PAL}), we show the following. 

\begin{lemma}\label{lem:sliding-ugamma}
Let~$s \in \left(\frac12,1\right)$ and~$\gamma \in (-1,1)$.

It holds that
\begin{equation}\label{eq:eq:ugamma-w0}
u_\gamma(x) \geq w_0(x)\;  \hbox{ for all } x \in \R \qquad \hbox{and} \qquad
u_\gamma(0) = w_0(0) = \gamma. 
\end{equation}
\end{lemma}

\begin{proof}
First, since~$-1 \leq u_\gamma, w_0 \leq \gamma$ in~$\R^-$ and~$\gamma = w_0 \leq u_\gamma \leq 1$ in~$\R^+$, 
we have that, for all~$\eps>0$, there exists~$k_\eps \in \R$ such that
\[
u_\gamma(x-k) + \eps > w_0(x) \quad \hbox{for all}~k \geq k_\eps~\hbox{and all}~x \in \R. 
\]
Take~$k_\eps$ as large as possible so that
\begin{equation}\label{eq:sliding-eps-inequality}
u_\gamma(x-k_\eps) + \eps \geq w_0(x) \quad \hbox{for all}~x \in \R
\end{equation}
and, for all~$j \in \N$, there exist~$\eta_{j,\eps}>0$ and~$x_{j, \eps} \in \R$ such that~$\lim_{j \to +\infty} \eta_{j,\eps} =0$ and 
\begin{equation}\label{eq:sliding-eps-j}
u_\gamma(x_{j,\eps}-(k_\eps +\eta_{j,\eps})) + \eps < w_0(x_{j,\eps}).
\end{equation}

We claim that
\begin{equation}\label{9403ytehwig45654u}
{\mbox{$x_{j,\eps}$ is a bounded sequence in~$j$.}}
\end{equation}
Indeed, if there existed a subsequence such that~$x_{j,\eps} \to +\infty$, then from~\eqref{eq:sliding-eps-j} and the continuity of~$w_0$ and~$u_\gamma$, 
\[
1 + \eps = \lim_{j \to +\infty} u_\gamma(x_{j,\eps}-(k_\eps + \eta_{j,\eps})) + \eps \leq \lim_{j \to +\infty} w_0(x_{j,\eps}) = \gamma,
\]
a contradiction. 
If instead there existed a subsequence such that~$x_{j,\eps} \to -\infty$, then, similarly, 
\[
-1 + \eps = \lim_{j \to +\infty} u_\gamma(x_{j,\eps}-(k_\eps + \eta_{j,\eps})) + \eps \leq \lim_{j \to +\infty} w_0(x_{j,\eps}) = -1,
\]
a contradiction. Hence, \eqref{9403ytehwig45654u} is proved.

Thus, thanks to~\eqref{9403ytehwig45654u}, up to a subsequence, there exists a point~$x_\eps \in \R$ such that
\[
\lim_{j \to +\infty} x_{j,\eps} = x_\eps. 
\]
Therefore, taking the limit in~$j$ in~\eqref{eq:sliding-eps-j} and recalling~\eqref{eq:sliding-eps-inequality}, we have that
\begin{equation}\label{eq:sliding-x-eps-equality}
u_\gamma(x_\eps-k_\eps) + \eps = w_0(x_\eps).
\end{equation}

Now, if~$x_\eps>0$ for some~$\eps>0$, by the strict monotonicity of~$u_\gamma$, \eqref{eq:sliding-eps-inequality} and~\eqref{eq:sliding-x-eps-equality},
\[
u_\gamma(0-k_\eps) + \eps < u_\gamma(x_\eps - k_\eps) + \eps = w_0(x_\eps) = w_0(0) \leq u_\gamma(0 -k_\eps) + \eps,
\]
a contradiction. 
Therefore, $x_\eps \leq 0$ for all~$\eps>0$. 

%Suppose for all $\eps_0>0$, there is a $0 < \eps \leq \eps_0$ such that $x_\eps = 0$. 
Now, suppose that there exists a subsequence such that~$x_\eps = 0$ for all~$\eps$. By~\eqref{eq:sliding-x-eps-equality} and since~$w_0(x_\eps) = \gamma$, we have that, for~$\eps>0$
sufficiently small,
\[
u_\gamma(-k_\eps) = \gamma - \eps \in (-1,1). 
\]
Since~$u_\gamma^{-1}:(-1,1) \to \R$ exists and is continuous, we have that 
\[
\lim_{\eps \searrow 0} k_\eps =  -\lim_{\eps \searrow 0} u_\gamma^{-1}(\gamma-\eps) = -u_\gamma^{-1}(\gamma) = 0. 
\]
Therefore, taking the limit in~\eqref{eq:sliding-eps-inequality} and~\eqref{eq:sliding-x-eps-equality} as~$\eps \searrow 0$ gives~\eqref{eq:eq:ugamma-w0} in this case. 

Consider now the scenario in which there exists~$\eps_0>0$ such that~$x_\eps<0$ for all~$\eps \in (0,\eps_0)$. 
Define the function 
\[
u_\gamma^\eps(x):= u_\gamma(x-k_\eps) + \eps. 
\]
Note that,
in light of~\eqref{eq:sliding-eps-inequality} and~\eqref{eq:sliding-x-eps-equality},
$u_\gamma^\eps(x) \geq w_0(x)$ for all~$x \in \R$ and~$u_\gamma^\eps(x_\eps) = w_0(x_\eps)$. Also,
for all~$x \in \R$,
\begin{align*}
- (-\Delta)^s u_\gamma^\eps(x) 
	= -(-\Delta)^s u_\gamma(x -k_\eps) 
	= W'(u_\gamma(x-k_\eps))
	= W'(u_\gamma^\eps(x) - \eps). 
\end{align*}
Consequently, since we have the equation for~$w_0$ at~$x_\eps<0$,
\begin{equation}\label{eq:ugamma-w0-eps-equation}
\begin{split}&
0 \leq \int_{\R} \frac{(u_\gamma^\eps- w_0)(y)}{|x_\eps-y|^{1+2s}} \, dy
	= -(-\Delta)^s (u_\gamma^\eps- w_0)(x_\eps)\\
	&\qquad= W'(u_\gamma^\eps(x_\eps) - \eps) - W'(w_0(x_\eps)) 
	= W'(w_0(x_\eps)- \eps) - W'(w_0(x_\eps)).
\end{split}
\end{equation}

We point out that
\begin{equation}\label{ewuogfougvbajdksveg94560}
{\mbox{$x_\eps$ is also bounded from below.}}\end{equation}
Indeed, if there existed a subsequence such that~$x_\eps \to -\infty$, then 
\begin{equation}\label{eq:w0-eps-unbounded}
\lim_{x \to -\infty} w_0(x_\eps) = -1. 
\end{equation}
Since~$W''(-1)>0$, there exists~$\delta>0$ such that
\begin{equation}\label{eq:potential-near-neg1}
W'(t) \geq W'(r) + \delta(t-r) \qquad \hbox{for all}~-1 \leq r \leq t \leq -1+\delta. 
\end{equation}
By~\eqref{eq:w0-eps-unbounded}, there exists~$\eps_1>0$ such that,
for all~$\eps \in (0,\eps_1)$, it holds that~$-1 \leq w_0(x_\eps)-\eps < w_0(x_\eps) \leq -1+\delta$. Therefore, we can employ~\eqref{eq:potential-near-neg1} and find that
\begin{equation}\label{eq:W-w0-xeps-bound}
W'(w_0(x_\eps)) \geq W'(w_0(x_\eps)-\eps) + \delta \eps > W'(w_0(x_\eps)-\eps),
\end{equation}
contradicting~\eqref{eq:ugamma-w0-eps-equation}
and thus proving~\eqref{ewuogfougvbajdksveg94560}. 

Thanks to~\eqref{ewuogfougvbajdksveg94560}, up to a subsequence, there exists a point~$x_0 \leq 0$ such that
\[
\lim_{\eps \searrow 0} x_\eps = x_0. 
\]

We claim that~$k_\eps$ is a bounded sequence in~$\eps$. Indeed, if there existed a subsequence such that~$k_\eps \to -\infty$, then from~\eqref{eq:sliding-x-eps-equality}
\[
1 = \lim_{\eps \searrow 0} u_\gamma(x_\eps - k_\eps) + \eps = \lim_{\eps \searrow 0} w_0(x_\eps) = w_0(x_0) \leq \gamma,
\]
a contradiction. 
If instead there existed a subsequence such that~$k_\eps \to +\infty$, then
\[
-1 = \lim_{\eps \searrow 0} u_\gamma(x_\eps - k_\eps) + \eps = \lim_{\eps \searrow 0} w_0(x_\eps) = w_0(x_0) > -1,
\]
a contradiction.

Therefore~$k_\eps$ is bounded, and so,
up to a subsequence, there
exists~$k_0 \in \R$ such that 
\[
\lim_{\eps \searrow 0} k_\eps = k_0. 
\]
We thus send~$\eps \searrow 0$ in~\eqref{eq:sliding-eps-inequality} and~\eqref{eq:sliding-x-eps-equality} to obtain
\begin{equation}\label{eq:ugamma-w0-limit}
u_\gamma(x-k_0) \geq w_0(x)\; \hbox{ for all }
x \in \R \qquad \hbox{and} \qquad u_\gamma(x_0 - k_0) = w_0(x_0).
\end{equation}
Accordingly, the desired result in~\eqref{eq:eq:ugamma-w0}
is proved if we show that
\begin{equation}\label{des4783564hkresu}
x_0 = k_0 = 0.\end{equation} 
Hence, we now focus on the proof of~\eqref{des4783564hkresu}.

Suppose by contradiction that~$x_0 < 0$. 
Sending~$\eps \searrow 0$ in~\eqref{eq:ugamma-w0-eps-equation} gives
\[
\int_{\R} \frac{u_\gamma(y-k_0) - w_0(y)}{|x_0-y|^{1+2s}} \, dy = 0.
\]
{F}rom~\eqref{eq:ugamma-w0-limit}, we know that the integrand is nonnegative. Therefore, it must be that~$u_\gamma(x-k_0) = w_0(x)$ for all~$x \in \R$, which is false for~$x>\max\{0,k_0\}$. Therefore,
we have that~$x_0 = 0$.

Moreover, since~$u_\gamma$ is invertible and~$u_\gamma( - k_0) = w_0(0) = \gamma$, we have that~$
-k_0 = u_\gamma^{-1}(\gamma) = 0$,
which completes the proof of~\eqref{des4783564hkresu}, as desired. 
\end{proof}

%%%%%%%%%%%%%%%%%%
\subsection{Existence of solutions for~$s=\frac12$}
%%%%%%%%%%%%%%%%%%

Now, we establish existence of global minimizers in Theorem~\ref{thm:1Dmin} for~$s =\frac12$. 
For this, we will show that a subsequence in~$s \in \left(\frac12,1\right)$ of the minimizers given by Proposition~\ref{lem:existence-s} converges as~$s \searrow \frac12$ to a minimizer for~$s = \frac12$. 

%For $R>0$, we emphasize the dependence on $s$ by writing
%\begin{equation} \label{eq:G-delta-0}
%\begin{aligned}
%\G_s(u,[-R,0])
%	&:= \iint_{[-R,0] \times [-R,0]} \frac{|u(x) - u(y)|^2}{|x-y|^{1+2s}} \, dx \, dy\\
%	&\quad+2\iint_{[-R,0] \times [-R,0]^c} \frac{|u(x) - \gamma|^2}{|x-y|^{1+2s}} \, dx \, dy
%	+ \int_{[-R,0]}W(u(x)) \, dx \\
%\G_s(u) 
%	&:= \begin{cases}
%	 \G_s(u,\R^-) & \hbox{if}~s \in (\frac12,1) \\
%	\displaystyle \lim_{R \to \infty} \frac{\G_{\frac12}(u,[-R,0])}{\ln R} & \hbox{if}~s = \frac12.
%	 \end{cases}
%\end{aligned}
%\end{equation}
For~$s \in \left(\frac12,1\right)$, let~$w_s$ denote the global minimizer of~$\mathcal{G}_s$ from Proposition~\ref{lem:existence-s}.  

\begin{lemma}\label{lem:existence-limit}
Let~$\gamma \in (-1,1)$. There exist~$\alpha \in (0,1)$ and~$C>0$ such that, for all~$s \in \left(\frac12,1\right)$,
\begin{equation}\label{eq:wdelta-holder}
\|w_s\|_{C^{\alpha}(\R)} \leq C. 
\end{equation}

Furthermore, there exists a subsequence~$w_{s_k}$ that
converges locally uniformly as~$s_k \searrow \frac12$ to a non-decreasing function~$w_0: \R \to [-1,\gamma]$ such that~$w_0\equiv\gamma$ in~$\R^-$ and~\eqref{eq:eq:ugamma-w0} holds for~$s = \frac12$. 
\end{lemma}

\begin{proof}
Since~$-1 \leq w_s \leq \gamma$, 
to prove~\eqref{eq:wdelta-holder},
we only need to check the uniform bound on the $C^\alpha$-seminorm.
By Proposition~\ref{lem:existence-s} and~\cite[Theorem~1]{ServadeiValdinoci}, $w_s \in H^{s}(\R)$ is also a viscosity solution to 
\[
\begin{cases}
 (-\Delta)^{s} w_s+W'(w_s)=0 & \hbox{in}~\R^-,\\
 w_s = \gamma & \hbox{in}~\R^+,\\
 -1 <w_s < \gamma & \hbox{in}~\R^-.
\end{cases}
\]

Note that there exists~$C_0>0$ such that, for all~$s \in \left(\frac12,1\right)$,
\begin{equation}\label{eq:CS-condition}
\int_{\R} \frac{|w_s(x)|}{1 + |x|^{2}} \, dx \leq \int_{\R} \frac{1}{1 + |x|^{2}} \, dx < +\infty \qquad \hbox{and} \qquad |W'(w_s)| \leq C_0.
\end{equation}
Consider a ball~$B_1(x_0) \Subset \R^-$. 
By~\cite[Theorem~26]{CS} with~\eqref{eq:CS-condition}, there exist~$\alpha_1 \in (0,1)$ and~$C>0$ such that, for every~$s \in \left(\frac12,1\right)$,
\[
[w_s]_{C^{\alpha_1}(B_{1/2}(x_0))} \leq C.
\]
On the other hand, by~\cite[Proposition~1.1]{ROSduke},
there exist~$\alpha_2 \in (0,1)$ and~$C>0$ such that, for every~$s \in \left(\frac12,1\right)$, 
\[
\left[\frac{w_s}{|x|^{s}}\right]_{C^{\alpha_2}(\overline{B_{1/2}^-(0)})} \leq C.
\]
The previous two displays yield~\eqref{eq:wdelta-holder}. 

By~\eqref{eq:wdelta-holder}, there exists a function~$w_0$ such that, up to a subsequence, 
\[
\lim_{s \searrow \frac12} w_s(x) = w_0(x) \quad \hbox{locally uniformly in}~\R. 
\]
Moreover, $w_0: \R \to [-1,\gamma]$, $w_0 \equiv \gamma$ in~$\R^-$, and~$w_0$ is non-decreasing. 

For all~$ s \in \left[\frac12,1\right)$,
let~$u_0^{s}$ denote the corresponding solution to~\eqref{eq:1DinR}. Let~$x_\gamma^s$ be such that~$u_0^s(x_\gamma^s) = \gamma$ and set~$
u_\gamma^s(x) := u_0^s(x+x_\gamma^s) $.
In particular, $u_\gamma^s$ solves
\begin{equation}\label{eq:ugamma0-eqn}
\begin{cases}
 (-\Delta)^{s} u_\gamma^s(x)+W'(u_\gamma^s(x))=0 & \;\hbox{for all}~x \in \R\\
 \displaystyle\lim_{x\to\pm\infty}  u_\gamma^s(x) =\pm1, &\\ u_\gamma^s(0)=\gamma, & \\
(u_\gamma^s)'>0.&
\end{cases}
\end{equation}
As in Corollary~\ref{lem:u0-limit}, we can show that, up to taking another subsequence, 
\begin{equation}\label{unifcnt347325}
\lim_{s \searrow \frac12} u_\gamma^s(x) = u_\gamma^{\frac12}(x)\quad \hbox{locally uniformly in}~\R. 
\end{equation}
Note here that 
\begin{equation}\label{ghtuyb48534b8y65rfgregrty98765fwje}
{\mbox{the sequence~$x_\gamma^s$ is bounded uniformly in~$s \in \left[\frac12,1\right]$.}}\end{equation} 
Indeed, if there existed a subsequence such that~$x_\gamma^s \to \pm \infty$ as~$s \searrow \frac12$, then, by Lemma~\ref{lem:u0-uniform-asymptotics},
\[
\lim_{s \searrow \frac12} |u_0^s(x_\gamma^s) \pm 1| \leq \lim_{s \searrow \frac12} \frac{C}{|x_\gamma^s|^{2s}} \leq \lim_{s \searrow \frac12} \frac{C}{|x_\gamma^s|} = 0.
\] 
On the other hand, since~$u_0^s(x_\gamma^s) = \gamma$ for all~$s \in \left(\frac12,1\right)$,
$$ \lim_{s \searrow \frac12} |u_0^s(x_\gamma^s) \pm 1|
=\lim_{s \searrow \frac12} |\gamma \pm 1|=|\gamma \pm 1|>0,$$
thus giving the desired contradiction and proving~\eqref{ghtuyb48534b8y65rfgregrty98765fwje}.

{F}rom~\eqref{ghtuyb48534b8y65rfgregrty98765fwje} we deduce that~$x_\gamma^s\to \bar x$ as~$s\searrow\frac12$. Moreover, we claim that
\begin{equation}\label{kju6hy5gtftwed6235r43EEEFvd3wyr}
u_0^{\frac12}(\bar x)=\gamma.
\end{equation}
To check this, we observe that 
\begin{eqnarray*}
|u_0^{\frac12}(\bar x)-\gamma|&\le&
|u_0^{\frac12}(\bar x)-u_0^{\frac12}(x_\gamma^s)|
+|u_0^{\frac12}(x_\gamma^s)-\gamma|\\& =&
|u_0^{\frac12}(\bar x)-u_0^{\frac12}(x_\gamma^s)|
+|u_0^{\frac12}(x_\gamma^s)-u_0^{s}(x_\gamma^s)|.
\end{eqnarray*}
Hence, taking the limit as~$s\searrow\frac12$ and using
the continuity of~$u_0^{\frac12}$ and the uniform convergence in~\eqref{unifcnt347325}, we obtain~\eqref{kju6hy5gtftwed6235r43EEEFvd3wyr}.

It remains to check that~\eqref{eq:eq:ugamma-w0} holds for~$s=\frac12$. To this end,
we recall that, for~$s \in \left(\frac12,1\right)$, Lemma~\ref{lem:sliding-ugamma} can be written as
\begin{equation}\label{fuow4yt8434369mnbvcwerae7}
u_\gamma^s(x) \geq w_s(x)\;\hbox{ for all } x \in \R \qquad \hbox{and} \qquad u_\gamma^s(0) = w_s(0) = \gamma. 
\end{equation}
Moreover, for all~$x\in\R$,
\begin{eqnarray*}
|u_\gamma^s(x)-u_0^{\frac12}(x+\bar x)|\le
|u_0^s(x+x_\gamma^s)-u_0^{\frac12}(x+x_\gamma^s)|
+|u_0^{\frac12}(x+x_\gamma^s)-u_0^{\frac12}(x+\bar x)|.
\end{eqnarray*}
Thus, using the uniform convergence in~\eqref{unifcnt347325} and
the continuity of~$u_0^{\frac12}$, we obtain that, for all~$x\in\R$,
$$ \lim_{s\searrow\frac12}
u_\gamma^s(x)=u_0^{\frac12}(x+\bar x).$$
Recalling~\eqref{kju6hy5gtftwed6235r43EEEFvd3wyr}, this gives that, for all~$x\in\R$,
$$ \lim_{s\searrow\frac12}
u_\gamma^s(x)=u_\gamma^{\frac12}(x).$$
We can therefore pass to the limit the inequalities 
in~\eqref{fuow4yt8434369mnbvcwerae7} and obtain that
\begin{equation*}
u_\gamma^{\frac12}(x) \geq w_0(x)\; \hbox{ for all }
x \in \R \qquad \hbox{and} \qquad u_\gamma^{\frac12}(0) = w_0(0) = \gamma.
\qedhere 
\end{equation*}
\end{proof}

\begin{proposition}\label{lem:existence-critical}
The function~$w_0 \in X_\gamma$ in
Lemma~\ref{lem:existence-limit} is a global minimizer of~$ \G_{\frac12}$.

Furthermore, $w_0 \in C^s(\R) \cap C^{\alpha}_{\text{loc}}(\R^-)$ for all~$\alpha \in (0,1)$, $w_0$ is strictly increasing in~$\R^-$,
$-1<w_0(x) < \gamma$ for all~$x\in\R^-$,
and~$w_0$ solves~\eqref{eq:PDE-half} for~$s = \frac12$. 
\end{proposition}

\begin{proof} 
We first check that
\begin{equation}\label{458y95476340517hgwejkqgye98765}
{\mbox{$w_0$ is a local minimizer of~$ \G_{\frac12}$.}}\end{equation} 
For this, fix~$R>0$ and let~$\phi$ be a measurable function with~$\operatorname{supp} \phi \subset B_R^-$ and such that~$w_0+\phi \leq \gamma$. 
We will show that
\begin{equation}\label{eq:G0-local}
{\G}_{\frac12}(w_0+\phi,B_R^-)  - {\G}_{\frac12}(w_0, B_R^-) \geq 0.
\end{equation}
It is enough to consider the case in which~$\phi$ is smooth with~$\operatorname{supp} \phi \Subset B_R^-$. Since~$w_s$ is a local minimizer of~$\G_s$ for~$s\in\left(\frac12,1\right)$, 
\begin{equation}\label{eq:Gs-local}
{\G}_s(w_s + \phi, B_R^-) 
	- {\G}_s(w_s, B_R^-) \geq0.
\end{equation}
By writing out the expression for the energies, we see that~\eqref{eq:Gs-local} is equivalent to
\begin{equation}\label{eq:local-expanded}
\begin{split}
%&{\G}_s(w_s + \phi, B_R^-) 
%	- {\G}_s(w_s, B_R^-)\\
0 &\leq  \int_{B_R^-} \int_{B_R^-} \frac{|(w_s+\phi)(x) - (w_s+\phi)(y)|^2}{|x-y|^{1+2s}}\, dy\, dx\\
	&\quad + 2 \int_{(B_R^-)^c} \int_{B_R^-} \frac{|(w_s+\phi)(x) - (w_s+\phi)(y)|^2}{|x-y|^{1+2s}} \,dy\, dx\\
&\quad -  \int_{B_R^-} \int_{B_R^-} \frac{|w_s(x) - w_s(y)|^2}{|x-y|^{1+2s}}\, dy \,dx
	- 2 \int_{(B_R^-)^c} \int_{B_R^-} \frac{|w_s(x) - w_s(y)|^2}{|x-y|^{1+2s}}\, dy\, dx\\
&\quad +\int_{B_R^-} \big[ W(w_s(x) + \phi(x)) - W(w_s(x)) \big] \, dx\\
%&= \int_{B_R^-} \int_{B_R^-} \frac{|w_s(x) - w_s(y)|^2 + 2 (w_s(x) - w_s(y))(\phi(x) - \phi(y)) + |\phi(x) - \phi(y)|^2}{|x-y|^{1+2s}} dy dx\\
%&\quad + 2 \int_{(B_R^-)^c} \int_{B_R^-} \frac{|w_s(x) - w_s(y)|^2 + 2 (w_s(x) - w_s(y))(\phi(x) - \phi(y)) + |\phi(x) - \phi(y)|^2}{|x-y|^{1+2s}} dy dx\\
%&\quad -  \int_{B_R^-} \int_{B_R^-} \frac{|w_s(x) - w_s(y)|^2}{|x-y|^{1+2s}} dy dx
%	- 2 \int_{(B_R^-)^c} \int_{B_R^-} \frac{|w_s(x) - w_s(y)|^2}{|x-y|^{1+2s}} dy dx\\
%&\quad +\int_{B_R^-} \left[ W(w_s(x) + \phi(x)) - W(w_s(x)) \right] \, dx\\
%&= \int_{B_R^-} \int_{B_R^-} \frac{ 2 (w_s(x) - w_s(y))(\phi(x) - \phi(y)) + |\phi(x) - \phi(y)|^2}{|x-y|^{1+2s}} dy dx\\
%&\quad + 2 \int_{(B_R^-)^c} \int_{B_R^-} \frac{ 2 (w_s(x) - w_s(y))(\phi(x) - \phi(y)) + |\phi(x) - \phi(y)|^2}{|x-y|^{1+2s}} dy dx\\
%&\quad +\int_{B_R^-} \left[ W(w_s(x) + \phi(x)) - W(w_s(x)) \right] \, dx\\
&=2   \int_{B_R^-} \int_{B_R^-} \frac{(w_s(x) - w_s(y))(\phi(x) - \phi(y))}{|x-y|^{1+2s}} \,dy \, dx\\
	&\quad +4  \int_{(B_R^-)^c} \int_{B_R^-} \frac{(w_s(x) - w_s(y))(\phi(x) - \phi(y))}{|x-y|^{1+2s}} \,dy \, dx\\
&\quad + \left[ \int_{B_R^-} \int_{B_R^-} \frac{  |\phi(x) - \phi(y)|^2}{|x-y|^{1+2s}} \,dy\, dx
	+ 2 \int_{(B_R^-)^c} \int_{B_R^-} \frac{ |\phi(x)|^2}{|x-y|^{1+2s}} \,dy\, dx\right]\\
&\quad +\int_{B_R^-} \big[ W(w_s(x) + \phi(x)) - W(w_s(x)) \big]\, dx.
\end{split}
\end{equation}
By the Dominated Convergence Theorem, since~$\phi$ is smooth and bounded,
\begin{align*}
\lim_{s \searrow \frac12} &\left[ \int_{B_R^-} \int_{B_R^-} \frac{  |\phi(x) - \phi(y)|^2}{|x-y|^{1+2s}} \,dy\, dx
	+ 2 \int_{(B_R^-)^c} \int_{B_R^-} \frac{ |\phi(x)|^2}{|x-y|^{1+2s}}\, dy\, dx\right]\\
		&=  \int_{B_R^-} \int_{B_R^-} \frac{  |\phi(x) - \phi(y)|^2}{|x-y|^{2}}\, dy\, dx
	+ 2 \int_{(B_R^-)^c} \int_{B_R^-} \frac{ |\phi(x)|^2}{|x-y|^{2}} \,dy \,dx
\end{align*}
and since $|w_s|\le1$ and, for~$s$ close to~$\frac12$, $|w_s+\phi| \leq 1$,
\[
\lim_{s \searrow \frac12} \int_{B_R^-} \big[ W(w_s(x) + \phi(x)) - W(w_s(x)) \big] \, dx
	= \int_{B_R^-} \big[ W(w_0(x) + \phi(x)) - W(w_0(x)) \big] \, dx.
\]

Moreover, thanks to~\eqref{eq:wdelta-holder},
there exists some~$C>0$, independent of~$s$, such that, for~$s$ close to~$\frac12$,
\begin{align*}
\int_{B_R^-} &\int_{B_R^-}
	\frac{|(w_s(x) - w_s(y))(\phi(x) - \phi(y))|}{|x-y|^{1+2s}} \,dy \, dx\\
&= \int_{B_R^-} \int_{B_R^- \cap \{|x-y|>1\}}
	\frac{|(w_s(x) - w_s(y))(\phi(x) - \phi(y))|}{|x-y|^{1+2s}} \,dy \, dx\\
	&\quad + \int_{B_R^-} \int_{B_R^- \cap \{|x-y|<1\}}
	\frac{|(w_s(x) - w_s(y))(\phi(x) - \phi(y))|}{|x-y|^{1+2s}} \,dy \, dx\\
&\leq \int_{B_R^-} \int_{B_R^- \cap \{|x-y|>1\}}
	\frac{C}{|x-y|^{1+2s}} \,dy \, dx
	 + \int_{B_R^-} \int_{B_R^- \cap \{|x-y|<1\}}
	\frac{C|x-y|^{\alpha+1}}{|x-y|^{1+2s}} \,dy \, dx \leq C
\end{align*}
and 
\begin{align*}
 \int_{(B_R^-)^c} &\int_{B_R^-} \frac{|(w_s(x) - w_s(y))(\phi(x) - \phi(y))|}{|x-y|^{1+2s}}\, dy \, dx\\
&= \int_{(B_R^-)^c} \int_{B_R^- \cap \{|x-y|>1\}} \frac{|(w_s(x) - w_s(y))(\phi(x) - \phi(y))|}{|x-y|^{1+2s}} \,dy \, dx\\
&\quad + \int_{(B_R^-)^c} \int_{B_R^- \cap \{|x-y|<1\}} \frac{|(w_s(x) - w_s(y))(\phi(x) - \phi(y))|}{|x-y|^{1+2s}}\, dy \, dx\\
&= C\int_{(B_R^-)^c} \int_{B_R^- \cap \{|x-y|>1\}} \frac{C}{|x-y|^{1+2s}} dy \, dx\\
&\quad + \int_{(B_R^-)^c} \int_{B_R^- \cap \{|x-y|<1\}} \frac{C|x-y|^{\alpha+1}}{|x-y|^{1+2s}} \,dy \, dx
\leq C. 
\end{align*}
Therefore, by the Dominated Convergence Theorem, 
\begin{align*}
\lim_{s \searrow \frac12}   &\int_{B_R^-} \int_{B_R^-} \frac{(w_s(x) - w_s(y))(\phi(x) - \phi(y))}{|x-y|^{1+2s}} \,dy \, dx\\
	&=    \int_{B_R^-} \int_{B_R^-} \frac{(w_0(x) - w_0(y))(\phi(x) - \phi(y))}{|x-y|^{2}} \,dy \, dx
\end{align*}
and
\begin{align*}
\lim_{s \searrow \frac12}  & \int_{(B_R^-)^c} \int_{B_R^-} \frac{(w_s(x) - w_s(y))(\phi(x) - \phi(y))}{|x-y|^{1+2s}}\, dy \, dx\\
	&=    \int_{(B_R^-)^c} \int_{B_R^-} \frac{(w_0(x) - w_0(y))(\phi(x) - \phi(y))}{|x-y|^{2}} \,dy \, dx.
\end{align*}
Collecting the above limits, we take the limit as~$s \searrow \frac12$ of~\eqref{eq:Gs-local} written as~\eqref{eq:local-expanded}
and obtain that~\eqref{eq:G0-local} holds. Namely, $w_0$ is a local minimizer, thus completing the proof of~\eqref{458y95476340517hgwejkqgye98765}.

We next show that
\begin{equation}\label{458y95476340517hgwejkqgye9876500}
{\mbox{$w_0$ is a global minimizer of~$\G_{\frac12}$.}}\end{equation}
In light of~\eqref{458y95476340517hgwejkqgye98765}, in order to establish~\eqref{458y95476340517hgwejkqgye9876500}
we only need to check that
\begin{equation*}
\G_{\frac12}(w_0) <+\infty.\end{equation*}
For this,
let~$h$ be as in Lemma~\ref{lem:h}. Fix~$s \in \left(\frac12,s_0\right]$ (with~$s_0$ as given in~\eqref{eq:s0}). Notice that the function~$h$ belongs to~$X_\gamma$ (recall
the definition of~$X_\gamma$ on page~\pageref{inventatiunatlabel}) and therefore, since~$w_s$
is a minimizer of~$\G_s$ in~$X_\gamma$, we have that
\begin{equation}\label{439gjk123456sdfgswehkrgRRytwsyt}
\G_s(w_s)  \leq \G_s(h) < +\infty.
\end{equation}

Let now~$R>3$. We point out that if~$u$ is such that~$\G_s(u)<+\infty$, then one can check that
\begin{equation}\label{4308tghsdeesdfghrilu7890}\begin{split}
&\G_s(u)-\G_s(u,[-R,0]) \\
&=  \int_{-\infty}^{-R} \int_{-\infty}^{-R} \frac{|u(x) - u(y)|^2}{|x-y|^{1+2s}} \, dy \, dx
		+  2\int_{-\infty}^{-R} \int_{0}^{+\infty} \frac{|u(x) - u(y)|^2}{|x-y|^{1+2s}} \, dy \, dx\\
		&\qquad\qquad+ \int_{-\infty}^{-R} W(u) \, dx
		%%%% \\&\geq0
		.
\end{split}\end{equation}
Using this with~$u:=h$, 
since~$h \equiv -1$ in~$(-\infty,-R)$ and~$h \equiv \gamma$ in~$\R^+$, we have that
$$
\G_s(h)-\G_s(h,[-R,0]) 
= 2\int_{-\infty}^{-R} \int_{0}^{+\infty} \frac{(1+\gamma)^2}{(y-x)^{1+2s}} \, dy \, dx.
$$
Therefore, exploiting~\eqref{4308tghsdeesdfghrilu7890} with~$u:=w_s$ and
recalling also~\eqref{439gjk123456sdfgswehkrgRRytwsyt},
\begin{align*}
\G_s(w_s,[-R,0])
	&\leq \G_s(w_s) -2\int_{-\infty}^{-R} \int_{0}^{+\infty} \frac{|w_s(x) - w_s(y)|^2}{|x-y|^{1+2s}} \, dy \, dx \\
	&\leq \G_s(h) -2\int_{-\infty}^{-R} \int_{0}^{+\infty} \frac{(\gamma - w_s(x))^2}{(y-x)^{1+2s}} \, dy \, dx \\
	&= \G_s(h,[-R,0]) +2\int_{-\infty}^{-R} \int_{0}^{+\infty} \frac{(\gamma+1)^2-(\gamma - w_s(x))^2}{(y-x)^{1+2s}} \, dy \, dx.
\end{align*}

Now, for~$R$ sufficiently large, Lemmata~\ref{lem:sliding-ugamma} and~\ref{lem:u0-uniform-asymptotics}
imply that
\begin{align*}&
\int_{-\infty}^{-R} \int_{0}^{+\infty} \frac{|(\gamma+1)^2-(\gamma - w_s(x))^2|}{(y-x)^{1+2s}} \, dy \, dx
= \frac{1}{2s}\int_{-\infty}^{-R} \frac{|(\gamma+1)^2-(\gamma - w_s(x))^2|}{|x|^{2s}}  \, dx\\
&\qquad\qquad\leq C\int_{-\infty}^{-R} \frac{1+w_s(x)}{|x|^{2s}}  \, dx
\leq C\int_{-\infty}^{-R} \frac{1+u_\gamma^s(x)}{|x|^{2s}}  \, dx\\&\qquad\qquad
\leq C\int_{-\infty}^{-R} \frac{dx}{|x|^{4s}}
\leq C\int_{-\infty}^{-R} \frac{dx}{|x|^{2}} = \frac{C}{R} \leq C,
\end{align*}
where~$C$ is independent of~$s \in \left(\frac12,s_0\right]$. 

With this and Lemma~\ref{lem:h}, we have that
\[
\lim_{s \searrow \frac12} \left[  \G_s(h,[-R,0]) +2\int_{-\infty}^{-R} \int_{0}^{+\infty} \frac{(\gamma+1)^2-(\gamma - w_s(x))^2}{(y-x)^{1+2s}} \, dy \, dx \right] \leq C (1+ \ln R). 
\]
Therefore, with Fatou's Lemma, 
\[
\G_{\frac12}(w_0,[-R,0]) \leq \lim_{s \searrow \frac12} \G_s(w_s,[-R,0]) \leq C(1 + \ln R),
\]
and we have that
\[
\G_{\frac12}(w_0) = \lim_{R \to +\infty} \frac{\G_{\frac12}(w_0,[-R,0]) }{\ln R} \leq C,
\]
as desired. Hence, $w_0$ is a global minimizer of~$\G_{\frac12}$
and~\eqref{458y95476340517hgwejkqgye9876500} is thereby established. 

Finally, we check the remaining properties in the statement of Proposition~\ref{lem:existence-critical}. 
Recall~$u_\gamma^{\frac12}$ in~\eqref{eq:ugamma0-eqn} and that, by Lemma~\ref{lem:existence-limit}, the estimate~\eqref{eq:eq:ugamma-w0} holds. 
Since~$u_\gamma^{\frac12}$ is strictly decreasing, this implies that~$w_0 < \gamma$ in~$\R^-$ and also
\[
0 \leq \lim_{x \to -\infty} (w_0(x) +1) \leq \lim_{x \to -\infty} (u_\gamma^{\frac12}(x) +1) = 0. 
\]
With this and the minimization property, $w_0 \in X_\gamma$ solves~\eqref{eq:PDE-half} for~$s = \frac12$. 
The remaining properties follow from Lemma~\ref{lem:w-PDE-properties}. 
\end{proof} 

%%%%%%%%%%%%%%%%%%
\subsection{Asymptotics}
%%%%%%%%%%%%%%%%%%

We now prove the asymptotic behavior at~$-\infty$ of global minimizers. 

\begin{lemma}\label{lem:asymptotics}
Let~$s \in \left[\frac12,1\right)$ and~$\gamma \in (-1,1)$. 
The global minimizers in Propositions~\ref{lem:existence-s} and~\ref{lem:existence-critical} satisfy the estimates~\eqref{eq:watinfinity} and~\eqref{eq:w-holder} and, if~$s \in \left(\frac12,1\right)$, satisfy~\eqref{eq:watinfinity-deriv}. 
\end{lemma}

\begin{proof}
We begin by proving~\eqref{eq:watinfinity}. 
By~\eqref{eq:PAL-asymp} together with~\eqref{eq:eq:ugamma-w0} (which holds also for~$s = \frac12$,
recall Lemma~\ref{lem:existence-limit}), there exist~$\tilde{C}$, $\tilde{R} \geq 1$ such that
\[
w_0(x) +1 \leq u_\gamma(x) +1  \leq \frac{C}{|x + x_\gamma|^{2s}} \quad \hbox{ for all } x + x_\gamma \leq -\tilde{R}. 
\]
In particular, \eqref{eq:watinfinity} holds with some~$C$, $R \geq 1$ depending on~$\gamma$. 

Let~$s \in \left(\frac12,1\right)$. We prove the first estimate in~\eqref{eq:watinfinity-deriv} for~$w_0'$.
For this, let~$C$, $R$ be given by~\eqref{eq:watinfinity},
fix~$x_0< -2R$ and set~$K := |x_0|$. 
Let~$v(x) := w_0(Kx-K)$.  
One can readily check that~$v: \R \to (-1,\gamma]$ satisfies
\[
\begin{cases}
-(-\Delta)^s v(x) = K^{2s} W'(v(x)) & \hbox{for } x <1,\\
v(x) = \gamma & \hbox{for }x > 1,\\
\displaystyle 0 < v(x)+1 < \frac{C}{K^{2s}|x-1|^{2s}} & \hbox{for }x < 1 - \frac{R}{K}. 
\end{cases}
\]
By the interior regularity theory for the fractional Laplacian, see for example~\cite{ROSerra},%\cite[Theorem 1.1]{ROSerra},
\begin{equation}\label{eq:int-reg}
\|v\|_{C^{2s}(B_{1/4})} \leq C\Big(K^{2s} \| W'(v)\|_{L^{\infty}(B_{1/2})} + \|v\|_{L^{\infty}(\R)}\Big).
\end{equation}
Since~$0 < \frac{R}{K}< \frac{1}{2}$, we have that
\begin{equation}\label{eq:rhs}
\sup_{B_{1/2}} |W'(v)| \leq C \sup_{B_{1/2}} |v+1|  \leq \frac{C}{K^{2s}} \sup_{B_{1/2}} |x-1|^{-2s} \leq \frac{C}{K^{2s}}.
\end{equation}
Moreover,
\begin{equation}\label{eq:K-deriv}
\|v\|_{C^{2s}(B_{1/4})} \geq \|v'\|_{L^{\infty}(B_{1/4})} = K \|w_0'\|_{L^{\infty}(B_{K/4}(-K))} \geq K w_0'(-K).
\end{equation}
Combining~\eqref{eq:int-reg}, \eqref{eq:rhs}, and~\eqref{eq:K-deriv}, and using that~$|v| \leq 1$, up to renaming~$C$, 
\[
K w_0'(-K) \leq C.
\]
In particular, 
\begin{equation}\label{eq:woprime-final-estimate}
w_0'(x_0) \leq \frac{C}{|x_0|} \quad \hbox{ for any } x_0<-2R
\end{equation}
and therefore the first estimate in~\eqref{eq:watinfinity-deriv} holds for~$w_0'$ with~$2R$ in place of~$R$. 

We now prove~\eqref{eq:w-holder}. 
Let~$s \in \left[\frac12,1\right)$ and~$\alpha \in (0,1)$. 
Take~$x_0$, $K$, and~$v$ as in the previous paragraph.  
By interior regularity (see again~\cite{ROSerra}) and estimating as above, 
\begin{align*}
[v]_{C^{\alpha}(B_{1/4})} 
	&\leq C\Big(K^{2s} \| W'(v)\|_{L^{\infty}(B_{1/2})} + \|v\|_{L^{\infty}(\R)}\Big) \leq C.
\end{align*}
Rescaling back, we get that
\[
[w_0]_{C^{\alpha}(B_{K/4}(-K))} \leq \frac{C}{K^\alpha}. 
\]
Recalling the definition of~$K$, we obtain that
\[
[w_0]_{C^{\alpha}(B_{|x_0|/4}(x_0))} \leq \frac{C}{|x_0|^\alpha}  \leq \frac{C}{R^\alpha} \quad \hbox{ for all } x_0 < - R
\]
and the estimate in~\eqref{eq:w-holder} follows. 

By a similar argument, we can use~\eqref{eq:int-reg} to show the second estimate in~\eqref{eq:watinfinity-deriv} when~$s \in \left(\frac12,1\right)$. 
\end{proof}

%%%%%%%%%%%%%%%%%
\subsection{Uniqueness}
%%%%%%%%%%%%%%%%% 

Finally, we show that minimizers in~$X_\gamma$ are unique. 

\begin{lemma}\label{lem:w0-unique}
Let~$s \in \left[\frac12,1\right)$ and~$\gamma \in (-1,1)$. 
The global minimizer~$w_0 \in X_\gamma$ of~$\G_s$ is unique. 
\end{lemma}

\begin{proof}
Let~$w_0$ be the minimizer found in Propositions~\ref{lem:existence-s} and~\ref{lem:existence-critical}, 
and let~$w \in X_\gamma$ be another minimizer. 
Since cutting at level~$-1$ decreases the energy, we have that~$-1 \leq w \leq \gamma$ in~$\R^-$.

Set~$v := w-w_0$. Since~$w$ and~$w_0$ solve 
\[
\begin{cases}
(-\Delta)^s w + W'(w) = (-\Delta)^s w_0 + W'(w_0) = 0 & \hbox{ in } \R^-,\\
w = w_0 = \gamma & \hbox{ in }\R^+,
\end{cases}
\]
we have that
\[
\begin{cases}
(-\Delta)^s v = f & \hbox{ in } \R^-,\\
v=0 & \hbox{ in }\R^+,
\end{cases}
\]
where~$f := W'(w_0) - W'(w) \in L^{\infty}(\R^-)$ since~$|w|, |w_0| \leq 1$. 

Since~$w$, $w_0 \in X_\gamma \subset H^s(\R)$, we have that~$v \in H^s(\R)$.
Moreover, since~$|v|\leq 2$, 
\[
\int_{\R} \frac{|v(y)|}{ (1+|y|)^{1+2s}} \, dy \leq \int_{\R} \frac{2}{ (1+|y|)^{1+2s}} \, dy  <+\infty.
\]
Fix~$R>0$. By~\cite[Corollary~6.3]{RosOtonWeidner}, there exists~$C>0$, depending only on~$s$, such that
\[
[v]_{C^s(B_R^-)} 
%	\leq CR^{-s} \left(\frac{1}{R} \int_{-R}^0 |v(x)| \, dx	
%		+ R^{2s} \int_{-\infty}^{-R}\frac{|v(x)|}{|x|^{1+2s}} \, dx
%		+ \|f\|_{L^{\infty}(B_{2R}^-)}\right) 
	\leq CR^{-s}. 
\]
Sending~$R \to +\infty$, we find that~$[v]_{C^s(\R^-)}  = 0$ which implies that~$v$ is constant in~$\R^-$. Therefore, 
\[
w = w_0 +C \;\hbox{ in } \R^- \qquad{\mbox{and}}\qquad
w = w_0 = \gamma \; \hbox{ in }\R^+
\] 
for some~$C\in \R$. Since $w$, $w_0 \in C^s(\R)$, it must be that~$C=0$. Consequently, $w \equiv w_0$ and thus the minimizer is unique. 
\end{proof}

\subsection{Proof of Theorem~\ref{thm:1Dmin}}
%%%%%%%%%%%%%%%%% 
With the work done so far, we can now complete the proof of Theorem~\ref{thm:1Dmin}.

\begin{proof}[Proof of Theorem~\ref{thm:1Dmin}] 
The theorem follows from Proposition~\ref{lem:existence-s} for~$s \in \left(\frac12,1\right)$ and Proposition~\ref{lem:existence-critical} for~$s = \frac12$ together with 
Lemmata~\ref{lem:asymptotics} and~\ref{lem:w0-unique}. 
\end{proof}

%%%%%%%%%%%%%%%%%%
%\subsection{The case $\gamma = \pm 1$}
%%%%%%%%%%%%%%%%%%
%
%\begin{proof}
%If $\gamma = -1$, then the function $w_0 \equiv -1 \in X_{-1}$ is clearly a minimizer of $\G_s(\cdot, \R^-)$ in $X_{-1}$ and the statement is trivial. 
%
%\red{On the other hand, if $\gamma = +1$, then \dots}
%\end{proof}

%%%%%%%%%%%%%%%%%%%%%%%%%
\section{Bounding the energy from below for~$s \in \left[\frac12,1\right)$}\label{sec:liminf}
%%%%%%%%%%%%%%%%%%%%%%%%%%

Here, we establish the $\Gamma-\liminf$ inequality for the first-order convergence in Theorem~\ref{THM:2b}. 

%%%%%%%%%%%%%%%
\subsection{Interpolation near the boundary}
%%%%%%%%%%%%%%%

We begin by proving an adaptation of~\cite[Proposition~4.1]{SV-gamma} near~$\partial \Omega$. 
For~$\bar{x} \in \partial \Omega$ and~$\rho_o>0$, consider~$B_{\rho_o}(\bar{x}) \subset \R$.
Set
\[
D := B_{\rho_o}(\bar{x}) \cap \Omega,
\]
and, for small~$t>0$, define
\[
D_t := \big\{x \in D \;{\mbox{ s.t. }}\;d_{\partial D \setminus \partial \Omega}(x)>t\big\}
\]
where~$d_{\partial D \setminus \partial \Omega}(x)$ denotes the distance from a point~$x$ to~$\partial D \setminus \partial \Omega = \partial D \setminus \{\bar{x}\}$. 

Recall also the notation~\eqref{eq:G-general-domain}.

\begin{proposition}\label{prop:interpolation}
Fix~$\delta>0$ small and $\gamma \in (-1,1)$. Let~$\eps_k \searrow 0^+$ as~$k \to +\infty$, and let~$u_k$ be a sequence in~$L^1(\Omega)$ and~$w_k$ be a sequence in~$L^1(B_{\rho_o}(\bar{x}) \cup \Omega)$ such that
\[
\begin{cases}
u_k - w_k \to 0 \quad \hbox{as}~k \to +\infty& \hbox{ in } L^1(D \setminus D_\delta), \\
u_k= g & \hbox{ in }\Omega^c, \\
w_k = g(0) & \hbox{ in } B_{\rho_o}(\bar{x}) \setminus D, \\
|u_k| \leq 1~\hbox{and}~ w_k\leq \gamma & \hbox{in}~\R\\
u_k \leq \gamma & \hbox{in}~D.  
\end{cases}
\]

Then, there exists a sequence~$v_k\leq \gamma$ such that
\[
v_k(x) = \begin{cases}
u_k(x) & \hbox{if}~x \in D_\delta,\\
w_k(x) & \hbox{if}~x \in D^c
\end{cases}
\]
and
\[
\limsup_{k \to +\infty} \F_{\eps_k}^{(1)}(v_k,\Omega) \leq 
\limsup_{k \to +\infty} \Big[ \F_{\eps_k}^{(1)}(w_k,\Omega) - \F_{\eps_k}^{(1)}(w_k,D_\delta) + I_{\eps_k}(u_k,B_{\rho_o}(\bar{x}),\Omega) \Big]. 
\] 
\end{proposition}

\begin{proof}
For ease in the proof, we assume that
\[
\bar{x}=0, \quad \rho_o=1, \quad D= B_1^-, \quad \hbox{so that} \quad D_t = B_{1-t}^-,
\]
and take~$s \in \left(\frac12,1\right)$. 
The general setting follows along the same lines, and the proof for~$s = \frac12$ is similar (see~\cite[Proposition~4.1]{SV-gamma}). 
Furthermore, we drop the notation for the subscript~$k$. 

Also, we may assume that there exists some~$C_0>0$ such that
\begin{equation}\label{eq:wkuk-C0}
\F_{\eps}^{(1)}(w,\Omega) - \F_{\eps}^{(1)}(w,D_\delta) + I_{\eps}(u,D,\Omega) \leq C_0,
\end{equation}
otherwise there is nothing to show. Since
\begin{equation}\label{eq:Omega-Ddelta-w}
\begin{split}
w(Q_\Omega) - w(Q_{D_\delta})
	&= w(\Omega \setminus D_\delta,\Omega \setminus D_\delta)
		+ 2 w(\Omega \setminus D_\delta, \Omega^c)\\
	&\geq  w(\Omega \setminus D_\delta, (D_\delta)^c),
\end{split}
\end{equation}
we obtain from~\eqref{eq:wkuk-C0} that
\begin{equation}\label{eq:Omega-Ddelta}
w(\Omega \setminus D_\delta, (D_\delta)^c) + u(D \setminus D_\delta, D) \leq C_0 \eps^{1-2s}.
\end{equation}

The rest of the proof is broken into four steps. First, 
we will partition the set~$D \setminus D_\delta$ into a finite sequence of intervals so that one of the intervals satisfies an estimate similar to~\eqref{eq:Omega-Ddelta} with right-hand side small for all~$\varepsilon$. 
Then, we perform a second partition to find a subinterval such that the difference between~$u$ and~$w$ is sufficiently small there. 
Next, we use the first two steps to appropriately partition~$\R$ and construct the desired function~$v$. Lastly, we estimate errors. 

\smallskip

\noindent
\underline{\bf Step 1}. (Partition~$D \setminus D_\delta$). 
Fix~$\sigma>0$ small. For~$M = M(\sigma)>1$, to be determined, set
\[
\tilde{\delta} := \frac{\delta}{M}. 
\]
We partition~$D \setminus D_\delta = (-1,-1+\delta)$ into~$M$ disjoint intervals~$D_{j\tilde{\delta}} \setminus D_{(j+1)\tilde{\delta}}$, so that
\[
D \setminus D_\delta = \bigcup_{j=0}^{M-1} D_{j\tilde{\delta}} \setminus D_{(j+1)\tilde{\delta}}.
\]
{F}rom~\eqref{eq:Omega-Ddelta}, and since~$D \subset \Omega$, we have that
\[
C_0 \eps^{1-2s}
	\geq \sum_{j=0}^{M-1} \left[
w(D_{j\tilde{\delta}} \setminus D_{(j+1)\tilde{\delta}}, (D_\delta)^c) + u(D_{j\tilde{\delta}} \setminus D_{(j+1)\tilde{\delta}}, D) \right]. 
\]
For~$M = M(\sigma)$ sufficiently large, there exists some~$0 \leq j \leq M-1$ such that
\[
w(D_{j\tilde{\delta}} \setminus D_{(j+1)\tilde{\delta}}, (D_\delta)^c) + u(D_{j\tilde{\delta}} \setminus D_{(j+1)\tilde{\delta}}, D) 
	\leq \sigma \eps^{1-2s}.
\]
Fix such a~$j$ and set
\[
\widetilde{D} := D_{j\tilde{\delta}}
\]
Since~$\widetilde{D} \subset D$, we have that
\begin{equation} \label{eq:wu-tildeD}
w(\widetilde{D} \setminus \widetilde{D}_{\tilde{\delta}}, (D_\delta)^c) + u(\widetilde{D} \setminus \widetilde{D}_{\tilde{\delta}}, \widetilde{D}) 
	\leq \sigma \eps^{1-2s}.
\end{equation}

\smallskip

\noindent
\underline{\bf Step 2}. (Partition~$\widetilde{D} \setminus \widetilde{D}_{\tilde{\delta}}$). 
Let~$N \in \N$ be the integer part of~$\tilde{\delta}/(2 \eps)$ and let~$\tilde{x}$ be the left endpoint of~$\widetilde{D}$. For~$0 \leq i \leq N-1$, we define
\[
A_i := \big\{x \in \widetilde{D}~\hbox{s.t.}~i \eps < (x - \tilde{x}) \leq (i+1) \varepsilon \big\}
\]
and observe that~$ A_i\subset \widetilde{D} \setminus \widetilde{D}_{\tilde{\delta}}$.

Denote by
\begin{equation}\label{eq:di-defn}
d_i(x) := d_{\partial \widetilde{D}_{\eps i}}(x),
\end{equation}
the distance from~$x$ to the boundary of~$\widetilde{D}_{\eps i}$. 
As in the proof of~\cite[Proposition~4.1]{SV-gamma}, we can show that there exists some~$0 \leq i \leq N-1$ such that
\begin{equation}\label{eq:Ai-estimate}
\int_{A_i} |u(x)-w(x)| \, dx
	+ \eps^{2s} \int_{\widetilde{D}_{(i+1)\eps} \setminus \widetilde{D}_{\tilde{\delta}}} |u(x)-w(x)| d_i(x)^{-2s} \, dx \leq \sigma \eps.
\end{equation}
{F}rom now on such an~$i$ is fixed once and for all.

\smallskip

\noindent
\underline{\bf Step 3}. (Partition~$\R$ and construct~$v$).
We partition~$\R$ into the following six disjoint regions:
\[
\begin{array}{lll}
P := \widetilde{D}_{\tilde{\delta}}, & 
Q:= \widetilde{D}_{(i+1) \varepsilon} \setminus \widetilde{D}_{\tilde{\delta}}, & 
R := A_i,\\
S:= \Omega \setminus \widetilde{D}_{\eps i}, &
T:=  \Omega^c \setminus B_\delta^+, &
U:= B_\delta^+.
\end{array}
\]
Note that
\begin{equation}\label{eq:Omega-decomp}
\Omega = P \cup Q \cup R \cup S \qquad{\mbox{and}}\qquad  \Omega^c = T \cup U.
\end{equation}
Let~$\phi$ be a smooth cutoff function such that~$\phi = 1$ on~$P \cup Q$, $\phi = 0$ on~$S \cup T$, and~$\|\phi'\|_{L^\infty(\R)} \leq 3/\eps$. Define~$v$ by
\[
v := \begin{cases}
\phi u + (1-\phi) w &  \hbox{in}~\R^- = P\cup Q \cup R \cup S \cup T^-,\\
w & \hbox{in}~\R^+ = T^+ \cup U.
\end{cases}
\]
By construction, it is clear that $v \leq \gamma$ in $\R$, $v = u$ in $P \cup Q \supset D_{\delta}$ and $v = w$ in $(B_1^-)^c = D^c$. 
It remains to show that~$v$ satisfies the energy estimate in Proposition~\ref{prop:interpolation}.

Starting with the kinetic energy, we will show that
\begin{equation}\label{eq:v-kinetic}
v(Q_\Omega) 
	\leq w(Q_\Omega) - w(Q_{D_\delta})
		+ u(B_1^-,B_1^-) + 2u(B_1^-,B_1^+) 
		+ O(\sigma \eps^{1-2s}) + O(\tilde{\delta}^{-2s}). 
\end{equation}
Since~$S \subset \Omega \setminus D_\delta$, 
one can check (see~\cite{SV-gamma}) that
\begin{align*}
w(S,S) + 2w(S,T \cup U) + w(Q_{D_\delta})
%&=w(S,S) + 2w(S,\Omega^c) 
%	+ w(D_\delta, D_\delta) +  2w(D_\delta, \Omega \setminus D_\delta) + 2w(D_\delta, \Omega^c)\\
%&= w(S,S)  
%	+ w(\Omega \setminus D_\delta, D_\delta)
%	+ w(D_\delta, \Omega) 
%	+ 2 w(S \cup D_\delta, \Omega^c)\\
%&\leq w(\Omega \setminus D_\delta, S)
%	 +  w(\Omega \setminus D_\delta, D_\delta)
%	+ w(D_\delta, \Omega) 
%	+ 2 w(\Omega, \Omega^c)\\
%&\leq w(\Omega \setminus D_\delta, \Omega)
%	+ w(D_\delta, \Omega) 
%	+ 2 w(\Omega, \Omega^c)\\
%&= w(\Omega , \Omega)
%	+ 2 w(\Omega, \Omega^c)
\leq w(Q_\Omega).
\end{align*}
With this and noting that
\[
u(P \cup Q, P \cup Q) + 2u(P\cup Q, U) \leq u(B_1^-, B_1^-) + 2u(B_1^-, B_1^+),
\]
in order to prove~\eqref{eq:v-kinetic}, it is enough to show that
\begin{equation}\label{eq:enoughttoshow}
v(Q_\Omega)  
	\leq w(S,S) + 2w(S,T \cup U)
	+u(P \cup Q, P \cup Q) + 2u(P\cup Q, U) 
		+ O(\sigma \eps^{1-2s}) + O(\tilde{\delta}^{-2s}).
\end{equation}

{F}rom the construction of~$v$, note that
\begin{eqnarray*} &&
v(S,S) = w(S,S), \quad v(S,T\cup U) = w(S, T\cup U), 
\\ &&
{\mbox{and }} \quad
 v(P \cup Q, P \cup Q) = u(P \cup Q, P \cup Q).
\end{eqnarray*}
Recalling~\eqref{eq:Omega-decomp}, we find that
\begin{align*}
v(\Omega, \Omega)
	%&= w(S,S) + 2v(S, P \cup Q \cup R) + v(P \cup Q \cup R,P \cup Q \cup R) \\
	&= w(S,S) + 2v(S, P \cup Q \cup R) 
	+ u(P \cup Q, P \cup Q )
	+ 2 v(P \cup Q, R) + v(R,R).
\end{align*}
and
\begin{align*}
v(\Omega, \Omega^c)
%	&= w(S, T \cup U) + v(P \cup Q \cup R, T \cup U) \\
%	&= w(S, T \cup U) +  v(P \cup Q, T \cup U)
%		+ v(R, T \cup U) \\
	&= w(S, T \cup U) 
		+ v(P \cup Q, T)+  v(P \cup Q, U) + v(R, T \cup U).
\end{align*}
Therefore,
\begin{equation}\label{eq:v-Omega-final}
v(Q_\Omega)
	= [w(S,S) + 2w(S, T \cup U) ]
		+ [u(P\cup Q, P \cup Q) + 2 u(P \cup Q, U)] + E
\end{equation}
where
\begin{align*}
E&:=  2 v(S, P \cup Q \cup R) + 2v(P \cup Q, R \cup T) + v(R,R)
	+ 2v(R, T \cup U)\\
&\quad + 2 \big(v(P \cup Q, U)- u(P \cup Q, U)\big). 
\end{align*}

\smallskip

\noindent
\underline{\bf Step 4}. (Error estimates).
We will show that there exists~$C>0$ such that
\begin{equation}\label{eq:E-error}
|E| \leq C\big(\sigma \eps^{1-2s} + \tilde{\delta}^{-2s}\big).
\end{equation}

Notice that~\eqref{eq:E-error}, together with~\eqref{eq:v-Omega-final}, gives the desired result in~\eqref{eq:enoughttoshow}
(and then in~\eqref{eq:v-kinetic}). Hence, we now focus on the proof of~\eqref{eq:E-error}.

First note that
\[
Q \cup R =  \widetilde{D}_{i \eps} \setminus \widetilde{D}_{\tilde{\delta}} \subset \widetilde{D} \setminus \widetilde{D}_{\tilde{\delta}}
\]
and, since~$D_\delta \subset \widetilde{D}_{\tilde{\delta}} \subset \widetilde{D}_{(i+1)\eps}$, 
\[
R \cup S \cup T \cup U= (\widetilde{D}_{(i+1)\eps})^c \subset (\widetilde{D}_{\tilde{\delta}})^c \subset (D_\delta)^c.
\]
Therefore, \eqref{eq:wu-tildeD} implies that
\begin{equation}\label{eq:w-sigma-errors}
w(Q \cup R, R \cup S \cup T \cup U) \leq \sigma \eps^{1-2s}. 
\end{equation}
Following the proof of~\cite[Proposition~4.1]{SV-gamma}, we use~\eqref{eq:wu-tildeD}, 
\eqref{eq:Ai-estimate}, and~\eqref{eq:w-sigma-errors} to find that
\begin{equation}\label{eq:SV-errors}
v(P, S \cup T \cup R) 
	+ v(Q, S \cup T \cup R) + v(S \cup T \cup U, R) + v(R,R) 
	\leq C\big(\sigma \eps^{1-2s} +  \tilde{\delta}^{-2s}\big).
\end{equation}

We are left to show that
\[
|v(P \cup Q, U)- u(P \cup Q, U)| \leq C\tilde{\delta}^{-2s}.
\]
For this, observe that
\[
 \int_{P \cup Q \cup U} \int_{|x-y|> \frac{\tilde{\delta}}{2}} \frac{1}{|x-y|^{1-2s}}dy \, dx
 	\leq C \int_{\frac{\tilde{\delta}}{2}}^{+\infty} r^{-1-2s} \,dr \leq C \tilde{\delta}^{-2s}. 
\]
%and similarly,
%\[
% \int_{U} \int_{|x-y|> \frac{\tilde{\delta}}{2}} \frac{1}{|x-y|^{1-2s}}\,dy \, dx \leq C \tilde{\delta}^{-2s}. 
%\]
Therefore, recalling that~$w(y) = g(0)$ and~$u(y) = g(y)$ for~$y \in U$, 
\begin{align*}
|v(P \cup Q, U)- u(P \cup Q, U)|
	&\leq  \int_{-\frac{\tilde{\delta}}{2}}^0 \int_0^{\frac{\tilde{\delta}}{2}} \frac{\big|
	|u(x) - g(0)|^2 - |u(x) - g(y)|^2\big|}{|x-y|^{1+2s}} \, dy \, dx +C\tilde{\delta}^{-2s}.
\end{align*}
Since~$g$ is Lipschitz continuous and~$|u|$, $|g| \leq 1$, 
\begin{align*}
\big||u(x) - g(0)|^2 - |u(x) - g(y)|^2\big|
	%&= |2u(x) - g(0) - g(y)||g(0) - g(y)|\\
	&\leq C |g(0) - g(y)| \leq C|y|. 
\end{align*}
Therefore,
\begin{eqnarray*}
&&|v(P \cup Q, U)- u(P \cup Q, U)|
	\leq  C\int_{-\frac{\tilde{\delta}}{2}}^0 \int_0^{\frac{\tilde{\delta}}{2}} \frac{|y|}{|x-y|^{1+2s}} \, dy \, dx + C\tilde{\delta}^{-2s} \\
	&&\qquad \qquad\leq C\int_0^{\frac{\tilde{\delta}}{2}} y^{1-2s} \, dy +C \tilde{\delta}^{-2s} 
	\leq C \tilde{\delta}^{2-2s}  +C \tilde{\delta}^{-2s} \leq C \tilde{\delta}^{-2s}.
\end{eqnarray*}
Together with~\eqref{eq:SV-errors}, this proves~\eqref{eq:E-error}.

\smallskip

\noindent
\underline{\bf Conclusion}. 
Regarding the potential energy, we use~\eqref{eq:Omega-decomp} to write
\[
\int_{\Omega} W(v) \, dx
	= \int_{P \cup Q} W(u) \, dx + \int_{S} W(w) \, dx + \int_{R} W(v) \, dx. 
\]
In~$R$, observe that
\[
W(v) \leq W(w) + C|v-w| \leq W(w) + C|u-w|.
\]
Therefore, with~\eqref{eq:Ai-estimate}, 
\begin{equation*}
\begin{split}
\int_{\Omega} W(v) \, dx
	&\leq \int_{P \cup Q} W(u) \, dx + \int_{S \cup R} W(w) \, dx + C \int_{R} |u-w| \, dx\\
	&\leq \int_{D} W(u) \, dx + \int_{\Omega \setminus D_\delta} W(w) \, dx + C \sigma \eps.
\end{split}
\end{equation*}

{F}rom this and~\eqref{eq:v-kinetic}, we obtain that
\begin{align*}
\mathcal{F}_{\eps}^{(1)}(v,\Omega)
	&= \eps^{2s-1} v(Q_\Omega) + \frac{1}{\eps} \int_{\Omega} W(v) \, dx \\
	&\leq \eps^{2s-1} \Big(w(Q_\Omega) - w(Q_{D_\delta})
		+ u(B_1^-,B_1^-) + 2u(B_1^-,B_1^+) \Big) \\
	&\quad +  \frac{1}{\eps}\left(\int_{D} W(u) \, dx + \int_{\Omega} W(w) \, dx  - \int_{D_\delta} W(w) \, dx \right)+ 
	C\big( \sigma + \tilde{\delta}^{-2s} \eps^{2s-1}\big)\\
	&= \mathcal{F}_{\eps}^{(1)}(w,\Omega) - \mathcal{F}_{\eps}^{(1)}(w,D_\delta) + I_{\eps}(u, B_1, \Omega)
		 +  C\big( \sigma + \tilde{\delta}^{-2s} \eps^{2s-1}\big).
\end{align*}
Therefore,
\[
\limsup_{k \to +\infty} \mathcal{F}_{\eps}^{(1)}(v,\Omega)
	\leq \limsup_{k \to +\infty}\Big[\mathcal{F}_{\eps}^{(1)}(w,\Omega) - \mathcal{F}_{\eps}^{(1)}(w,D_\delta) + I_{\eps}(u, B_1, \Omega)\Big]  + C \sigma. 
\]
Since~$\sigma>0$ was arbitrary, the proof Proposition~\ref{prop:interpolation} is complete. 
\end{proof}

%%%%%%%%%%%%%%%%%%%%%%%%%%
\subsection{Contribution near the boundary}
%%%%%%%%%%%%%%%%%%%%%%%%%%
%We begin with the following result that will be used to compute the contribution near $\partial \Omega$.
We now establish an energy estimate near~$\partial \Omega$.

\begin{proposition}\label{lem:ue-inside}
Assume that~$|g|<1$ on~$\partial \Omega$ and let~$ \kappa \in\left(0,\frac{ |\Omega|}2\right)$.
Let~$u_\eps:\R\to[-1,1]$ be such that
\begin{eqnarray}\label{98aydet7rf65e8637t4ie}
&& u_\eps(x)=g(x) {\mbox{ for any }}x\in\R\setminus\overline\Omega,\\
&&\label{9uic567sdsd89} \lim_{\eps\searrow0}u_\eps(x)=\pm1 {\mbox{ for any }}x\in B_{\rho_o}(\bar{x}) \cap\Omega,\\
{\mbox{and }} &&\label{new-condition}
{{\mbox{either~$u_\eps(x) \geq g(\bar{x}) \geq 0$ or~$u_\eps(x) \leq g(\bar{x}) \leq 0$, for all~$x \in B_{\kappa}(\bar{x}) \cap \Omega$,}}}
\end{eqnarray}
for some~$\rho_o \in(0, \kappa]$ and some~$\bar{x} \in \partial \Omega$.

Then,
$$ \liminf_{\rho\searrow0}\,\liminf_{\eps\searrow0}
I_{\eps}(u_\eps, B_{\rho}(\bar{x}),\Omega)
\ge \Psi(\pm 1,g(\bar{x})).$$
\end{proposition}

\begin{proof} 
Without loss of generality, we assume that
\[
0 = \bar{x} \in \partial \Omega  \qquad \hbox{and} \qquad B_{\rho_o}(\bar{x}) \cap \Omega = B_{\rho_o}^-=(-\rho_o,0) 
\]
and that 
\begin{equation}\label{eq:WLOG-interior}
 \lim_{\eps\searrow0}u_\eps(x)=-1 {\mbox{ for any }}x\in B_{\rho_o}^-.
 \end{equation}
Consequently, we deduce from~\eqref{new-condition} that
\begin{equation}\label{new-condition-WLOG}
{u_\eps(x) \leq g(0) \leq 0 {\mbox{ for any }}x\in B_{\kappa}^- \supset B_{\rho_o}^-.}
\end{equation}

We then want to show that
\[
\liminf_{\rho\searrow0}\,\liminf_{\eps\searrow0} 
I_{\eps}(u_\eps, B_{\rho}^-, B_{\rho})
\ge \Psi(-1,g(0)).
\]
For this, we set
\begin{align*}
{\mathcal{J}}_{\rho,\eps}&:= I_{\eps}(u_\eps, B_{\rho}^-,B_\rho)\\
{\mathcal{J}}_\rho &:= \liminf_{\eps\searrow0}{\mathcal{J}}_{\rho,\eps}
=
\liminf_{\eps\searrow0}I_{\eps}(u_\eps, B_{\rho}^-,B_\rho)\\
{\mbox{and }}\quad {\mathcal{J}}&:=
\liminf_{\rho\searrow0} {\mathcal{J}}_\rho=
\liminf_{\rho\searrow0} \liminf_{\eps\searrow0} I_{\eps}(u_\eps, B_{\rho}^-, B_\rho).\end{align*}

We take a sequence~$\rho_i>0$ which is infinitesimal as~$i\to+\infty$ and such that
\begin{equation}\label{LIM1}
{\mathcal{J}}=
\lim_{i\to+\infty}{\mathcal{J}}_{\rho_i}
.\end{equation}
Then, for any fixed~$i\in\N$, we take a sequence~$\eps_{k,i}$
which is infinitesimal as~$k\to+\infty$ and such that
\begin{equation}\label{LIM2}
{\mathcal{J}}_{\rho_i}=\lim_{k\to+\infty}{\mathcal{J}}_{\rho_i,\eps_{k,i}}
.\end{equation}
We point out that, up to extracting a subsequence,
we can suppose that
\begin{equation}\label{i4630mnbgrfewop5834h5h769rei}
\eps_{k,i}\le\frac{\rho_i}{k}.\end{equation}

Since we need to extract subsequences in a delicate and appropriate way,
given~$m\in\N$ (to be taken large in the sequel),
we use~\eqref{LIM1} to find~$i_0(m)$ such that if~$i\ge i_0(m)$ then
$$ |{\mathcal{J}}-{\mathcal{J}}_{\rho_i}|\le \frac{e^{-m}}{2}.$$
Then, fixed~$i\in\N$, we use~\eqref{LIM2} to find~$k_0(i,m)$ such that
for any~$k\ge k_0(i,m)$ we have that
$$ |{\mathcal{J}}_{\rho_i}-{\mathcal{J}}_{\rho_i,\eps_{k,i}}|\le \frac{e^{-m}}2.$$
In particular, if~$i\ge i_0(m)$ and~$k\ge k_0(i,m)$,
\begin{equation}\label{709uiIJSJJSJS}
|{\mathcal{J}}-{\mathcal{J}}_{\rho_i,\eps_{k,i}}|\le {e^{-m}}.\end{equation}

Now, we define
\begin{equation}\label{4yihcdhhldhfs83y843} 
u_{k,i}(x):= u_{\eps_{k,i}} ({{{\rho}}}_i x) \qquad{\mbox{and}}\qquad
w_{k,i}(x):=w_{\eps_{k,i}} ({{{\rho}}}_i x)=w_0\left( \frac{{{{\rho}}}_i x}{\eps_{k,i}}\right),
\end{equation}
where~$w_0(x):= w_0(x;-1,g(0))$ is given by Theorem~\ref{thm:1Dmin} with~$\gamma := g(0)\in (-1,1)$ (recall the notation in~\eqref{wie7658ktoupkjhgf}).
In particular,
\begin{equation}\label{UAI:1}
{\mbox{if~$x>0$, \quad then~$w_{k,i}(x)=g(0)$.}}
\end{equation}

Since~$w_0(x) \to -1$ uniformly as~$x \to -\infty$, by the monotonicity of~$w_0$,
we have that
\[
 \lim_{k\to+\infty} \sup_{x\le-\rho_i} |w_{k,i}(x)+1|
 	=  \lim_{k\to+\infty} \sup_{x\le-\rho_i} w_0\left( \frac{\rho_ix}{\eps_{k,i}}\right) +1
	=  \lim_{k\to+\infty} w_0\left( -\frac{\rho_i^2}{\eps_{k,i}}\right) +1
	 =0.
\]
Hence, for any~$i$, $m \in \N$, there exists~$k_1(i,m)\in\N\cap[m,+\infty)$ such that
\begin{equation}\label{12:0}
\sup_{x\le-\rho_i} |w_{k,i}(x)+1|\le e^{-m}
\quad \hbox{for any}~k \geq k_1(i,m).
\end{equation}

Now, let~$i_1 \in \N$ be such that~$\rho_i < 1$ for all~$i \geq i_1$.
If~$x \in (0,\rho_o)$, we have that
\begin{equation}\label{eq:i1}
\rho_i x \in (0, \rho_o)  \quad \hbox{for all}~i \geq i_1, 
\end{equation}
so, by~\eqref{98aydet7rf65e8637t4ie},
%$$ \lim_{k\to+\infty} u_{k,i}(x)=\lim_{k\to+\infty} u_{\eps_k} ({{{\rho}}}_i x)=
%g({{{\rho}}}_i x).$$
$$
u_{k,i}(x)=u_{\eps_{k,i}} ({{{\rho}}}_i x)=g({{{\rho}}}_i x).
$$
%Accordingly, for any~$m\in\N$, there exists~$k_3(m)\in\N\cap[m,+\infty)$ such that,
%for any~$i\ge i_1$, any~$k\ge k_3(m)$ and any~$x\in (0,\rho_o)$,
%$$ |u_{k,i}(x)-g({{{\rho}}}_i x)|\le \frac{e^{-m}}{2}.$$
%On the other hand, 
Since~$g$ is Lipschitz continuous,
there exists~$i_2(m)\in\N\cap[m,+\infty)$ such that,
for any~$x\in  (0,\rho_o)$ and any~$i\ge i_2(m)$,
$$ |g({{{\rho}}}_i x)-g(0)|\le e^{-m}.$$
Therefore, if $i_3(m):=i_1+i_2(m)$, we have that,
for any~$i\ge i_3(m)$, any~$k\in \N$ and any~$x\in (0,\rho_o)$, 
\begin{equation}\label{12:2}
|u_{k,i}(x)-g(0)|\le e^{-m}.
\end{equation}

Now we define
\begin{equation}\label{J80:00A}
\begin{split}
&{i_m}:= i_0(m)+i_3(m)=i_0(m)+i_1+i_2(m)
\\ {\mbox{and }}\quad&{k_m}:= k_0(i_m,m)+k_1(i_m,m)%+k_2(m)+k_3(m)
.\end{split}\end{equation}
Notice that both~${k_m}$ and~${i_m}$
diverge as~$m\to+\infty$. We also set
$$ \eps_m:=\eps_{{k_m},{i_m}},\;\qquad
w_m(x):=w_{{k_m},{i_m}}(x),\;\qquad{\mbox{ and }}\;\qquad
u_m(x):=u_{{k_m},{i_m}}(x).$$
Note that~$\eps_m/\rho_{i_m}$ is infinitesimal as~$m \to+ \infty$,
thanks to~\eqref{i4630mnbgrfewop5834h5h769rei}.

{F}rom~\eqref{12:2}, we have that
that~$u_m\to g(0)$ in~$(0, \rho_o)$,
as~$m\to+\infty$.
An obvious consequence of~\eqref{UAI:1} is that
also~$w_m\to g(0)$ in~$(0,\rho_o)$,
as~$m\to+\infty$. 
{F}rom~\eqref{12:0}, we have that~$w_m\to-1$ in~$(-\rho_o,0)$, as~$m\to+\infty$. 
On the other hand, if $x \in (-\rho_o,0)$, then by \eqref{eq:i1} and \eqref{J80:00A}, 
\[
\rho_{i} x \in (-\rho_o, 0) \quad \hbox{for all } i \geq i_m, 
\]
and so, thanks to~\eqref{eq:WLOG-interior}, 
\[
\lim_{k \to +\infty} u_{k,i}(x) = \lim_{k \to +\infty} u_{\eps_{k,i}}({{{\rho}}}_i x) = -1.
\]
That is, for any~$x\in (-\rho_o,0)$ and any~$m\in\N$, there exists~$k_2(m,x)\in \N \cap[m,+\infty)$ such that,
for any~$i\ge i_1$ and any~$k\ge k_2(m,x)$,
%\begin{equation}\label{12:1}
\[
|u_{k,i}(x)+1|\le e^{-m}.
\]
%\end{equation}
In particular, taking $m$ large enough to guarantee that $k_m \geq k_2(m,x)$, it holds that
\[
|u_{m}(x) + 1| \leq e^{-m}
\]
and $u_m(x) \to -1$ as $m \to+ \infty$. 
Therefore, by collecting these pieces, we find that~$u_m-w_m\to0$ a.e.~in~$(-\rho_o, \rho_o)$, as~$m\to+\infty$.
%Thanks to~\eqref{eq:WLOG-interior},
%\[
%\lim_{m \to +\infty} u_{k_m,i_m}(x) = \lim_{m \to +\infty} u_{\eps_{k_m}}({{{\rho}}}_{i_m} x) = -1.
%\]

%
%Similarly, from~\eqref{12:1}
%we have that~$u_m\to-1$ in~$(-\rho_o,0)$, and then thanks to~\eqref{eq:WLOG-interior},
%\[
%\lim_{k \to +\infty} u_{k,i}(x) = \lim_{k \to +\infty} u_{\eps_k}({{{\rho}}}_i x) = -1.
%\]
%That is, 
%for any~$m\in\N$, there exists~$k_2(m)\in\N\cap[m,+\infty)$ such that,
%for any~$i\ge i_1$, any~$k\ge k_2(m)$ and any~$x\in (-\rho_o,0)$,
%\begin{equation}\label{12:1}
%|u_{k,i}(x)+1|\le e^{-m}.
%\end{equation}

%\red{
%Then, if~$x \in (-\rho_o,0)$, we have that
%\[
%\rho_i x \in (-\rho_o,0)\quad \hbox{for all}~i \geq i_1, 
%\]
%and so, thanks to~\eqref{eq:WLOG-interior}, 
%\[
%\lim_{k \to +\infty} u_{k,i}(x) = \lim_{k \to +\infty} u_{\eps_k}({{{\rho}}}_i x) = -1.
%\]
%That is, for any~$x\in (-\rho_o,0)$ and any~$m\in\N$, there exists~$k_2(m,x)\in\N\cap[m,+\infty)$ such that,
%for any~$i\ge i_1$ and any~$k\ge k_2(m,x)$
%\begin{equation}\label{12:1}
%|u_{k,i}(x)+1|\le e^{-m}.
%\end{equation}
%}

Without loss of generality, let us take now take~$\rho_o=1$. 
Fix~$\delta\in\left(0,\frac14\right)$ and set 
\begin{align*}
D&:=B_{1}\cap\frac{\Omega}{\rho_{{i_m}}}  = B_1^-\\ {\mbox{and }}\quad
D_\delta &:=\left\{ x\in D {\mbox{ s.t. }}d_{\partial D \setminus \partial \Omega}(x)>\delta \right\}  = B_{1-\delta}^-.% (-1+\delta, -\delta). 
\end{align*}

By \eqref{new-condition-WLOG}, we have that~$u_m \leq g(0)=\gamma$ in~$B_1^-$. 
We are now in a position to apply Proposition~\ref{prop:interpolation} 
to guarantee the existence of~$v_m \leq \gamma$ such that
\[
v_m(x) = \begin{cases}
	u_m(x) & \hbox{ if } x \in D_\delta,\\
	w_m(x) & \hbox{ if } x \in D^c,
\end{cases}
\]
and satisfying
\begin{equation}\label{098loiyscdy8ywuhjt7gy}
\limsup_{m\to+\infty} {\mathcal{F}}^{(1)}_{\tilde{\eps}_m} (v_m,\Omega)\le
\limsup_{m\to+\infty}\left[{\mathcal{F}}^{(1)}_{\tilde{\eps}_m} (w_m,\Omega)-
{\mathcal{F}}^{(1)}_{\tilde{\eps}_m} (w_m, D_\delta)
+I_{\tilde{\eps}_m}(u_m,D,\Omega)\right]
\end{equation}
where~$\tilde{\eps}_m: = \eps_m/\rho_{i_m}$. We recall that~$\tilde{\eps}_m$ is infinitesimal as~$m\to+\infty$.
%By Proposition~\ref{prop:interpolation}, 
%% of \cite[Proposition 4.1]{SV-gamma}, we have $v_m = w_m$ in all of $D^c$ (not just in $\Omega \setminus D$) and 
%$v_m \leq \gamma$ (since we can assume that~$u_m \leq \gamma$ in~$(-1,0)$) for~$m$ sufficiently large. 

Let~$s \in \left(\frac12,1\right)$. Since~$v_m \leq \gamma$ and~$v_m = w_m$ outside~$\Omega$, in particular outside~$\Omega \cap \R^-$,
we are in the position of applying~\eqref{eq:energy-difference}
(with~$A:=\Omega$ and~$B:=\R^-$),
obtaining that
$$ \F_{\tilde{\eps}_m}^{(1)}(w_m, \Omega) - \F_{\tilde{\eps}_m}^{(1)}(v_m, \Omega)
	= \F_{\tilde{\eps}_m}^{(1)}(w_m, \R^-) - \F_\eps^{(1)}(v_m, \R^-).$$
Thus, rescaling as in~\eqref{eq:g-rescale}
(with~$A:=\R^-$) and using the fact that~$w_0$ is a global minimizer of~$\G_s$, we find that
\begin{equation*}
\F_{\tilde{\eps}_m}^{(1)}(w_m, \Omega) - \F_{\tilde{\eps}_m}^{(1)}(v_m, \Omega)
	=\G_s(w_0, \R^-) - \G_s(v_m(\tilde{\eps}_m (\cdot)), \R^-) \leq 0.
\end{equation*}
Consequently with~\eqref{098loiyscdy8ywuhjt7gy} and scaling as in~\eqref{eq:rho-scaling}-\eqref{eq:g-rescale}, we have that 
\begin{align*}
0&\le
\limsup_{m\to+\infty}\left[ -{\mathcal{F}}^{(1)}_{\tilde{\eps}_m} (w_m, D_\delta)+I_{\tilde{\eps}_m}(u_m,D,\Omega)\right]\\
&=
\limsup_{m\to+\infty}\left[ -{\mathcal{F}}^{(1)}_{\tilde{\eps}_m} (w_m, D_\delta)+I_{\tilde{\eps}_m}(u_m,B_1^-, B_1)\right]\\
%	&= \limsup_{m\to+\infty}\bigg[ 
%	-{\mathcal{F}}^{(1)}_{\eps_m} (w_{\eps_m}, B_{(1-\delta)\rho_{i_m}}^-)
%	+I_{\eps_m}(u_{\eps_m},B_{\rho_{i_m}}^-)
%\bigg]\\
	&\leq \limsup_{m\to+\infty}\left[ 
	-\G_s(w_0, D_\delta/\tilde{\eps}_m)
	+I_{\eps_m}(u_{\eps_m},B_{\rho_{i_m}}^-, B_{\rho_{i_m}})
\right].
\end{align*}
Since~$\delta>0$ is arbitrary, we obtain, recalling also~\eqref{eq:G}
and~\eqref{eq:Psi}, that
\begin{equation}\label{eq:I-est}
\limsup_{m\to+\infty} I_{\eps_m}(u_{\eps_m},B_{\rho_{i_m}}^-,B_{\rho_{i_m}})
	\geq \liminf_{m \to +\infty}\G_s(w_0, B_{\tilde{\eps}_m^{-1}}^-) =
	\G_s(w_0)=\Psi(-1,g(0)). 
\end{equation}

For the case~$s = \frac{1}{2}$, using here that~$w_0$ is a local minimizer in~$B_{\tilde{\eps}_m^{-1}}^{-}$, we similarly estimate (since~$v_m = w_m$ in~$(B_1^-)^c$),
\begin{align*}&
\F_{\tilde{\eps}_m}^{(1)}(w_m, \Omega) - \F_\eps^{(1)}(v_m, \Omega)
	= \F_{\tilde{\eps}_m}^{(1)}(w_m, B_1^-) - \F_\eps^{(1)}(v_m, B_1^-)\\
	&\qquad\qquad= \frac{1}{|\ln \tilde{\eps}_m|}\left[\G_s(w_0, B_{\tilde{\eps}_m^{-1}}^{-}) - \G_s((v_m(\tilde{\eps}_m (\cdot)), B_{\tilde{\eps}_m^{-1}}^{-}) \right]\leq 0,
\end{align*}
so that
\begin{align*}
0 
&\leq \limsup_{m\to+\infty}\left[ -{\mathcal{F}}^{(1)}_{\tilde{\eps}_m} (w_m, D_\delta)+I_{\tilde{\eps}_m}(u_m,D,\Omega)\right]\\
&= \limsup_{m\to+\infty}\left[ -{\mathcal{F}}^{(1)}_{\tilde{\eps}_m} (w_m, D_\delta)+I_{\tilde{\eps}_m}(u_m,B_1^{-},B_1)\right]\\
&\leq \limsup_{m\to+\infty}\left[ 
	- \frac{1}{|\ln \tilde{\eps}_m|} \G_s(w_0, D_\delta/\tilde{\eps}_m)+I_{\eps_m}(u_{\eps_m},B_{\rho_{i_m}}^-,B_{\rho_{i_m}})
\right].
\end{align*}
Consequently, 
\begin{equation}\label{eq:I-est-criticals}
\limsup_{m\to+\infty} I_{\eps_m}(u_{\eps_m},B_{\rho_{i_m}}^-, B_{\rho_{i_m}})
	\geq \liminf_{m \to +\infty} \frac{1}{|\ln \tilde{\eps}_m|} \G_s(w_0, B_{\tilde{\eps}_m^{-1}}^-) = \Psi(-1,g(0)). 
\end{equation}

Now, for all~$s \in \left[\frac12,1\right)$, we have from~\eqref{eq:I-est} and~\eqref{eq:I-est-criticals} that 
\begin{equation} \label{sq9ydiugv2etirguwfjhkhdoie}
\begin{split}
\Psi(-1,g(0))
\leq
\limsup_{m\to+\infty} I_{\eps_m}(u_{\eps_m},B_{\rho_{i_m}}^-,B_{\rho_{i_m}})
 = \limsup_{m\to+\infty}
{\mathcal{J}}_{\rho_{i_m},\eps_{k_m,i_m}}.
\end{split}\end{equation}
Recalling~\eqref{J80:00A}, we see that~$i_m\ge i_0(m)$
and~$k_m\ge k_0(i_m,m)$, thus we can use~\eqref{709uiIJSJJSJS}
to find that~${\mathcal{J}}_{\rho_{i_m},\eps_{k_m,i_m}}\le 
{\mathcal{J}}+{e^{-m}}$.
Hence, we infer from~\eqref{sq9ydiugv2etirguwfjhkhdoie} that
$$  \Psi(-1,g(0))\le \limsup_{m\to+\infty}
(\mathcal{J}+e^{-m}) ={\mathcal{J}},$$
as desired.
\end{proof}

%%%%%%%%%%%%%%%%%%%%
\subsection{Proof of the $\liminf$ inequality}
%%%%%%%%%%%%%%%%%%%%
We recall that~$X$ is the space of all the measurable
functions~$u:\R^n\to\R$ such that the restriction of~$u$ to~$\Omega$
belongs to~$L^1(\Omega)$. Moreover, $X$ is endowed
with the metric of~$L^1(\Omega)$, as made clear in~\eqref{CON:DE}.
Also, the spaces~$X_g$ and~$Y_\kappa$ are
defined in~\eqref{Xgdefinizione} and~\eqref{eq:Ykappa}, respectively.

\begin{proposition}\label{prop:liminf}
Assume that~$|g| <1$ on~$\partial \Omega$ and let~$ \kappa \in\left(0,\frac{ |\Omega|}2\right)$. 
Let~$u_j$ be such that~$u_j\to \bar u$
in~$X$ and~$\eps_j\searrow0$ as~$j\to+\infty$.

Then,
$$ \liminf_{j\to+\infty} 
{\mathcal{F}}^{(1)}_{\eps_j}(u_j)\ge
{\mathcal{F}}^{(1)}(\bar u)$$
where~$\mathcal{F}^{(1)}$ is given by
\[
{\mathcal{F}}^{(1)}(u):= 
\begin{cases}
c_\star \per\,(E,\Omega)+\displaystyle\int_{\partial\Omega}
\Psi(u(x),g(x))\,d{\mathcal{H}}^{0}(x)
&\begin{matrix}
{\mbox{ if~$X \cap Y_\kappa\ni u=\chi_E-\chi_{E^c}$
a.e. in~$\Omega$,}} \\ {\mbox{ for some~$E \subset \R^n$,}}\end{matrix} \\+\infty & {\mbox{ otherwise.}}
\end{cases}
\]
\end{proposition}

\begin{proof}
Notice that we can assume that~$u_j \in X_g \cap Y_\kappa$,
otherwise~$\F_{\eps_j}^{(1)}(u_j)=+\infty$ (recall the setting in Section~\ref{sec:setup}) and we are done.

Also, in light of~\cite[Proposition~3.3]{SV-gamma}, we can assume that~$\bar u=\chi_E-\chi_{E^c}$ a.e.
in~$\Omega$,
with~$E$ of finite perimeter.
In particular, $\bar u$ can be defined
along~$\partial\Omega$ in the trace sense
(see e.g.~\cite{Giusti}).
Moreover, since~$u_j \in Y_\kappa$ and~$u_j \to \bar{u}$ in $X$, it must be that~$\bar{u} \in Y_\kappa$.

For clarity in the proof, we assume that~$\Omega = (\bar{x}_1, \bar{x}_2)$. 
Consider a covering
of small intervals~$\{ B_{\rho_i}(\bar{x}_1), B_{\rho_i}(\bar{x}_2)\}$
of~$\partial\Omega = \{\bar{x}_1, \bar{x}_2\}$, with~$\rho_i$ infinitesimal as~$i\to+\infty$,
and with
\begin{equation}\label{78iuh9uiokhjuoihkj}
{\mbox{either~$\bar u=-1$
or~$\bar u=1$ in~$B_{{\rho}_{i}}(\bar{x}_k)\cap\Omega$, with~$k=1,2$.}} 
\end{equation}
We can satisfy~\eqref{78iuh9uiokhjuoihkj} since~$E\subset \R$ has finite perimeter. 

Now we fix an interval~$\Omega'\Subset\Omega$.
Then, for large~$i$,
the balls~$B_{{{{\rho}}}_{i}}(x_\pm)$ lie outside~$\Omega'$, and, by inspection,
one sees that 
\begin{equation}\label{ZDENT}
\begin{aligned}
{\mathcal{F}}^{(1)}_{\eps_j}(u_j,\Omega)
	&\ge {\mathcal{F}}^{(1)}_{\eps_j}(u_j,\Omega')
		+ I_{\eps_j}(u_j, B_{{\rho}_{i}}(\bar{x}_1),\Omega)
		+ I_{\eps_j}(u_j, B_{{\rho}_{i}}(\bar{x}_2),\Omega).
\end{aligned}
\end{equation}
By~\cite[Proposition 4.5]{SV-gamma}, it holds that
\[
\liminf_{j\to+\infty} {\mathcal{F}}^{(1)}_{\eps_j}(u_j,\Omega')\ge
c_\star\per\,(E,\Omega').
\]
With this and Proposition~\ref{lem:ue-inside}, we find in~\eqref{ZDENT} that
\[
\liminf_{i \to +\infty} \liminf_{j \to +\infty} {\mathcal{F}}^{(1)}_{\eps_j}(u_j,\Omega)
	\geq c_\star\per\,(E,\Omega')
		+ \Psi(\bar u(\bar{x}_1),g(\bar{x}_1))
		+ \Psi(\bar u(\bar{x}_2),g(\bar{x}_2)).
\]
We remark that Proposition~\ref{lem:ue-inside} can be used in this
framework, since~\eqref{98aydet7rf65e8637t4ie}
and~\eqref{new-condition}
are consequences of the fact that~$u_j \in X_g\cap Y_\kappa$, 
while~\eqref{9uic567sdsd89} follows from~\eqref{78iuh9uiokhjuoihkj}.
The desired result now follows by taking~$\Omega'$
arbitrarily close to~$\Omega$ and noticing that
\begin{equation*}
\int_{\partial\Omega}
\Psi(\bar u(x), g(x))\,d{\mathcal{H}}^{0}(x)
= \Psi(\bar u(\bar{x}_1),g(\bar{x}_1)) + \Psi(\bar u(\bar{x}_2),g(\bar{x}_2)).\qedhere
\end{equation*}
\end{proof}

%%%%%%%%%%%%%%%%%%%%%%%%%%
\section{Bounding the energy from above for~$s \in \left[\frac12,1\right)$}\label{sec:limsup}
%%%%%%%%%%%%%%%%%%%%%%%%%%

In this section, we establish the $\limsup$-inequality in Theorem~\ref{THM:2b} 
by constructing a recovery sequence corresponding
to some infinitesimal sequence~$\eps_j\searrow0$. 
To this aim, we let~$\eps:=\eps_j$ for short and we take~$\rho:=\rho_j>0$
which is infinitesimal as~$j\to+\infty$ and such that,
recalling~\eqref{eq:ab-tilde}, 
\begin{equation}\label{LA:EPSRHO}
0 =  \lim_{j\to+\infty} \frac{\tilde{a}_\eps}{\rho^{2s}}  = \begin{cases}
 \displaystyle \lim_{j\to+\infty} \frac{1}{| \ln \eps| \rho} &  \hbox{ if } s = \frac{1}{2} ,\\[.75em]
   \displaystyle \lim_{j\to+\infty}  \frac{\eps^{2s-1}}{\rho^{2s}} &  \hbox{ if } s  \in \left(\frac12,1\right).
   \end{cases}
\end{equation}
The quantity~$\rho$ will play a crucial role in the detection
of the recovery sequence near the boundary of~$\Omega$.
We will use the notation~$r:=\rho/\eps$. Notice that~$r\to+\infty$ as~$j\to+\infty$, thanks to~\eqref{LA:EPSRHO}.
We also denote by~$o(1)$ quantities that are infinitesimal as~$j\to+\infty$.

Now, let~$E \subset \R$ be such that 
\begin{equation}\label{TRN}
\per(E,\Omega) < +\infty 
\end{equation}
and let $\tilde{d}$ denote the signed distance to to $\partial E$ 
%\begin{equation}\label{eq:d-A}
%\tilde{d}~\hbox{denote the signed distance to}~\partial E
%\end{equation}
with the convention that~$\tilde{d} \geq 0$ in~$E$. 
Then, by~\cite[Proposition~4.6]{SV-gamma}, for any~$\Omega' \Subset\Omega$, it holds that
\begin{equation}\label{eq:limsup-SV}
 \limsup_{\eps\searrow0} {\mathcal{F}}^{(1)}_{\eps}(u_\eps)\le
c_\star\per(E,\Omega') \qquad \hbox{where} \quad u_\eps(x):= u_0\left(\frac{\tilde d(x)}{\eps}\right)
\end{equation}
and~$c_\star>0$ depends only on~$s$ and~$W$.
Here above and in the rest of this section,
$u_0$ is the unique solution of~\eqref{eq:1DinR}.

We denote by~$d$ the signed distance from~$\partial\Omega$,
with the convention that~$d\ge0$ in~$\Omega$. 
Set~$\pi:\R\to\partial\Omega$
to be the projection along the boundary of~$\Omega$.  
In particular, we have that
\[
d(x) = \begin{cases}
|x-\pi(x)| & \hbox{ if } x \in \Omega,\\
-|x-\pi(x)|& \hbox{ if } x \in \Omega^c. 
\end{cases}
\]
Also, we set
\[
\Omega_\rho:=\{x\in\Omega {\mbox{ s.t. }} d(x)\in(0,\rho)\}. 
\]

We fix~$\rho_o>0$ sufficiently small so that~$|\Omega|\ge 10\rho_o$.
Moreover, since~\eqref{TRN} holds, up to taking~$\rho_o$ smaller,
we can suppose that
\begin{equation}\label{eq:rho0}
%|\tilde{d}(x)|>\frac{\rho_o}{10}\; \hbox{ in } \Omega_{2\rho_o}
%\qquad \hbox{and} \qquad
%|\Omega| \geq 5\rho_o.
{\mbox{$\sgn(\tilde{d}(x)) = \pm 1$ does not change sign in~$\Omega_{3\rho_o}$.}}
\end{equation}
%Consequently, $|\tilde{d}(x)| >\rho$ for $x \in \Omega_{2\rho} \setminus \Omega_\rho$ for $j$ sufficiently large. 
For clarity and ease in notation, we also define the trace of $\sgn(\tilde{d})$ in~$\overline{\Omega}$ by setting
\[
\sgn_E(x) := \begin{cases}
\sgn(\tilde{d}(x)) & \hbox{ if } x \in \Omega,\\
\displaystyle \lim_{\Omega\ni y \to x} \sgn(\tilde{d}(y))  & \hbox{ if } x \in \partial \Omega. 
\end{cases}
\]

Finally, for sufficiently large~$j$ (so that~$\rho \leq \rho_o$), we define the function~$v_\eps:\R \to [-1,1]$ by 
\begin{equation}\label{v:eps:2}
v_\eps(x):=\begin{cases}
g(x) & {\mbox{ if }}x\in\Omega^c,\\
w_{\eps}(x)% \left(-d(x); \sgn(\tilde{d}(x)), g(\pi(x)) \right)
& {\mbox{ if }} x\in\Omega_\rho,\\
\frac{2\rho-d(x)}{\rho}\left[
w_{\eps}(x)% \left(-d(x); \sgn(\tilde{d}(x)), g(\pi(x)) \right)
-u_\eps(x)
\right]+u_\eps(x)
& {\mbox{ if }} x\in\Omega_{2\rho}\setminus\Omega_\rho,\\
u_\eps(x) & {\mbox{ if }}x\in\Omega\setminus \Omega_{2\rho}.
\end{cases}
\end{equation}
where~$u_\eps$ is as given by~\eqref{eq:limsup-SV} and 
\[
w_\eps(x) :=  w_0\left(-\frac{d(x)}{\eps}; 
\sgn_E(x), g(\pi(x)) \right).
 \]
Recall the notation in~\eqref{wie7658ktoupkjhgf} for~$w_0$.
 
We will show that~$v_\eps$ in~\eqref{v:eps:2} is the so-called recovery sequence. % to be used in the proof of $\Gamma$-convergence. 

Let~$\Omega': = \Omega \setminus \Omega_{2 \rho_o} \Subset\Omega$ be fixed. 
Moreover, for clarity, assume that
\begin{equation}\label{r438743876543rfhdsjfgefgzxcvb}
\Omega = (\bar{x}_1,\bar{x}_2)
\end{equation} and note that~$B_{\rho}(\bar{x}_1) \cup B_\rho(\bar{x}_2)$ is a finite disjoint covering of~$\partial \Omega$ for~$\rho \leq \rho_o$.  
We also have for~$j$ sufficiently large that
%~$|\bar{x}_1 - \bar{x}_2| \geq 5 \rho_0$, 
%thanks to~\eqref{eq:rho0}, and
\begin{align*}
\Omega_\rho
	&= (B_{\rho}(\bar{x}_1) \cap \Omega)\cup (B_{\rho}(\bar{x}_2) \cap \Omega)
	= B_\rho^+(\bar{x}_1) \cup B_\rho^-(\bar{x}_2)\\
	{\mbox{and }}\quad 
\Omega_{2\rho} 
	&= (B_{2\rho}(\bar{x}_1) \cap \Omega)\cup (B_{2\rho}(\bar{x}_2) \cap \Omega)
	= B_{2\rho}^+(\bar{x}_1) \cup B_{2\rho}^-(\bar{x}_2),
\end{align*}
where
\[
B_\delta^+(\bar{x}) := \{x \in B_\delta(\bar{x}) \hbox{ s.t. } 
x>\bar{x}\} \qquad \hbox{and} \qquad 
B_\delta^-(\bar{x}) := \{x \in B_\delta(\bar{x}) \hbox{ s.t. } x<\bar{x}\}.
\]
Recalling the definition of the set~$Y_\kappa$ in~\eqref{eq:Ykappa}, we assume that~$\tilde{d}$ (or equivalently the set~$E$) satisfies, for~$i=1,2$,
\[
\begin{array}{ll}
\hbox{if}~g(\bar{x}_i)> 0, & \hbox{then}~\tilde{d}(x)>0~\hbox{for all}~x \in B_{\kappa}(\bar{x}_i) \cap \Omega ; \\
\hbox{if}~g(\bar{x}_i)< 0, & \hbox{then}~\tilde{d}(x)<0~\hbox{for all}~x \in B_{\kappa}(\bar{x}_i) \cap \Omega;\\
\hbox{if}~g(\bar{x}_i) = 0, & \hbox{then either}~
\tilde{d}(x)>0~\hbox{or}~\tilde{d}(x)<0~\hbox{for all}~x \in B_{\kappa}(\bar{x}_i) \cap \Omega;
 \end{array}
\]
Consequently, $v_\eps \in X_g \cap Y_\kappa$. 

We now estimate the kinetic and potential energies for~$\mathcal{F}_\eps^{(1)}(v_\eps)$ separately. 

%%%%%%%%%%%%%%%%%
\subsection{Kinetic energy estimates}
%%%%%%%%%%%%%%%%%

We first estimate the kinetic energy for~$v_\eps$ as 
\begin{eqnarray}&&
v_\eps(\Omega,\Omega) + 2v_\eps(\Omega,  \Omega^c)  \label{eq:veps-decomposition}\\
	&&\qquad\leq v_\eps(\Omega', \Omega')  + 2v_\eps(\Omega', (\Omega')^c) \label{eq:Omega'}\\
	&&\qquad\quad + v_\eps(B_\rho^+(\bar{x}_1), B_\rho^+(\bar{x}_1)) 
		+ 2v_\eps(B_\rho^+(\bar{x}_1), B_\rho^-(\bar{x}_1)) \label{eq:bdry1}\\
	&&\qquad\quad + v_\eps( B_\rho^-(\bar{x}_2), B_\rho^-(\bar{x}_2)) 
		+ 2v_\eps(B_\rho^-(\bar{x}_2), \Omega_\rho^+(\bar{x}_2)) \label{eq:bdry2}\\
	&&\qquad\quad + 2v_\eps(B_\rho^+(\bar{x}_1), (\Omega^c \cup \Omega_\rho) \setminus B_\rho(\bar{x}_1))
	+2v_\eps(B_\rho^-(\bar{x}_2), (\Omega^c \cup \Omega_\rho) \setminus B_\rho(\bar{x}_2))
\label{eq:err2}\\
	& &\qquad\quad + 2v_\eps(\Omega_{2\rho_o} \setminus \Omega_\rho, \Omega_{2\rho_o} \cup \Omega^c).
	\label{eq:err1}
\end{eqnarray}
Under the scaling~\eqref{eq:ab-tilde}, we will use~\eqref{eq:limsup-SV} (with the corresponding potential energy in~$\Omega'$) to estimate the quantity in line~\eqref{eq:Omega'}. 
We then show that the boundary energy arises in-part from lines~\eqref{eq:bdry1} and~\eqref{eq:bdry2}, and that the remaining terms~\eqref{eq:err2} and~\eqref{eq:err1} vanish as~$j \to +\infty$.

First, we estimate from above the kinetic energy contributions of~$v_\eps$
near the boundary of~$\Omega$, coming from inside~$\Omega$.
To this end, for~$\gamma \in (-1,1)$, we define
\begin{equation}\label{KAH098YAaA:AAItghjklS}
 \Psi_1^r(\pm1,\gamma):= 
b_r\iint_{B_r^{-}\times B_r^{-}}
\frac{|w_0(x;\pm1, \gamma)-w_0(y;\pm1, \gamma)|^2
}{|x-y|^{1+2s} }\,dx\,dy 
\end{equation}
where, recalling~\eqref{eq:initial-scaling}, we set
\begin{equation}\label{BR2436490000}
b_r = \begin{cases}
\frac{1}{|\ln r|} & \hbox{ if } s = \frac12,\\
1 & \hbox{ if } s \in \left(\frac12,1\right). 
\end{cases}
\end{equation}
This quantity is well-defined for a fixed~$r>0$ and, since~$w_0$ is a global minimizer of~$\G_s$, recalling the notation in~\eqref{eq:Psi},
$$ \limsup_{r\to+\infty} 
\Psi_1^r(\pm1,\gamma)\leq \Psi(\pm1,\gamma) <+\infty.$$

\begin{lemma}\label{LE:CON:1}
If~$\bar x\in\partial\Omega$, then
\[
\tilde{a}_\eps v_\eps(
B_\rho(\bar x)\cap\Omega,\,
B_\rho(\bar x)\cap\Omega)\leq \Psi_1^r\big(\sgn_E(\bar{x}),\,g(\bar x)\big). 
\]
\end{lemma}

\begin{proof} 
Without loss of generality, assume that 
\begin{equation}\label{eq:normalize} 
\bar{x}=0  \qquad \hbox{and} \qquad B_{\rho_o}(\bar{x}) \cap \Omega = B_{\rho_o}^-=(-\rho_o,0). 
\end{equation}
Consequently,
\begin{equation}\label{eq:normalize2} 
\pi(x) = 0 \qquad \hbox{and} \qquad
d(x) = -x 
\quad \hbox{for } x \in B_{\rho_o}^-. 
\end{equation}
For ease, set 
\begin{equation}\label{eq:w0-ease}
w_0(x) := w_0(x; \sgn_E(0), g(0)).
\end{equation}
With this and recalling~\eqref{eq:rho0} and~\eqref{v:eps:2}, we use the change of variables~$\tilde{x}=x/\eps$ and~$\tilde{y}=y/\eps$ to find that
\begin{align*}
\tilde{a}_\eps v_\eps(
B_\rho\cap\Omega,\,
B_\rho\cap\Omega)
	&=
\tilde{a}_\eps \int_{-\rho}^0 \int_{-\rho}^0
\frac{\left|
w_0\left( \frac{x}{\eps}\right)-
w_0\left( \frac{y}{\eps}\right)\right|^2
}{|x-y|^{1+2s}}\,dy \,dx\\
&= \eps^{1-2s}\tilde{a}_\eps
\int_{-\frac{\rho}{\eps}}^0 \int_{-\frac{\rho}{\eps}}^0
\frac{\left|
w_0(\tilde{x})-w_0(\tilde{y}) \right|^2
}{|\tilde{x}-\tilde{y}|^{1+2s}}\,d\tilde{y}\, d\tilde{x}.
\end{align*}

Recall that~$r= \rho/\eps$.
Thus, when~$s \in \left(\frac12,1\right)$, we have that~$\eps^{1-2s}\tilde{a}_\eps=1$ and the lemma holds with equality. 

If instead~$s =\frac12$,
recall that~$\tilde{a}_\eps = 1/|\ln \eps|$. Notice that, for~$j$ large,
\begin{equation}\label{eq:lnr}
\frac{1}{|\ln \eps|} =\frac{|\ln r|}{|\ln \eps|} \frac{1}{|\ln r|} 
= \frac{ |\ln \eps| - | \ln \rho|}{|\ln \eps|} \frac{1}{|\ln r|} 
\leq \frac{1}{|\ln r|}.
\end{equation}
Therefore, 
\begin{align*}
\frac{1}{|\ln \eps|} v_\eps(B_\rho\cap\Omega,\,B_\rho\cap\Omega)
& \leq  \frac{1}{|\ln r|}\int_{-r}^0 \int_{-r}^0
\frac{\left|w_0(\tilde{x})-w_0(\tilde{y}) \right|^2
}{|\tilde{x}-\tilde{y}|^{2}}\,d\tilde{y}\,d\tilde{x},
\end{align*}
as desired.
\end{proof}

Now we address the counterpart of Lemma~\ref{LE:CON:1},
in which
we estimate from above the mixed
energy contributions of~$v_\eps$
near the boundary of~$\Omega$, coming from the interactions
between~$\Omega$ and its complement. 
To this aim,
we set
\begin{equation}\label{9iK789aa}
\Psi_2^r(\pm1,\gamma):=
b_r\iint_{B_r^{-}\times B_r^{+}}
\frac{\left|w_0(x;\pm1, \gamma)-\gamma\right|^2
}{|x-y|^{1+2s}}\,dx\, dy.
\end{equation}
The reader may compare~\eqref{KAH098YAaA:AAItghjklS}
and~\eqref{9iK789aa} to appreciate the different interactions
coming from inside the domain and the ones coming from
the inside/outside relations.
Note that 
$$ \limsup_{r\to+\infty}\Psi_2^r(\pm1,\gamma)\leq \Psi(\pm1,\gamma) <+\infty.$$

\begin{lemma}\label{LE:CON:2}
If~$\bar x\in\partial\Omega$, 
then
\[
\tilde{a}_\eps v_\eps(B_\rho(\bar x)\cap\Omega,\,
B_\rho(\bar x)\setminus\Omega)\leq \Psi_2^r\big(\sgn_E(\bar{x}),g(\bar x)\big) + o(1). 
\]
\end{lemma}

\begin{proof} 
As in the proof of Lemma~\ref{LE:CON:1},
assume, without loss of generality, that~\eqref{eq:normalize},  \eqref{eq:normalize2}, and~\eqref{eq:w0-ease} hold. 
We recall~\eqref{v:eps:2} to find that
\begin{equation*}
v_\eps(x)-v_\eps(y)
=w_0\left(\frac{x}{\eps}\right) - g(y) \quad \hbox{for~$x \in  B_\rho^-$ and $y \in B_\rho^+$.}
\end{equation*}
Consequently,
\begin{align*}
|v_\eps(x)-v_\eps(y)|
	&\leq  \left|w_0\left( \frac{x}{\eps}\right) -g(0)\right| + |g(0) - g(y)| \\
	&\leq  \left|w_0\left( \frac{x}{\eps}\right) -g(0)\right| +C \min \{1, |x-y|\}.
\end{align*}

Now, we use the following elementary inequality: if~$\alpha$, $\beta\ge0$,
then
\[%\begin{equation}\label{9iHKSHSKHSHSHKSHKSHKHSAAA}
2\alpha\beta = 2\,\frac{\alpha}{\sqrt{|\ln\rho|}}\cdot
({\sqrt{|\ln\rho|}}\,\beta)
\le \frac{\alpha^2}{{|\ln\rho|}}+ |\ln\rho|\,\beta^2,
\]%\end{equation}
so that
\begin{equation}\label{9iHKSHSKHSHSHKSHKSHKHSAAA}
(\alpha + \beta)^2 = \alpha^2 +2\alpha\beta + \beta^2 \leq \left( 1 + \frac{1}{|\ln \rho|}\right) \alpha^2 + \big(1+|\ln \rho|\big) \beta^2.
\end{equation}
Hence, 
we find that
\begin{eqnarray*}&&
|v_\eps(x)-v_\eps(y)|^2
\leq\left(1+\frac{1}{|\ln\rho|}\right)\,
\left|w_0\left( \frac{x}{\eps}\right) -g(0)\right|^2
+C\,|\ln\rho|\,\min\{ 1,\,|x-y|^2\},
\end{eqnarray*}
for~$j$ sufficiently large and for some~$C>0$.

This gives that
\begin{align*}
\tilde{a}_\eps v_\eps(B_\rho^-,\,B_\rho^+)
&\leq \left(1+\frac{1}{|\ln\rho|}\right)\,\tilde{a}_\eps
\int_{-\rho}^0 \int_0^\rho
\frac{\left|w_0\left( \frac{x}{\eps}\right)  -g(0)\right|^2
}{|x-y|^{1+2s}}\,dy\,dx\\
&\qquad+C\,\tilde{a}_\eps\,|\ln\rho|
\int_{-\rho}^0 \int_0^\rho
\frac{ \min\{1,\,|x-y|^2\}
}{|x-y|^{1+2s}}\,dy\,dx\\
&\le \left(1+\frac{1}{|\ln\rho|}\right)\,\tilde{a}_\eps
\int_{-\rho}^0 \int_0^\rho
\frac{\left|
w_0\left( \frac{x}{\eps}\right)  -g(0)\right|^2
}{|x-y|^{1+2s}}\,dy\,dx
	+ C\,\tilde{a}_\eps\,|\ln\rho|\cdot\rho\\
 &=\left(1+\frac{1}{|\ln\rho|}\right)\,\eps^{1-2s}\tilde{a}_\eps
\int_{-\frac{\rho}{\eps}}^0 \int_0^{\frac{\rho}{\eps}}
\frac{\left|
w_0\left(\tilde{x}\right)  -g(0)\right|^2
}{|\tilde{x}-\tilde{y}|^{1+2s}}\,d\tilde{y}\,d\tilde{x}
	+C\,\tilde{a}_\eps\,|\ln\rho|\cdot\rho 
\end{align*}
where we use the substitutions~$\tilde{x}=x/\eps$ and~$\tilde{y}=y/\eps$.
 
For~$s \in \left(\frac12,1\right)$, we have that
\begin{align*}
\eps^{2s-1} v_\eps(
B_\rho\cap\Omega,\,
B_\rho\setminus\Omega)
 &\leq \left(1+\frac{1}{|\ln\rho|}\right)
\int_{-r}^0 \int_0^{r}
\frac{\left|
w_0\left(\tilde{x}\right)  -g(0)\right|^2
}{|\tilde{x}-\tilde{y}|^{1+2s}}\,d\tilde{y}\,d\tilde{x}
	+C\,|\ln\rho|\cdot\rho^2\\
&\leq \int_{-r}^0 \int_0^{r}
\frac{\left|
w_0\left(\tilde{x}\right) -g(0)\right|^2
}{|\tilde{x} - \tilde{y}|^{1+2s}}\,d\tilde y\, d\tilde x+o(1),
\end{align*}
as desired.
 
For~$s = \frac12$, recalling~\eqref{eq:lnr}, we see that
\begin{align*}
\frac{1}{|\ln \eps|}& v_\eps(B_\rho\cap\Omega,\,
B_\rho\setminus\Omega)\\
&\leq \left(1+\frac{1}{|\ln\rho|}\right) \frac{1}{|\ln \eps|}
\int_{-r}^0 \int_0^{r}
\frac{\left|w_0\left(\tilde{x}\right)  -g(0)\right|^2
}{|\tilde{x}-\tilde{y}|^{2}}\,d\tilde{y}\,d\tilde{x}
+ \frac{C}{|\ln \eps|}\,|\ln\rho|\cdot\rho \\
&= \left(1+\frac{1}{|\ln\rho|}\right)\frac{1}{|\ln r|}\int_{-r}^0 \int_0^{r}
\frac{\left|w_0\left(\tilde{x}\right) -g(0)\right|^2
}{|\tilde{x} - \tilde{y}|^{2}}\,d\tilde y\, d\tilde x  + \frac{C}{|\ln \eps|}\,|\ln\rho|\cdot\rho\\
&\leq \frac{1}{|\ln r|}\int_{-r}^0 \int_0^{r}
\frac{\left|w_0\left(\tilde{x}\right) -g(0)\right|^2
}{|\tilde{x} - \tilde{y}|^{2}}\,d\tilde y \,d\tilde x   + o(1). 
\end{align*}
This completes the proof. 
\end{proof}

It remains to estimate the error terms in lines~\eqref{eq:err2} and~\eqref{eq:err1}.
For this, we first prove the following regularity estimate away from~$\partial \Omega$. 

\begin{lemma}\label{lem:Holder-Lip-forErrors}
Let~$\bar{x} \in \partial \Omega$. 
If~$x$, $y \in (B_{2\rho_o}(\bar{x}) \cap \Omega) \setminus (B_{\frac{\rho}{10}}(\bar{x})\cap \Omega) $, then, for any~$\alpha\in(0,1)$, there exists a constant~$C>0$ such that
\[
|v_\eps(x) - v_\eps(y)| \leq C \left(\frac{|x-y|^\alpha}{\rho^\alpha} + \frac{|x-y|}{\rho}\right).
\]

Moreover, if~$x$, $y \in (B_{2\rho_o}(\bar{x}) \cap \Omega) \setminus (B_{2\rho}(\bar{x})\cap \Omega)$, there exists~$C>0$ such that
\[
|v_\eps(x) - v_\eps(y)| \leq C \frac{|x-y|}{ {\rho}}.
\]
\end{lemma}

\begin{remark}\label{rem:alpha=1}
{\rm{If~$s \in \left(\frac{1}{2},1\right)$ then, thanks to~\eqref{eq:watinfinity-deriv}, one can check that Lemma~\ref{lem:Holder-Lip-forErrors} holds for~$\alpha =1$. }}
\end{remark}

\begin{proof}[Proof of Lemma~\ref{lem:Holder-Lip-forErrors}]
Recall the notation in~\eqref{r438743876543rfhdsjfgefgzxcvb} and notice that, without loss of generality, we can assume that~$\bar{x} = \bar{x}_2 = 0$. 

In this setting, we write that~$d(x) =- x$ for all~$x\in(-2\rho_o,0)$.
In light of~\eqref{eq:rho0}, we also have that~$\sgn_E(x)=\sgn_E(0)$
for all~$x\in(-2\rho_o,0)$.

Assume also that~$\sgn_E(x) =-1$ and set~$w_0(x) := w_0(x;-1,g(0))$. 

We now break the proof into cases based on
the definition of~$v_\eps(x)$ and~$v_\eps(y)$,
according to~\eqref{v:eps:2}.
\smallskip

\noindent
\underline{\bf Case 1}. Suppose that~$-\rho < x < y < -\frac{\rho}{10}$. Note that~$x$, $y \in \Omega_\rho$, therefore, 
$$  |v_\eps(x) - v_\eps(y)|=|w_\eps(x) - w_\eps(y)|
=\left|w_0 \left( \frac{x}{\eps}\right)-w_0 \left( \frac{y}{\eps}\right) \right| .$$
Thus, using also~\eqref{eq:w-holder}, we have that
\begin{equation*}
|v_\eps(x) - v_\eps(y)|
	\leq [w_0]_{C^\alpha((-\frac{\rho}{\eps}, -\frac{\rho}{10\eps}))} \left| \frac{x-y}{\eps} \right|^\alpha 
	\leq C \left| \frac{x-y}{\rho} \right|^\alpha.
\end{equation*}

\smallskip

\noindent
\underline{\bf Case 2}. Suppose that~$-2\rho < x < y < -\rho$. 
Since~$x$, $y \in \Omega_{2\rho} \setminus \Omega_\rho$, we write
\begin{align*}&
|v_\eps(x) - v_\eps(y)|\\
	&= \left| \left( \frac{2\rho+x}{\rho}\left( w_\eps(x) - u_\eps(x) \right) + u_\eps(x) \right)
		-\left( \frac{2\rho+y}{\rho}\left( w_\eps(y) - u_\eps(y) \right) + u_\eps(y) \right) \right|\\
	&\leq  \left| \frac{2\rho +x}{\rho} \right| \left| w_\eps(x) - w_\eps(y)\right|
		+ \left | \frac{\rho+x}{\rho} \right| \left| u_\eps(x) - u_\eps(y)\right|
		+ \left| \frac{y-x}{\rho} \right| \left| w_\eps(y) - u_\eps(y) \right| \\
	&\leq \left| w_\eps(x) - w_\eps(y)\right|
		+ \left| u_\eps(x) - u_\eps(y)\right|
		+ 2\left| \frac{y-x}{\rho} \right|.	
\end{align*}
We use~\eqref{eq:w-holder} to estimate
\[
 \left| w_\eps(x) - w_\eps(y)\right| 
	\leq \left| w_0 \left(\frac{x}{\eps} \right) - w_0 \left(\frac{y}{\eps} \right)\right|
	\leq [w_0]_{C^\alpha((-\frac{2\rho}{\eps}, -\frac{\rho}{\eps}))} \left| \frac{x-y}{\eps} \right|^\alpha 
	\leq C \left| \frac{x-y}{\rho} \right|^\alpha.
\]
Also, as a consequence of~\eqref{eq:rho0}, we have that~$|\tilde{d}(z)|>\rho$ for~$z \in (-2\rho,-\rho)$ for~$j$ sufficiently large. 
Hence, by~\eqref{eq:PAL-asymp}, we have that
\[
| u_\eps'(z)| = \left| u_0' \left( \frac{\tilde{d}(z)}{\eps} \right)\right| \left| \frac{ \tilde{d}'(z)}{\eps} \right|
	\leq \frac{C\eps^{1+2s}}{|\tilde{d}(z)|^{1+2s}} \frac{1}{\eps} \leq \frac{C \eps^{2s}}{\rho^{1+2s}}. 
\]
With this, we obtain that
\begin{equation}\label{eq:ueps-nabla}
\left| u_\eps(x) - u_\eps(y)\right|
	\leq \| u_\eps'\|_{L^\infty((-2\rho,-\rho))} |x-y| 
	\leq \frac{C \eps^{2s}}{\rho^{1+2s}} |x-y| \leq C \frac{|x-y|}{\rho}. 
\end{equation}
Therefore, 
\[
|v_\eps(x) - v_\eps(y)|
	\leq C \left| \frac{x-y}{\rho} \right|^\alpha + C \frac{|x-y|}{\rho} + 2 \frac{|x-y|}{\rho}. 
\]

\smallskip

\noindent
\underline{\bf Case 3}. Suppose that~$-2\rho_o < x < y < -2\rho$. 
Since~$x$, $y \in \Omega_{2\rho_o} \setminus \Omega_{2\rho}$, we recall~\eqref{eq:rho0} and estimate as in~\eqref{eq:ueps-nabla} to find that
\[
|v_\eps(x) - v_\eps(y)|
	= \left| u_\eps(x) - u_\eps(y)\right|
	\leq \| u_\eps'\|_{L^\infty((-2\rho_o,-2\rho))} |x-y| 
	\leq \frac{C \eps^{2s}}{\rho^{1+2s}}|x-y|
	\leq C \frac{|x-y|}{\rho}. 
\]
%Since~$\rho < \rho_o$, we also have the worse estimate~$|v_\eps(x) - v_\eps(y)| \leq C|x-y|/\rho$. 

\smallskip

\noindent
\underline{\bf Case 4}. Suppose that~$-2\rho < x < -\rho < y  < -\frac{\rho}{10}$. 
In this case, $x \in \Omega _{2\rho} \setminus \Omega_{\rho}$ and~$y \in \Omega_\rho$, so we write
\begin{align*}
|v_\eps(x) - v_\eps(y)|
	&= \left| \left( \frac{2\rho+x}{\rho}\left( w_\eps(x) - u_\eps(x) \right) + u_\eps(x) \right)
		- w_\eps(y) \right|\\
	&\leq \left| \frac{2\rho+x}{\rho}\right| \left| w_0\left( \frac{x}{\eps}\right) -w_0\left( \frac{y}{\eps}\right)  \right|
		 + \left| \frac{\rho+x}{\rho} \right| \left| w_0 \left( \frac{y}{\eps}\right) - u_\eps(x) \right|.
\end{align*}
Recalling~\eqref{eq:w-holder}
and using that~$|\rho+x| =-\rho-x<y-x= |x-y|$, we find that
\begin{align*}
|v_\eps(x) - v_\eps(y)|
	&\leq [w_0]_{C^\alpha((-\frac{2\rho}{\eps}, -\frac{\rho}{10\eps}))} \left| \frac{x-y}{\eps} \right|^\alpha 
		 +2 \frac{|\rho+x|}{\rho}
	\leq C \left| \frac{x-y}{\rho} \right|^\alpha + 2 \frac{|x-y|}{\rho}.
\end{align*}

\smallskip

\noindent
\underline{\bf Case 5}. Suppose that~$-2\rho_o < x < -2\rho < y  < -\rho$.
This is the case in which~$x \in \Omega_{2\rho_o}\setminus \Omega_{2 \rho}$ and~$y \in \Omega_{2\rho}\setminus \Omega_\rho$. 
Estimating as in~\eqref{eq:ueps-nabla} and using that~$|2\rho+y| 
<|x-y|$, we obtain that
\begin{align*}
|v_\eps(x) - v_\eps(y)|
	&= \left| u_\eps(x) - \left( \frac{2\rho+y}{\rho}\left( w_\eps(y) - u_\eps(y) \right) + u_\eps(y) \right)\right|\\
	&\leq |u_\eps(x) - u_\eps(y)| + \left| \frac{2\rho+y}{\rho}\right| \left|  w_\eps(y) - u_\eps(y)\right|\\
	&\leq \| u_\eps'\|_{L^\infty((-2\rho_o,-\rho))} |x-y| 
		 + 2 \frac{|x-y|}{\rho}\\
	&\leq C \frac{|x-y|}{\rho}
		 + 2 \frac{|x-y|}{\rho} \\&\leq C \frac{|x-y|}{\rho}. 
\end{align*}

\smallskip

\noindent
\underline{\bf Case 6}. Suppose that~$-2\rho_o < x < -2\rho$ and~$-\rho < y  < -\frac{\rho}{10}$.
This case is trivial since~$|v_\eps| \leq 1$:
\[
|v_\eps(x) - v_\eps(y)|\leq 2 \leq 2 \frac{|x-y|}{\rho}. 
\]

\smallskip

We have exhausted all cases and the proof of Lemma~\ref{lem:Holder-Lip-forErrors}
is complete. 
\end{proof}

\begin{lemma}\label{lem:errors}
 It holds that
\begin{equation}\label{eq:err2-ok}
\tilde{a}_\eps \Big[v_\eps(B_\rho^-(\bar{x}_1), (\Omega^c \cup \Omega_\rho) \setminus B_\rho(\bar{x}_1))
+v_\eps(B_\rho^-(\bar{x}_2), (\Omega^c \cup \Omega_\rho) \setminus B_\rho(\bar{x}_2))\Big]
= o(1)
\end{equation}
and
\begin{equation}\label{eq:err1-ok}
\tilde{a}_\eps v_\eps(\Omega_{2\rho_o} \setminus \Omega_\rho, \Omega_{2\rho_o} \cup \Omega^c) = o(1).
\end{equation}
\end{lemma}

\begin{proof}
For any~$\delta>0$, we use that~$v_\eps \in [-1,1]$ to note that
\begin{equation}\label{mnbvcjuyhtre98765476345zxcvbnmasdfgh}
 \int_{|x-y|>\delta} \frac{|v_\eps(x) - v_\eps(y)|^2}{|x-y|^{1+2s}} \, dy
 	\leq C \int_\delta^{+\infty} r^{-1-2s} \, dr = C \delta^{-2s}. 
\end{equation}
With this, we first estimate
\begin{eqnarray*}&&
\tilde{a}_\eps v_\eps(B_\rho^-(\bar{x}_2), (\Omega^c \cup \Omega_\rho) \setminus B_\rho(\bar{x}_2))
	\leq \tilde{a}_\eps \int_{B_\rho^-(\bar{x}_2)} \int_{|x-y|>\rho} \frac{|v_\eps(x) - v_\eps(y)|^2}{|x-y|^{1+2s}} \, dy \, dx\\
	&&\qquad\qquad\leq C\frac{\tilde{a}_\eps}{\rho^{2s}} |B_\rho^-(\bar{x}_2)|=
	 C\frac{\tilde{a}_\eps}{\rho^{2s-1}} = o(1)
\end{eqnarray*}
and similarly find that~$v_\eps(B_\rho^+(\bar{x}_1), (\Omega^c \cup \Omega_\rho) \setminus B_\rho(\bar{x}_1)) = o(1)$. 
This proves~\eqref{eq:err2-ok}. 

Regarding~\eqref{eq:err1-ok}, we first estimate the long-range interactions in the same way to find that
\begin{eqnarray*}&&
\tilde{a}_\eps v_\eps\big(B_{2\rho_o}^-(\bar{x}_2) \setminus B_{\rho}^-(\bar{x}_2) ,  B_{2\rho_o}^+(\bar{x}_1)\cup \Omega^c\big) 
\leq \tilde{a}_\eps \int^{x_2-\rho}_{x_2-2\rho_o} \int_{|x-y|>\rho} \frac{|v_\eps(x) - v_\eps(y)|^2}{|x-y|^{1+2s}} \, dy dx \\
&&\qquad\qquad\leq C\frac{\tilde{a}_\eps}{\rho^{2s}} (2\rho_0-\rho)= o(1). 
\end{eqnarray*}
and similarly find that~$\tilde{a}_\eps 
v_\eps(B_{2\rho_o}^+(\bar{x}_1) \setminus B_{\rho}^+(\bar{x}_1),
B_{2\rho_o}^-(\bar{x}_2)\cup\Omega^c) = o(1)$. 

It is left to show that
\[
\tilde{a}_\eps \Big[ v_\eps( B_{2\rho_o}^-(\bar{x}_2), B_{2\rho_o}^-(\bar{x}_2) \setminus B_{\rho}^-(\bar{x}_2))
	+ v_\eps( B_{2\rho_o}^+(\bar{x}_1), B_{2\rho_o}^+(\bar{x}_1) \setminus B_{\rho}^+(\bar{x}_1))
\Big] =o(1).
\]
We only prove the estimate around~$\bar{x}_2$ since the estimate near~$\bar{x}_1$ is similar. 
For this, we assume, without loss of generality, that~$\bar{x}_2 = 0$, so that
\[
B_{2\rho_o}^-(\bar{x}_2) = (-2\rho_o,0) \qquad \hbox{and} \qquad 
B_{2\rho_o}^-(\bar{x}_2) \setminus B_{\rho}^-(\bar{x}_2) = (-2\rho_o,-\rho).
\]
We thus want to prove that
\begin{equation}\label{wanttoprofe4735want}
\tilde{a}_\eps v_\eps((-2\rho_o,0),(-2\rho_o,-\rho))=o(1).
\end{equation}
For this, we split the remaining error around~$\bar{x}_2 = 0$ as
\begin{equation}\label{245asdfghjkpoiuy09876reutyi4i8rWTRQuir}
\begin{split}
v_\eps((-2\rho_o,0),(-2\rho_o,-\rho))
	&= v_\eps( (-\rho/10,0), (-2\rho_o,-\rho))\\
	&\quad +v_\eps((-2\rho_o,-\rho/10),(-2\rho,-\rho)) \\
	&\quad + v_\eps((-2\rho,-\rho/10),(-2\rho_o,-2\rho)) \\
	&\quad +  v_\eps((-2\rho_o,-2\rho),(-2\rho_o,-2\rho)).
\end{split}\end{equation}
Using again~\eqref{mnbvcjuyhtre98765476345zxcvbnmasdfgh}, we see that
\begin{align*}
\tilde{a}_\eps v_\eps( (-\rho/10,0), (-2\rho_o,-\rho))
	&\leq \tilde{a}_\eps \int^{-\rho}_{-2\rho_o} \int_{|x-y|>\frac{9\rho}{10}} \frac{|v_\eps(x) - v_\eps(y)|^2}{|x-y|^{1+2s}} \, dy \,dx = o(1).
\end{align*}

In the following, we let
\[
\alpha = \alpha(s) := \begin{cases}
\frac{9}{10} & \hbox{ if } s = \frac{1}{2}, \\
1 & \hbox{ if } s  \in \left(\frac{1}{2},1\right). 
\end{cases}
\]
We use Lemma~\ref{lem:Holder-Lip-forErrors} (with Remark~\ref{rem:alpha=1}) to estimate
\begin{eqnarray*}&&
\tilde{a}_\eps v_\eps((-2\rho_o,-\rho/10),(-2\rho,-\rho))
\\&\leq& \tilde{a}_\eps \left[ \int^{-\rho}_{-2\rho} \int_{ \{|x-y|<\rho\} \cap \{y<-\rho/10\}} \frac{|v_\eps(x) - v_\eps(y)|^2}{|x-y|^{1+2s}} \, dy \, dx
\right.\\&&\qquad\qquad\left.+  \int^{-\rho}_{-2\rho} \int_{ \{|x-y|>\rho\}} \frac{|v_\eps(x) - v_\eps(y)|^2}{|x-y|^{1+2s}} \, dy \, dx\right] \\
	&\leq& C\tilde{a}_\eps \left[ \int^{-\rho}_{-2\rho} \int_{ \{|x-y|<\rho\} }\frac{\rho^{-2\alpha}|x-y|^{2\alpha}}{|x-y|^{1+2s}} \, dy \, dx
	+\int^{-\rho}_{-2\rho} \int_{ \{|x-y|<\rho\} }\frac{\rho^{-2}|x-y|^{2}}{|x-y|^{1+2s}} \, dy \, dx
	\right.
\\&&\qquad\qquad\left.
	+ \int^{-\rho}_{-2\rho} \int_{ \{|x-y|>\rho\}}  \frac{1}{|x-y|^{1+2s}} \, dy \, dx\right] \\
	& \leq& C\tilde{a}_\eps \rho\left[\frac{\rho^{2(\alpha-s)}}{\rho^{2\alpha}}+\frac{\rho^{2(1-s)}}{\rho^2}
	 + \frac1{\rho^{2s}} \right] 
	 = \frac{C \tilde{a}_\eps}{\rho^{2s-1}} = o(1).  
\end{eqnarray*}
Similarly, 
\begin{eqnarray*}
\tilde{a}_\eps  v_\eps((-2\rho,-\rho/10),(-2\rho_o,-2\rho))
%%\\&\leq& \tilde{a}_\eps  \left[ \int^{-\rho/10}_{-2\rho} \int_{ \{|x-y|<\rho\} \cap\{y<-2\rho\}} \frac{|v_\eps(x) - v_\eps(y)|^2}{|x-y|^{1+2s}} \, dy \, dx
%%+ \int^{-\rho/10}_{-2\rho} \int_{ \{|x-y|>\rho\}} \frac{|v_\eps(x) - v_\eps(y)|^2}{|x-y|^{1+2s}} \, dy \, dx\right] \\
\leq \frac{C\tilde{a}_\eps}{ \rho^{2s-1}} = o(1).
\end{eqnarray*}
Lastly, by Lemma~\ref{lem:Holder-Lip-forErrors}, we similarly estimate to find that
\begin{eqnarray*}&&
\tilde{a}_\eps v_\eps((-2\rho_o,-2\rho),(-2\rho_o,-2\rho))
	\\&\leq& C \tilde{a}_\eps \left[ \int^{-2\rho}_{-2\rho_o} \int_{ \{|x-y|<\rho\} \cap(-2\rho_o,-2\rho)} \frac{\rho^{-2} |x-y|^2}{|x-y|^{1+2s}} \, dy \, dx\right.\\
&&\qquad\qquad\left.+ \int^{-2\rho}_{-2\rho_o} \int_{ \{|x-y|>\rho\}\cap(-2\rho_o,-2\rho)} \frac{1}{|x-y|^{1+2s}} \, dy \, dx\right] \\
	 &\leq& C (\rho_o-\rho)\frac{\tilde{a}_\eps}{\rho^{2s}} = o(1). 
\end{eqnarray*} 
%\begin{align*}
%\tilde{a}_\eps v_\eps((-2\rho_o,-2\rho),(-2\rho_o,-2\rho))
%	&\leq C \tilde{a}_\eps \bigg[ \int^{-2\rho}_{-2\rho_o} \int_{ \{|x-y|<\delta\} \cap(-2\rho_o,-2\rho)} \frac{\rho_o^{-2} |x-y|^2}{|x-y|^{1+2s}} \, dy \, dx\\
%	 &\qquad+ \int^{-2\rho}_{-2\rho_o} \int_{ \{|x-y|>\delta\}\cap(-2\rho_o,-2\rho)} \frac{1}{|x-y|^{1+2s}} \, dy \, dx\bigg] \\
%	 &\leq C\tilde{a}_\eps = o(1).
%\end{align*} 

Putting together these pieces, we obtain from~\eqref{245asdfghjkpoiuy09876reutyi4i8rWTRQuir} that~$\tilde{a}_\eps v_\eps((-2\rho_o,0),(-2\rho_o,-\rho)) = o(1)$,
which is the desired claim in~\eqref{wanttoprofe4735want}.
The proof of Lemma~\ref{lem:errors} is thereby complete. 
\end{proof}

Recalling~\eqref{eq:veps-decomposition}, we combine Lemmata~\ref{LE:CON:1}, \ref{LE:CON:2}, and~\ref{lem:errors}  to obtain following estimate for the kinetic part of the energy:

\begin{lemma}\label{lem:sup-main}
It holds that
\begin{align*}
\tilde{a}_\eps [v_\eps(\Omega, \Omega) + v_\eps(\Omega, \Omega^c)]
	&\leq \tilde{a}_\eps [v_\eps(\Omega',\Omega') +2v_\eps(\Omega', (\Omega')^c)]\\
	&\quad + \Psi_1^r(\sgn_E(\bar{x}_1),g(\bar{x}_1)) + 2 \Psi_2^r(\sgn_E(\bar{x}_1),g(\bar{x}_1)) \\
	&\quad+ \Psi_1^r(\sgn_E(\bar{x}_2),g(\bar{x}_2)) + 2 \Psi_2^r(\sgn_E(\bar{x}_2),g(\bar{x}_2))
	 + o(1). 
\end{align*}
\end{lemma}

%%%%%%%%%%%%%%%%%
\subsection{Potential energy estimates}
%%%%%%%%%%%%%%%%%

We now estimate the potential energy.
For this, recall that~$\Omega' = \Omega \setminus \Omega_{2\rho_o}$
and write
\begin{align*}
\int_{\Omega} W(v_\eps(x)) \, dx
	&= \int_{\Omega'} W(v_\eps(x)) \, dx
		+\int_{ \Omega_{2\rho_o} \setminus \Omega_\rho} W(v_\eps(x)) \, dx\\
	&\qquad\quad+\int_{B_\rho^+(\bar{x}_1)} W(v_\eps(x)) \, dx+\int_{B_\rho^-(\bar{x}_2)} W(v_\eps(x)) \, dx.
\end{align*}

Near the boundary~$\partial \Omega$, we have the following. 

\begin{lemma}\label{LE:CON:3}
If~$\bar{x} \in \partial\Omega$, then
\[
\tilde{b}_\eps \int_{B_\rho(\bar{x})\cap \Omega} W(v_\eps(x)) \, dx 
	\leq b_r
	\int_{B_r^{-}} W(w_0(x; \sgn_E(\bar{x}), g(\bar{x})) \, dx,
\]
with~$b_r$ given by~\eqref{BR2436490000}.
\end{lemma}

\begin{proof}
As in the proof of Lemma~\ref{LE:CON:1}, assume without loss of generality that~\eqref{eq:normalize}, \eqref{eq:normalize2}, and~\eqref{eq:w0-ease} hold. 
Then, the result follows with the change of variables~$\tilde{x} = x/\eps$:
\begin{eqnarray*}&&
\tilde{b}_\eps \int_{B_\rho\cap \Omega} W(v_\eps(x)) \, dx 
= \tilde{b}_\eps \int_{-\rho}^0 W\left( w_0\left(\frac{x}{\eps}   \right) \right) \, dx\\
&&\qquad\qquad= \eps \tilde{b}_\eps \int_{-\frac{\rho}{\eps}}^0 W\left( w_0(\tilde{x}) \right) \, d\tilde{x} 
	=   \eps \tilde{b}_\eps \int_{-r}^0 W\left( w_0(\tilde{x}) \right) \, d\tilde{x}.
\end{eqnarray*}
When~$s \in \left(\frac12,1\right)$, note that~$\eps \tilde{b}_\eps =1= b_r$, so the lemma holds with equality. 
If~$s = \frac12$, then~$\eps\tilde{b}_\eps = 1/|\ln \eps|$ and, using~\eqref{eq:lnr}, the statement holds with an inequality. 
\end{proof}

We now check the error terms. 

\begin{lemma}
It holds that
\[
\tilde{b}_\eps \int_{ \Omega_{2\rho_o} \setminus \Omega_\rho} W(v_\eps(x)) \, dx = o(1). 
\]
\end{lemma}

\begin{proof}
We split~$\Omega_{2\rho_o} \setminus \Omega_\rho
%\Omega \setminus (\Omega' \cup \Omega_{\rho}) 
= (\Omega_{2\rho_o} \setminus \Omega_{2\rho}) \cup (\Omega_{2\rho} \setminus \Omega_\rho)$.
First let~$x \in \Omega_{2\rho_o} \setminus \Omega_{2\rho}$. 
Recalling~\eqref{eq:rho0}, we use that~$W(\pm 1) = 0$  and~\eqref{eq:PAL-asymp} to find that
\begin{align*}&
W(v_\eps(x))
	= W\left( u_0 \left( \frac{\tilde{d}(x)}{\eps} \right) \right) - W\left( \sgn\left( \frac{\tilde{d}(x)}{\eps} \right) \right)
	\\&\qquad\quad\leq C \left| u_0 \left( \frac{\tilde{d}(x)}{\eps} \right) - \sgn\left( \frac{\tilde{d}(x)}{\eps} \right) \right|
\leq C \frac{\eps^{2s}}{\rho_o^{2s}} = C\eps^{2s}.
\end{align*}
Therefore, 
\begin{equation}\label{eq:potentialerror1}
\tilde{b}_\eps \int_{\Omega_{2\rho_o} \setminus \Omega_{2\rho}} W(v_\eps(x)) \, dx
	\leq \tilde{b}_\eps \int_{\Omega_{2\rho_o} \setminus \Omega_{2\rho}} C \eps^{2s} \, dx
	\leq C\tilde{b}_\eps \eps^{2s} = o(1). 
\end{equation}

Now let~$x \in \Omega_{2 \rho} \setminus \Omega_\rho$. 
As a consequence of~\eqref{eq:rho0}, we have that~$|\tilde{d}(x)|>\rho$. 
Using the regularity of~$W$, \eqref{eq:watinfinity},
and~\eqref{eq:PAL-asymp}, we estimate
\begin{align*}
W&(v_\eps(x))
	= W\left( \frac{2\rho-d(x)}{\rho}\,\left[
w_0\left( \frac{-d(x)}{\eps}  \right)
-u_\eps(x)
\right]+u_\eps(x)\right)\\
	&\leq W(u_\eps(x)) + C \left| \frac{2\rho-d(x)}{\rho}\,\left[
w_0\left( \frac{-d(x)}{\eps} \right)
-u_\eps(x)
\right]\right|\\
&\leq %C \eps^{2s} +
C\left| u_0\left( \frac{\tilde{d}(x)}{\eps} \right) -\sgn \left(\frac{\tilde{d}(x)}{\eps} \right)\right|+ 
C \left|
w_0\left( \frac{-d(x)}{\eps} \right) - \sgn \left(\frac{\tilde{d}(x)}{\eps} \right)\right|\\
&\leq %C \eps^{2s} +
 %C \frac{\eps^{2s}}{\rho^{2s}} +  C \frac{\eps^{2s}}{\rho^{2s}} \leq
  \frac{C \eps^{2s}}{\rho^{2s}}, 
\end{align*}
where~$w_0(x) = w_0(x; \sgn(\tilde{d}(x)), g(\pi(x)))$.
Therefore,
\begin{equation}\label{eq:potentialerror2}
\tilde{b}_\eps\int_{ \Omega_{2 \rho} \setminus \Omega_\rho}W(v_\eps(x)) \, dx 
	\leq \tilde{b}_\eps \int_{ \Omega_{2 \rho} \setminus \Omega_\rho} \frac{C\eps^{2s}}{\rho^{2s}}  \, dx
	= \frac{ C\tilde{b}_\eps \eps^{2s}}{\rho^{2s-1}} 
	= o(1). 
\end{equation}
The lemma follows from~\eqref{eq:potentialerror1} and~\eqref{eq:potentialerror2}. 
\end{proof}

Combing the previous two lemmata, we have the following: 

\begin{lemma}\label{lem:potential-main}
It holds that
\begin{align*}
\tilde{b}_\eps \int_{\Omega} W(v_\eps(x)) \, dx
&=\tilde{b}_\eps  \int_{\Omega'} W(u_\eps(x)) \, dx
+ b_r\int_{B_r^{-}} W(w_0(x; \sgn_E(\bar{x}_1), g(\bar{x}_1)) \, dx\\
&\qquad\quad+b_r\int_{B_r^{-}} W(w_0(x; \sgn_E(\bar{x}_2), g(\bar{x}_2)) \, dx
+o(1),
\end{align*}
with~$b_r$ given by~\eqref{BR2436490000}.
\end{lemma}

%%%%%%%%%%%%%%%%%%%%%%%%%%
\subsection{Proof of $\limsup$ inequality}
%%%%%%%%%%%%%%%%%%%%%%%%%%

We are finally prepared to prove the $\limsup$ inequality in Theorem~\ref{THM:2b}. 
Let~$\Psi(\pm1,\gamma)$ be as in~\eqref{eq:Psi}. Recalling~\eqref{KAH098YAaA:AAItghjklS} and~\eqref{9iK789aa}, notice that
\begin{equation}\label{eq:Psi-limit}
\Psi(\pm1,\gamma) = 
\lim_{j \to+ \infty} \left[ \Psi_1^r(\pm1,\gamma) + 2\Psi_2^r(\pm1,\gamma)  + b_r\int_{B_r^-} W(w_0(x; \pm1, \gamma)) \, dx \right].
\end{equation}

We recall that~$X$ is the space of all the measurable
functions~$u:\R^n\to\R$ such that the restriction of~$u$ to~$\Omega$
belongs to~$L^1(\Omega)$. Moreover, $X$ is endowed
with the metric of~$L^1(\Omega)$, as made clear in~\eqref{CON:DE}.
Also, the space~$Y_\kappa$ is
defined in~\eqref{eq:Ykappa}.

\begin{proposition}\label{prop:limsup}
Assume that~$|g| <1$ on~$\partial \Omega$. Let~$ \kappa \in\left[0,\frac{ |\Omega|}2\right)$ and~$E \subset \R$
be a measurable set.

Then, for any sequence~$\eps_j\searrow0$, there exists a sequence~$v_j\in X$
such that~$v_j\to\chi_E-\chi_{E^c} =: u$ in~$X$ and
$$ \limsup_{j\to+\infty} {\mathcal{F}}^{(1)}_{\eps_j}(v_j)\le
c_\star\,\per(E,\Omega)
+ \int_{\partial\Omega} \Psi(u(x),g(x)) \, d \mathcal{H}^0(x).$$
\end{proposition}

We stress that Proposition~\ref{prop:limsup} holds also for~$\kappa = 0$.

\begin{proof}[Proof of Proposition~\ref{prop:limsup}]
We assume that~\eqref{TRN} holds and~$u \in X \cap Y_\kappa$, otherwise we are done. 

Let~$\Omega' := \Omega \setminus \Omega_{2\rho_o}$ and~$v_j := v_{\eps_j}$, where~$v_{\eps_j}$ is as in~\eqref{v:eps:2}.
Observe that~$v_j \to \chi_E - \chi_{E^c} = u$ in~$X$. 
In the trace sense, $u(\bar{x}_i) =\sgn_E(\bar{x}_i)$ for~$i=1,2$.

By Lemmata~\ref{lem:sup-main} and~\ref{lem:potential-main}, it holds that
\begin{align*}
{\mathcal{F}}^{(1)}_{\eps_j}(v_j)
	&\leq  \tilde{a}_\eps [v_\eps(\Omega',\Omega') +2v_\eps(\Omega', (\Omega')^c)] +\tilde{b}_\eps \int_{\Omega'} W(v_\eps(x)) \, dx\\
	&\quad + \Psi_1^r(u(\bar{x}_1),g(\bar{x}_1)) 
		+ 2 \Psi_2^r(u(\bar{x}_1),g(\bar{x}_1)) 
		+b_r\int_{B_r^-} W(w_0(x;u(\bar{x}_1),g(\bar{x}_1))) \, dx\\
	&\quad+ \Psi_1^r(u(\bar{x}_2),g(\bar{x}_2)) 
		+ 2 \Psi_2^r(u(\bar{x}_2),g(\bar{x}_2))
		+b_r\int_{B_r^-} W(w_0(x;u(\bar{x}_2),g(\bar{x}_2))) \, dx\\
	&\quad + o(1).
\end{align*}
Now, by~\eqref{eq:limsup-SV}, %\cite[Proposition 4.6]{SV-gamma} (see \eqref{eq:limsup-SV}), 
\[
\limsup_{j \to +\infty} \left[ \tilde{a}_\eps \big[v_\eps(\Omega',\Omega') +2v_\eps(\Omega', (\Omega')^c)\big] + \tilde{b}_\eps \int_{\Omega'} W(v_\eps(x)) \, dx\right]
 \leq c_{\star} \per(E,\Omega').
\]
Therefore, with~\eqref{eq:Psi-limit}, we have that
\begin{align*}
\limsup_{j \to +\infty} {\mathcal{F}}^{(1)}_{\eps_j}(v_j)
	&\leq c_{\star} \per(E,\Omega')
		+ \Psi(u(\bar{x}_1), g(\bar{x}_1))+\Psi(u(\bar{x}_2), g(\bar{x}_2)).
\end{align*}
Since~$\Omega' = \Omega \setminus \Omega_{2\rho_o}$, 
we send~$\rho_o \to 0$ to obtain the desired result. 
\end{proof}

%%%%%%%%%%%%%%%%%%%%%%%%%%
\subsection{Proof of Theorem~\ref{THM:2b}}
%%%%%%%%%%%%%%%%%%%%%%%%%%

We can now complete the proof of Theorem~\ref{THM:2b}.

\begin{proof}[Proof of Theorem~\ref{THM:2b}]
By Propositions~\ref{prop:liminf} and~\ref{prop:limsup}, we have that
$\displaystyle \F^{(1)} = \Gamma- \lim_{\varepsilon \searrow 0} \F^{(1)}_{\eps}$. 
Together with Lemma~\ref{L:0} and Remark~\ref{rem:zero}, this proves
Theorem~\ref{THM:2b}.  
\end{proof}

\begin{appendix}
%%%%%%%%%%%%%%%%%%%%%%%%%%
\section{Some auxiliary results}\label{sec:appendix}
%%%%%%%%%%%%%%%%%%%%%%%%%%

Here, we collect some auxiliary results that are used in Section~\ref{sec:connections}.

First, we show that minimizers~$u \in X_\gamma$ of~$\G_s$ in~\eqref{eq:G} must be strictly less than~$\gamma$ in a left-sided neighborhood of the origin. 

\begin{lemma}\label{lem:gamma-plateau}
Let~$s \in \left(\frac12,1\right)$ and~$\gamma \in (-1,1)$. 
Suppose that~$u \in X_\gamma$ is such that~$-1 \leq u \leq \gamma$. 
Assume that~$u = \gamma$ in~$(a,0)$ for some~$a < 0$ and define
\[
w(x) := \begin{cases}
\gamma & \hbox{if}~x>0 \\
u(x+a) & \hbox{if}~x \leq 0. 
\end{cases}
\]

Then, $\G_s(w) < \G_s(u)$. 
\end{lemma}

\begin{proof}
By a change of variables and using that~$u \equiv \gamma$ in~$(a,+\infty)$, we have that
\begin{align*}
w(\R^-, \R^-) 
	&= \int_{-\infty}^0 \int_{-\infty}^0 \frac{|u(x+a) - u(y+a)|^2}{|x-y|^{1+2s}} \, dy \, dx \\
	&= \int_{-\infty}^{a} \int_{-\infty}^{a} \frac{|u(\bar x) - u(\bar y)|^2}{|\bar x-\bar y|^{1+2s}} \, d\bar y \, d\bar x \\
	&= u((-\infty,a), (-\infty, a)) \\
	&= u(\R^-, \R^-) - 2 u((-\infty,a), (a,0)) -u((a,0),(a,0))\\
	&= u(\R^-, \R^-) - 2 u((-\infty,a), (a,0))
\end{align*}
and
\begin{align*}
w(\R^-, \R^+)
	&= \int_{-\infty}^0 \int_{0}^{+\infty} \frac{|u(x+a) - \gamma|^2}{|x-y|^{1+2s}} \, dy \, dx \\
	&= \int_{-\infty}^{a} \int_{a}^{+\infty} \frac{|u(\bar x) - \gamma|^2}{|\bar x-\bar y|^{1+2s}} \, d\bar y \, d\bar x \\
	&= \int_{-\infty}^{0} \int_{0}^{+\infty} \frac{|u(\bar x) - \gamma|^2}{|\bar x-\bar y|^{1+2s}} \, d\bar y \, d\bar x 
	 +  \int_{-\infty}^{a} \int_{a}^0 \frac{|u(\bar x) - u(\bar{y})|^2}{|x-y|^{1+2s}} \, d\bar y \, d\bar x \\
	&=  u(\R^-,\R^+) + u((-\infty, a), (a,0)).
\end{align*}
Therefore, the kinetic energy for~$w$ satisfies
\[
w(\R^-, \R^-)  + 2 w(\R^-, \R^+)
	= u(\R^-, \R^-)+2u(\R^-,\R^+). 
\]
On the other hand, since~$\gamma \in (-1,1)$, we have that~$W(\gamma)>0$. In particular, $|a|W(\gamma)>0$, so that
\begin{eqnarray*}&&
\int_{-\infty}^0W(w(x)) \, dx
	=\int_{-\infty}^0W(u(x+a)) \, dx
	=\int_{-\infty}^a W(u(\bar{x}))  \, d \bar{x} \\
	&&\qquad\qquad< \int_{-\infty}^a W(u(\bar{x}))  \, d \bar{x}  + |a| W(\gamma) 
	= \int_{-\infty}^0W(u(x)) \, dx.
\end{eqnarray*}
The result follows from the previous two displays. 
\end{proof}

Next, we present a useful competitor for energy estimates in Section~\ref{sec:connections}. 

\begin{lemma}\label{lem:h}
Let~$s \in \left(\frac12,1\right)$ and~$\G_s$ as in~\eqref{eq:G}. 
Define the function 
\[
h(x) := \begin{cases}
-1 & \hbox{ if } x \leq -2,\\
(\gamma+1) x + 2\gamma +1 & \hbox{ if } x \in (-2,-1),\\
\gamma & \hbox{ if } x \geq -1. 
\end{cases}
\]

Then, there exists a constant~$C>0$ such that, for all~$s \in \left(\frac12,1\right)$,
\begin{equation}\label{eq:h-delta}
\G_s(h) \leq C\left(1 + \frac{1}{2s-1}\right).
\end{equation}
Moreover, there exists a constant~$C>0$ such that, for all~$R>3$,
\begin{equation}\label{eq:h-R}
\lim_{s \searrow \frac12} \G_s(h,[-R,0]) \leq C(1 + \ln R). 
\end{equation}
\end{lemma}

\begin{proof}
To start, we calculate some integrals that will be used to prove both~\eqref{eq:h-delta} and~\eqref{eq:h-R}. 
In the following, $C$ is an arbitrary constant, independent of both~$R$ and~$s$. 

Observe that
\begin{equation}\label{eq:hcomputation1}
\begin{aligned}
&\int_{-2}^{-1} \int_{-1}^{0} \frac{|h(x) - h(y)|^2}{|x-y|^{1+2s}} \, dy \, dx
	%&= \int_{-2}^{-1} \int_{-1}^{0} \frac{|h(x) - \gamma|^2}{(y-x)^{2+2\delta}} \, dy \, dx\\
	=  \int_{-2}^{-1} \int_{-1}^{0} \frac{(\gamma+1)^2|x+1|^2}{(y-x)^{1+2s}} \, dy \, dx\\
	&\qquad\quad=  (\gamma+1)^2\int_{-2}^{-1}|x+1|^2 \int_{-1}^{0} \frac{1}{(y-x)^{1+2s}} \, dy \, dx\\
%	&= \frac{(\gamma+1)^2}{1+2\delta}\int_{-2}^{-1}|x+1|^2 \frac{-1}{(y-x)^{1+2\delta}}  \bigg|_{-1}^0 \, dx\\
	&\qquad\quad= \frac{(\gamma+1)^2}{2s}\int_{-2}^{-1}|x+1|^2 \left[\frac{1}{|x+1|^{2s}} -\frac{1}{|x|^{2s}} \right] \, dx\\
	&\qquad\quad\leq \frac{(\gamma+1)^2}{2s}\int_{-2}^{-1}|x+1|^{2-2s} \, dx \leq C
\end{aligned}
\end{equation}
and
\begin{equation}\label{eq:hcomputation2}
\begin{aligned}
&\int_{-2}^{-1} \int_{-2}^{-1} \frac{|h(x) - h(y)|^2}{|x-y|^{1+2s}} \, dy \, dx
	= \int_{-2}^{-1} \int_{-2}^{-1} \frac{(\gamma+1)^2|x-y|^2}{|x-y|^{1+2s}} \, dy \, dx\\
	&\qquad\qquad=	(\gamma+1)^2 \int_{-2}^{-1}\int_{-2}^{-1} |x-y|^{1-2s} \, dy \, dx  \leq C. 
%	&=  \frac{(\gamma+1)^2}{(1-2\delta)} \int_{-2}^{-1}\left[ -(y-x)^{1-2\delta} \big|_{-2}^{y}+(x-y)^{1-2\delta}\big|_{y}^{-1}  \right] dy \\
%	&=  \frac{(\gamma+1)^2}{(1-2\delta)} \int_{-2}^{-1}\left[ |x+2|^{1-2\delta} +|x+1|^{1-2\delta} \right] dx \leq C. 
\end{aligned}
\end{equation}
Next, we find
\begin{equation}\label{eq:hcomputation3}
\begin{aligned}
&\int_{-\infty}^{-2} \int_{-2}^{-1} \frac{|-1 - h(y)|^2}{(y-x)^{1+2s}} \, dy \, dx
	= \int_{-\infty}^{-2} \int_{-2}^{-1} \frac{(\gamma+1)^2|y+2|^2}{(y-x)^{1+2s}} \, dy \, dx\\
	&\qquad\qquad= (\gamma+1)^2 \int_{-2}^{-1}|y+2|^2  \left[\int_{-\infty}^{-2} \frac{dx}{(y-x)^{1+2s}}\right] \, dy \\
	&\qquad\qquad= \frac{(\gamma+1)^2}{2s} \int_{-2}^{-1} |y+2|^{2-2s} \, dy \leq C. 
\end{aligned}
\end{equation}
Finally, we estimate
\begin{equation}\label{eq:hcomputation4}
\begin{aligned}
&\int_{-2}^{0} \int_{0}^{+\infty} \frac{|h(x) - \gamma|^2}{|x-y|^{1+2s}} \, dy \, dx
	= \int_{-2}^{-1} \int_{0}^{+\infty} \frac{(\gamma+1)^2|x+1|^2}{(y-x)^{1+2s}} \, dy \, dx\\
	&\qquad\qquad= (\gamma+1)^2\int_{-2}^{-1}|x+1|^2 \int_{0}^{+\infty} \frac{1}{(y-x)^{1+2s}} \, dy \, dx\\
	&\qquad\qquad= \frac{(\gamma+1)^2}{2s}\int_{-2}^{-1} \frac{|x+1|^2}{|x|^{2s}}  \, dx 
	\leq C.
\end{aligned}
\end{equation}

We now prove~\eqref{eq:h-delta}. We begin with the interactions in~$\R^- \times \R^-$. Using~\eqref{eq:hcomputation1},
\eqref{eq:hcomputation2}, and~\eqref{eq:hcomputation3}, we find that
\begin{align*}
h(\R^-,\R^-)
%\int_{-\infty}^0 \int_{-\infty}^0 &\frac{|h(x) - h(y)|^2}{|x-y|^{2+2\delta}} \, dy \, dx\\
	&= 2\int_{-\infty}^{-2} \int_{-2}^{0} \frac{|-1 - h(y)|^2}{(y-x)^{1+2s}} \, dy \, dx
		+ \int_{-2}^0 \int_{-2}^0 \frac{|h(x) - h(y)|^2}{|x-y|^{1+2s}} \, dy \, dx\\
	&= 2\int_{-\infty}^{-2} \int_{-2}^{-1} \frac{|-1 - h(y)|^2}{(y-x)^{1+2s}} \, dy \, dx
		+ 2\int_{-\infty}^{-2} \int_{-1}^{0} \frac{(1+\gamma)^2}{(y-x)^{1+2s}} \, dy \, dx\\
	&\quad + \int_{-2}^{-1} \int_{-2}^{-1} \frac{|h(x) - h(y)|^2}{|x-y|^{1+2s}} \, dy \, dx
		+ 2 \int_{-2}^{-1} \int_{-1}^0 \frac{|h(x) - \gamma|^2}{(y-x)^{1+2s}} \, dy \, dx\\
	&=  2\int_{-\infty}^{-2} \int_{-1}^{0} \frac{(1+\gamma)^2}{(y-x)^{1+2s}} \, dy \, dx +C. 
\end{align*}
Since
\begin{align*}
\int_{-\infty}^{-2}  \int_{-1}^0 \frac{(\gamma+1)^2}{(y-x)^{1+2s}} \, dy \, dx
	&\leq \frac{(\gamma+1)^2}{2s} \int_{-\infty}^{-2}  \frac{dx}{|1+x|^{2s}}
	\leq \frac{(\gamma+1)^2}{2s(2s-1)} \leq \frac{C}{2s-1},
\end{align*}
we arrive at
\[
h(\R^-,\R^-)
%\int_{-\infty}^0 \int_{-\infty}^0 \frac{|h(x) - h(y)|^2}{|x-y|^{2+2\delta}} \, dy \, dx 
	\leq C\left(\frac{1}{2s-1}+1\right). 
\]

Next, we consider the interactions in~$\R^- \times \R^+$. For this, we use~\eqref{eq:hcomputation4} to estimate
\begin{align*}&
\int_{-\infty}^0 \int_{0}^{+\infty} \frac{|h(x) - \gamma|^2}{|x-y|^{1+2s}} \, dy \, dx\\
&\quad
	= \int_{-\infty}^{-2} \int_{0}^{+\infty} \frac{(1+\gamma)^2}{(y-x)^{1+2s}} \, dy \, dx
		+ \int_{-2}^{-1} \int_{0}^{+\infty} \frac{|h(x) - \gamma|^2}{(x-y)^{1+2s}} \, dy \, dx\\
	&\quad= \frac{(1+\gamma)^2}{2s} \int_{-\infty}^{-2} \frac{dx}{|x|^{2s}} +C\
	=  \frac{(1+\gamma)^2}{2s(2s-1)} 2^{1-2s} +C \\&\quad
	\leq C\left(\frac{1}{2s-1}+1\right). 
\end{align*}
Lastly, for the potential energy, we simply notice that
\[
\int_{-\infty}^0 W(h(x)) \, dx = \int_{-2}^0 W(h(x)) \, dx \leq C.
\]
Combining the last three displays gives~\eqref{eq:h-delta}.

Next, we prove~\eqref{eq:h-R}. Beginning with the interactions in the container~$[-R, 0] \times [-R,0]$, we use~\eqref{eq:hcomputation1}, \eqref{eq:hcomputation2}, and~\eqref{eq:hcomputation3} to find
\begin{equation}\label{r4793tfdscvekjhghtty4t5yu34o}\begin{split}&
\int_{-R}^0\int_{-R}^0 \frac{|h(x) - h(y)|^2}{|x-y|^{1+2s}} \, dy \, dx\\
%	&= \int_{-R}^{-2}\int_{-2}^0 \frac{|1+h(y)|^2}{|x-y|^{2+2\delta}} \, dy \, dx
%		+2 \int_{-2}^0\int_{-2}^0 \frac{|h(x) - h(y)|^2}{|x-y|^{2+2\delta}} \, dy \, dx\\
	&= 2 \int_{-R}^{-2}\int_{-2}^{-1} \frac{|1+h(y)|^2}{|x-y|^{1+2s}} \, dy \, dx
		  + 2\int_{-R}^{-2}\int_{-1}^0 \frac{|1+\gamma|^2}{|x-y|^{1+2s}} \, dy \, dx\\
	&\quad+ \int_{-2}^{-1}\int_{-2}^{-1} \frac{|h(x) - h(y)|^2}{|x-y|^{1+2s}} \, dy \, dx
		+2 \int_{-2}^{-1} \int_{-1}^0 \frac{|h(x) - \gamma|^2}{|x-y|^{1+2s}} \, dy \, dx\\
	&\leq 2\int_{-R}^{-2}\int_{-1}^0 \frac{|1+\gamma|^2}{|x-y|^{1+2s}} \, dy \, dx+C. 
\end{split}\end{equation}
Observe that, exploiting~\eqref{eq:ln-limit} (with~$a:=1$ and~$b:=2$ and with~$a:=R-1$
and~$b:=R$),
\begin{align*}&
\lim_{s \searrow \frac12} \int_{-R}^{-2}  \int_{-1}^0\frac{(\gamma+1)^2}{(y-x)^{1+2s}} \, dy \, dx\\
%	&= \lim_{\delta \to 0^+} \frac{(\gamma+1)^2}{(1+2\delta)} \int_{-R}^{-2} \left[ (-1-x)^{-1-2\delta}- (-x)^{-1-2\delta} \right]  \, dx \\
	&= \lim_{s \searrow \frac12} \frac{(\gamma+1)^2}{2s(2s-1)} 
	 \Big[ (1- 2^{1-2s})-\big( (R-1)^{1-2s}- R^{1-2s}\big) \Big]  \\
	 &=(\gamma+1)^2 \left[ \ln 2  - \ln \left(\frac{R}{R-1}\right)\right] 
	 \leq C.
\end{align*}
This and~\eqref{r4793tfdscvekjhghtty4t5yu34o} give that
\begin{equation}\label{eq:h-inside}
\lim_{s \searrow \frac12}\int_{-R}^{0}  \int_{-R}^{0} \frac{|h(x) - h(y)|^2}{|x-y|^{1+2s}} \, dx \, dy \leq C. 
\end{equation}

Now, we estimate the interactions in~$[-R,0] \times [-R,0]^c$. 
With~\eqref{eq:hcomputation4}, we have that
\[
\int_{-R}^{0}  \int_{0}^{+\infty} \frac{|h(x) - \gamma|^2}{|x-y|^{1+2s}} \, dy \, dx
	\leq \int_{-R}^{-2}\int_{0}^{+\infty} \frac{|1+\gamma|^2}{|x-y|^{1+2s}} \, dy \, dx+ C. 
\]
Using again~\eqref{eq:ln-limit} (with~$a:=2$ and~$b:=R$), we next observe that
\begin{align*}&
\lim_{s \searrow \frac12}\int_{-R}^{-2}  \int_{0}^{+\infty} \frac{|h(x) - h(y)|^2}{|x-y|^{1+2s}} \, dy \, dx 
= \lim_{s \searrow\frac12}\int_{-R}^{-2}  \int_{0}^{+\infty} \frac{(\gamma+1)^2}{(y-x)^{1+2s}} \, dy \, dx\\
%	&= \lim_{\delta \to 0^+} \frac{(\gamma+1)^2}{(1+2\delta)} \int_{-R}^{-2}   (-y)^{-1-2\delta} \, dy \\
	&\qquad\quad= \lim_{s \searrow \frac12} \frac{(\gamma+1)^2}{2s(2s-1)} \big(
	2^{1-2s} - R^{1-2s} \big) =(\gamma+1)^2 \big( \ln R - \ln 2 \big) 
	 \leq C \ln R.
\end{align*}

For the remaining term, we use that~$h \in [-1,\gamma]$ and that~$R \geq 3$ to find that
\begin{align*}&
\int_{-R}^{0}  \int_{-\infty}^{-R} \frac{|h(x) - h(y)|^2}{|x-y|^{1+2s}} \, dy \, dx
= \int_{-2}^{0}  \int_{-\infty}^{-R} \frac{|h(x) - h(y)|^2}{(x-y)^{1+2s}} \, dy \, dx\\
&\qquad\quad\leq  \int_{-2}^{0}  \int_{-\infty}^{-3} \frac{(1+\gamma)^2}{(x-y)^{1+2s}} \, dy \, dx
=\frac{(1+\gamma)^2}{2s} \int_{-2}^{0} \frac{1}{(x+3)^{2s}} \, dx \leq C. 
\end{align*}
Gathering these pieces of information, we see that 
\begin{equation}\label{eq:h-outside}
\lim_{s \searrow\frac12}\int_{[-R,0]}  \int_{[-R,0]^c} \frac{|h(x) - h(y)|^2}{|x-y|^{1+2s}} \, dx \, dy \leq C(1+ \ln R). 
\end{equation}
The inequality in~\eqref{eq:h-R} follows from~\eqref{eq:h-inside}, \eqref{eq:h-outside}, and the simply observing that 
\begin{equation*}
\lim_{s \searrow \frac12} \int_{-R}^0 W(h(x)) \, dx
	=  \int_{-2}^{0} W(h(x)) \, dx \leq C. 
	\qedhere
\end{equation*}
\end{proof}

\end{appendix}

\section*{Acknowledgments}
%%%%%%%%%%%%%%%

It is a pleasure to thank Hidde Sch\"onberger for fruitful discussions. 
SD has been supported by the Australian Future Fellowship
FT230100333 ``New perspectives on nonlocal equations''.
SP has been supported by the NSF Grant DMS-2155156 ``Nonlinear PDE methods in the study of interphases.'' 
EV and MV have been supported by the Australian Laureate Fellowship FL190100081 ``Minimal surfaces, free boundaries and partial differential equations.''

%%\red{From before:}
%%This work has been supported by Alexander von Humboldt Foundation, 
%%NSF grant DMS-1262411 ``Regularity and stability results in variational problems'' and
%%ERC grant 277749 ``EPSILON Elliptic PDE's and Symmetry of Interfaces and Layers for Odd Nonlinearities".

%\bibliography{gamma}
%\bibliographystyle{is-alpha}
%%\bibliographystyle{plain}

%%%%%%%%%%%%%%%%%%%%%%%%%%
\bibliographystyle{imsart-number}
\bibliography{gamma}

\end{document}